\documentclass{amsart}
\numberwithin{equation}{section}

\usepackage{tikz}
\usetikzlibrary{calc}
\usetikzlibrary{cd}
\usepackage{circuitikz}
\usepackage[all,knot]{xy}
\xyoption{all}
\usepackage{amssymb}

\usepackage{epigraph}

\usepackage{hyperref}

\theoremstyle{plain}
\newtheorem{theorem}{Theorem}[subsection]
\newtheorem{lemma}[theorem]{Lemma}
\newtheorem{proposition}[theorem]{Proposition}
\newtheorem{corollary}[theorem]{Corollary}

\theoremstyle{definition}
\newtheorem{definition}[theorem]{Definition}
\newtheorem{example}[theorem]{Example}
\newtheorem{xca}[theorem]{Exercise}
\newtheorem{hypothesis}[theorem]{Hypothesis}

\theoremstyle{remark}
\newtheorem{remark}[theorem]{Remark}

\def\to{\longrightarrow} %I hate short arrows

\DeclareMathOperator{\interior}{\mathrm{int}}
\DeclareMathOperator{\pt}{\mathrm{pt}}
\DeclareMathOperator{\Spec}{\mathrm{Spec}}
\DeclareMathOperator{\Spc}{\mathrm{Spc}}
\DeclareMathOperator{\Proj}{\mathrm{Proj}}
\DeclareMathOperator{\modu}{\mathrm{mod}}
\DeclareMathOperator{\Modu}{\mathrm{Mod}}

\DeclareMathOperator{\Comodu}{\mathrm{Comod}}
\DeclareMathOperator{\Ind}{\mathrm{Ind}}
\DeclareMathOperator{\Thom}{\mathrm{Thom}}
\DeclareMathOperator{\supp}{\mathrm{supp}}
\DeclareMathOperator{\Supp}{\mathrm{Supp}}

\DeclareMathOperator{\Hom}{\mathrm{Hom}}
\DeclareMathOperator{\Ext}{\mathrm{Ext}}
\DeclareMathOperator{\im}{\mathrm{im}}
\DeclareMathOperator{\colim}{\mathrm{colim}}
\DeclareMathOperator{\Idl}{\mathrm{Idl}}
\DeclareMathOperator{\Fun}{\mathrm{Fun}}
\DeclareMathOperator{\Idem}{\mathrm{Idem}}

\def\op{\mathrm{op}}
\def\thick{\mathrm{thick}}
\def\Thick{\mathrm{Thick}}
\def\loc{\mathrm{loc}}
\def\Loc{\mathrm{Loc}}
\def\coloc{\mathrm{coloc}}

\def\unit{\mathbf{1}}

\title{Some notes on tensor triangular geometry}

\author{Greg Stevenson}
\address{Greg Stevenson, Aarhus University, Department of Mathematics, Ny Munkegade 118, bldg. 1530
DK-8000 Aarhus C, Denmark
}
\email{greg@math.au.dk}
\begin{document}
%

%
%\thanks{}

\maketitle

\setcounter{tocdepth}{1}
\tableofcontents

%--------------------------------------------------------------------------------------------------------------------------------------------------

%--------------------------------------------------------------------------------------------------------------------------------------------------

\section{Overview}

%\epigraph{\itshape ``Uh, excuse me. Mark it zero. Next frame.''}{---Walter Sobchak, \textit{The Big Lebowski}}

%\vspace{5pt}

These are the notes from the lectures I gave at the Oberwolfach seminar `Tensor Triangular Geometry and Interactions' which was held in October 2025.

The aim of these notes is to give an introduction to tensor triangular geometry, for both small and large categories, through the lens of lattice theory. We do not try to be exhaustive and this is reflected in both the content and the bibliography. For instance we are quite light on triangulated preliminaries, especially for compactly generated categories.

The approach for essentially small categories is by now quite standard. We begin by giving a short introduction to the theory of lattices, frames, and Stone duality. After some further preliminary material on tensor triangulated categories (tt-categories) we use Stone duality to define the spectrum of an essentially small tt-category. This is developed in parallel to the usual story from commutative algebra. In this way the spectrum is tautologically equipped with a universal property. We expand on this universal property and discuss supports which arise naturally from the universal property and make apparent the connection to the original definitions of Balmer.

We next concentrate on examples. Continuing with our commutative algebra analogy we study the perfect complexes over the polynomial ring $k[x]$. In this example we can compute essentially everything by hand. We then discuss some main classes of examples that motivated the theory. We conclude the chapter on examples with a more or less complete tt-theoretic proof of Thomason's theorem computing the spectrum for the perfect complexes over a quasi-compact and quasi-separated scheme. The argument we give completely avoids the machinery of tensor nilpotence and does not appear elsewhere in the literature to my knowledge (although it has been in private circulation for some years).

In the final chapter we study supports for rigidly compactly generated tensor triangulated categories. The approach here is novel (perhaps experimental would be more accurate) and we again prioritize lattice-theoretic methods. The fundamental object is a `categorified lattice' which is inspired by Clausen and Scholze's notion of categorified locale. The advantage of this approach is that we can treat the small and big supports, both smashing and in the spectrum of the compacts, in a uniform way as well as capturing the universal example and other examples that may emerge. We conclude with a proof of the local-to-global principle, as formulated in our new generality, for supports based on $T_D$ spatial frames with Cantor-Bendixson rank.

These notes are in some sense complementary to the introduction in \cite{Stevensontour}. Different topics are prioritized, we consider different examples, and these notes are both longer and contain more new material which reflects a more evolved understanding of the topic.

\subsection{Acknowledgements}
I am grateful to my co-organizers Nat\`alia Castellana Vila and Henning Krause who not only performed the small miracle of convincing me to write these notes, but were invaluable in their preparation. Thanks are of course also due to the Mathematisches Forschungsinstitut Oberwolfach for their support of the seminar. I am very thankful for useful comments from Scott Balchin, Martin Gallauer, Juan Omar Gom\'ez, Sira Gratz, and John Aslak Wee Kleven.

The writing of these notes was partially supported by the Danmarks Frie Forskningsfond (grant ID: 10.46540/4283-00116B). The first sections were written in sunny Okinawa and I thank the Okinawa Institute of Science and Technology (OIST) through the Theoretical Sciences Visiting Program (TSVP) for very pleasant working conditions.

%----------------------------------------------------------------------------------------------------------------------------------------------
% End of chapter
%----------------------------------------------------------------------------------------------------------------------------------------------

\section{Stone duality and the (dual) spectrum}

The aim of this chapter is to introduce the spectrum, \`a la Balmer \cite{Balmer:2005a}, for an essentially small tt-category. We will first introduce some elements of lattice theory which we will use to construct the spectrum and then develop the spectrum in parallel to the usual notion, from algebraic geometry, for a commutative ring.

%----------------------------------------------------------------------------------------------------------------------------------------------

%----------------------------------------------------------------------------------------------------------------------------------------------

\subsection{Lattice theory}

The basis for our investigation will be lattice theory. There is a deep connection between lattices and topological spaces. From every topological space one can produce a well-behaved lattice namely its collection of open subsets. For many spaces of interest this lattice completely determines the space and so provides an equivalence, known as Stone duality, between certain spaces and lattices of open subsets.

With this in mind we'll introduce the definitions necessary to capture the essence of the collection of open subsets of a space: unions, intersections, infinite distributivity, and an abstract version of a point.

\begin{definition}
A \emph{lattice} is a partially ordered set $(L,\leq)$ such that for every pair of elements $m,n\in L$ there is a least upper bound $m\vee n$ called the \emph{join} of $m$ and $n$ and a greatest lower bound $m\wedge n$ called the \emph{meet} of $m$ and $n$. 

We say that $L$ is \emph{bounded} if it has a greatest element $1$ and a least element $0$. 

We say that $L$ is \emph{complete} if any set $S$ of elements of $L$ admits a supremum $\vee S$ and an infimum $\wedge S$ in $L$, called their join and meet respectively.
\end{definition}

\begin{remark}
A number of remarks are in order:
\begin{enumerate}
\item We simply refer to $L$ as a lattice and suppress the partial order which is understood to be present.
\item Being a lattice is a property of a partially ordered set and not a structure. This is also true for being bounded or complete.
\item The meet, as a binary operation, determines the order on $L$ since $m\leq n$ if and only if $m\wedge n = m$ and analogously for the join.
\item A lattice $L$ is complete if and only if the join exists for any set: we can define the meet of a subset $S$ by
\[
\bigwedge S = \bigvee \{m \in L \mid m\leq s \text{ for all } s\in S\}.
\]
The existence of all meets implies the existence of joins in a similar fashion.
\end{enumerate}
\end{remark}

We will be particularly interested in complete lattices satisfying a strong version of the \emph{distributive law}.

\begin{definition}
A lattice $L$ is \emph{distributive} if for every triple of elements $l,m,n \in L$ we have
\[
l \wedge (m\vee n) = (l\wedge m) \vee (l\wedge n).
\]
That is, the meet distributes over the join.
\end{definition}

\begin{remark}
This is equivalent to the dual formula $l \vee (m\wedge n) = (l\vee m) \wedge (l\vee n)$ for all $l,m,n\in L$.
\end{remark}

\begin{xca}
Prove that we have $l \wedge (m\vee n) = (l\wedge m) \vee (l\wedge n)$ for all triples of elements $l,m,n$ in a lattice $L$ if and only if $l \vee (m\wedge n) = (l\vee m) \wedge (l\vee n)$ for all triples of elements.
\end{xca}

\begin{example}
There are two minimal non-distributive lattices which can be depicted as
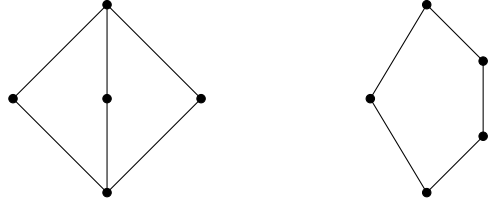
\[
\begin{tikzpicture}
\tikzstyle{every node}=[font=\LARGE]
\draw [short] (8,4.5) -- (8.75,5.25);
\draw [short] (8.75,5.25) -- (8.75,6.25);
\draw [short] (8.75,6.25) -- (8,7);
\node at (8,4.5) [circ] {};
\node at (8.75,5.25) [circ] {};
\node at (8.75,6.25) [circ] {};
\node at (8,7) [circ] {};
\draw [short] (8,4.5) -- (7.25,5.75);
\draw [short] (7.25,5.75) -- (8,7);
\node at (7.25,5.75) [circ] {};
\node at (3.75,4.5) [circ] {};
\node at (5,5.75) [circ] {};
\draw [short] (3.75,4.5) -- (3.75,7);
\draw [short] (3.75,4.5) -- (2.5,5.75);
\draw [short] (3.75,4.5) -- (5,5.75);
\draw [short] (5,5.75) -- (3.75,7);
\draw [short] (2.5,5.75) -- (3.75,7);
\node at (3.75,5.75) [circ] {};
\node at (2.5,5.75) [circ] {};
\node at (3.75,7) [circ] {};
\end{tikzpicture}
\]
where we view the lowest nodes as minimal elements and the ordering increases as we follow edges up the page. Any non-distributive lattice contains one of these as a sublattice.
\end{example}

\begin{example}
Any lattice of subsets, i.e.\ a sublattice of the powerset of some set ordered by inclusion, is distributive since intersections distribute over (arbitrary) unions. For instance, if $A$ is a set then the powerset $\mathcal{P}(A)$ is a complete distributive lattice. If $A$ is infinite then the set of finite subsets $\mathcal{P}^\omega(A)$ is a distributive lattice which is neither complete nor bounded.
\end{example}

\begin{definition}
A complete lattice $F$ is a \emph{frame} if it satisfies the folowing infinite distributive law: for any subset $S\subseteq F$ and an element $f\in F$ 
\[
f \wedge (\bigvee_{s\in S} s) = \bigvee_{s\in S}(f\wedge s).
\]
\end{definition}

\begin{remark}
Since $F$ is complete it is a bounded lattice. We denote by $0$ and $1$ the bottom and top (i.e.\ minimal and maximal) elements of $F$. Taking the set $S$ to have two elements we see that a frame is a distributive lattice.
\end{remark}

\begin{example}
The prototypical example for us is as follows. Let $X$ be a topological space. We can consider the set $\Omega(X)$ of open subsets of $X$ which is partially ordered under inclusion. This is a complete lattice with joins given by unions. The lattice $\Omega(X)$ is a frame. 
\end{example}

\begin{xca}
Describe the meets in $\Omega(X)$ and prove that it is a frame.
\end{xca}

\begin{example}
An important special case is the frame 
\[
\mathbf{2} = \Omega(\{\ast\}) = \{0,1\}
\]
which is the unique frame with exactly two elements. 
\end{example}

\begin{example}
Let $X$ be a topological space. An open subset $A\subseteq X$ is \emph{regular} if $A = \interior\overline{A}$ that is $A$ is the interior of its closure. The poset of regular open subsets of $X$, denoted $\Omega^\mathrm{reg}(X)$, is a frame. The join of a set $\{A_i \mid i\in I\}$ of regular opens is given by
\[
\bigvee_i A_i = \interior\overline{\bigcup_i A_i}.
\]
\end{example}

The frames $\Omega(X)$ and $\Omega^\mathrm{reg}(X)$ exhibit quite different behaviour in general. In order to understand this we need the notion of a morphism of frames and of a point.

\begin{definition}
Let $E$ and $F$ be frames. A morphism $f\colon E\to F$ of frames is a function which preserves arbitrary joins and finite meets. In particular, a map of frames is order preserving and preserves the top and bottom elements. We denote by $\mathsf{Frm}$ the corresponding category.

Similarly, if $L$ and $M$ are complete lattices then a morphism $L\to M$ of complete lattices is a function preserving arbitrary joins and finite meets.
\end{definition}

\begin{definition}
Let $F$ be a frame. A \emph{point} of $F$ is a morphism $p\colon F\to \mathbf{2}$. We denote by $\pt(F)$ the set of points of $F$. This is given the following topology: the open subsets of $\pt(F)$ are given by
\[
U_f = \{p\in \pt(F) \mid p(f)=1 \}
\]
as $f$ varies over the elements of $F$.

This defines a functor $\pt \colon \mathsf{Frm}^\op \to \mathsf{Top}$ from frames to topological spaces. The action of $\pt$ on morphisms is given by composition in the obvious way, i.e.\ for a frame map $\phi\colon F\to F'$ we have
\[
\pt(\phi)(F'\stackrel{p}{\to} \mathbf{2}) = (F\stackrel{\phi}{\to} F' \stackrel{p}{\to} \mathbf{2}).
\]
\end{definition}

\begin{remark}
The intuition for this definition is that a map of topological spaces $f\colon X\to Y$ induces a corresponding frame map $\Omega(Y) \to \Omega(X)$ by sending an open $U$ of $Y$ to $f^{-1}(U)$. Thus a point of $X$, i.e.\ a morphism $\ast \to X$ from $\ast$ the one point space, gives rise to a point $\Omega(X) \to \Omega(\ast) \cong \mathbf{2}$ of the frame $\Omega(X)$.
\end{remark}

\begin{remark}
The category $\mathsf{Frm}^\op$ is called the category of \emph{locales}. So a frame is the same sort of object as a locale, but the morphisms differ (they are reversed). Using locales reduces the appearance of ops. However, we will encounter a lot of order reversals, and cannot escape all of it, so it is better for us to start getting used to it.
\end{remark}

\begin{remark}
The space $\pt(F)$ is also sometimes called the \emph{spectrum} of $F$ and denoted by $\Spec(F)$.
\end{remark}

\begin{xca}
Let $F$ be a frame. Verify that the subsets 
\[
U_f = \{p\in \pt(F) \mid p(f)=1 \}
\]
are the open subsets for a topology on $\pt(F)$.
\end{xca}

\begin{definition}\label{Gdef:latticeideal}
Let $L$ be a bounded distributive lattice. An \emph{ideal} $I$ of $L$ is a downward closed and upward directed subset of $L$, i.e.\ $I$ is closed under meets with arbitrary elements of $L$ and is closed under finite joins. Dually a subset $F$ of $L$ is called a \emph{filter} if $F$ is upward closed and downward directed, i.e.\ closed under joins with arbitrary elements of $L$ and closed under finite meets.

An ideal $P$ of $L$ is \emph{prime} if $P$ is a proper subset and whenever $m\wedge n \in P$ then one of $m$ or $n$ lies in $P$. The notion of prime filter is dual.
\end{definition}

Giving a morphism $p\colon F\to \mathbf{2}$ is equivalent to specifying $p^{-1}(0)$, which is a prime ideal of $F$, or to specifying $p^{-1}(1)$, which is a prime filter of $F$. 

\begin{definition}\label{Gdef:prime}
An element $\mathfrak{p}$ of a frame $F$ is \emph{meet-irreducible} or \emph{(meet-)prime} if whenever $m\wedge n\leq \mathfrak{p}$, for $m,n\in F$, we have $m\leq \mathfrak{p}$ or $n\leq \mathfrak{p}$ and $\mathfrak{p}\neq 1$.
\end{definition}

\begin{xca}\label{Gxca:indcompletion}
Let $L$ be a bounded distributive lattice. Show that the poset $\Idl(L)$ of ideals of $L$ ordered by inclusion is a frame. Then prove that an ideal $P$ of $L$ is prime if and only if it is meet-irreducible in $\Idl(L)$.
\end{xca}

\begin{lemma}\label{Glem:prime}
Let $F$ be a frame. The assignments sending a point $p$ of $F$ to $\vee p^{-1}(0)$ and a prime element $\mathfrak{p}$ to the point defined by sending the elements of the prime ideal $\{f\in F\mid f\leq \mathfrak{p}\}$ to $0$ are inverse bijections between $\pt(F)$ and the set of prime elements of $F$.
\end{lemma}
\begin{proof}
Exercise.
\end{proof}

\begin{example}
Consider the case that $X = \mathbb{R}$ is the real line. If we consider $\Omega(\mathbb{R})$ then we get a point $p_a$ for every element $a\in \mathbb{R}$ by
\[
p_a(U) = \begin{cases} 1 &\text{ if } a\in U \\ 0 &\text{ if } a\notin U \end{cases}
\]
In fact, these are precisely the points of $\Omega(\mathbb{R})$.
\end{example}

The functors $\pt$ and $\Omega$ are intimately connected.

\begin{theorem}\label{thm:stoneadjunction}
There is an adjunction
\[
\begin{tikzcd}
\mathsf{Frm}^\op \arrow[rr, shift right=0.75ex, "\pt"'] &  & \mathsf{Top} \arrow[ll, shift right=0.75ex, "\Omega"']
\end{tikzcd}
%\xymatrix{
%\mathsf{Frm}^\op \ar[rr]<-0.75ex>_-{\pt} \ar@{<-}[rr]<0.75ex>^-{\Omega} && \mathsf{Top}
%}
\]
where $\pt$ is right adjoint to $\Omega$.
\end{theorem}
\begin{proof}
We do not give a full proof, but just note that the unit and counit are given as follows. For a frame $F$ the counit is the map of frames $F\to \Omega\pt(F)$ which sends $f\mapsto U_f$, and for a space $X$ the unit is the map $X\to \pt\Omega(X)$ which sends $x$ to the point $p_x$ defined by
\[
p_x(U) = \begin{cases} 1 &\text{ if } x\in U \\ 0 &\text{ if } x\notin U \end{cases}
\]
\end{proof}

This result is the precursor to \emph{Stone duality} which identifies full subcategories of frames and spaces on which the above adjunction restricts to an equivalence.

\begin{definition}
A frame $F$ is \emph{spatial} if there is a space $X$ such that $F\cong \Omega(X)$. We denote by $\mathsf{SFrm}$ the full subcategory of $\mathsf{Frm}$ consisting of spatial frames (this is just the essential image of $\Omega$).
\end{definition}

\begin{remark}\label{Grem:spatial}
There is a more intrinsic definition: $F$ is spatial precisely when given elements $m,n\in F$ with $m \nleq n$ there is a point $p$ of $F$ witnessing this, i.e.\ $p(m)=1$ and $p(n)=0$. 
\end{remark}

\begin{example}
The frame $\Omega(\mathbb{R})$ is spatial and its space of points is precisely $\mathbb{R}$. This is a special case of Theorem~\ref{thm:stoneduality}. On the other hand, $\Omega^\mathrm{reg}(\mathbb{R})$ the frame of regular open subsets of $\mathbb{R}$ is not spatial. %Regular opens are a Boolean algebra and so its enough to show it is atomless. Since $\mathbb{R}$ is T_2 an atom would be an isolated point and hence there are none.
\end{example}

Now let us introduce the corresponding class of spaces. In order to align with our definition of spatial frames we start with the definition in terms of our adjunction and then give the usual definition in Remark~\ref{rem:soberdefn}

\begin{definition}
A topological space $X$ is sober if there is a frame $F$ such that $X\cong \pt(F)$. We denote by $\mathsf{Sob}$ the full subcategory of $\mathsf{Top}$ consisting of sober spaces (this is just the essential image of $\pt$).
\end{definition}

\begin{remark}\label{rem:soberdefn}
An equivalent, and perhaps more standard, definition is that $X$ is sober if every non-empty irreducible closed subset of $X$ is the closure of a unique point (called the generic point of the irreducible closed subset). This implies that $X$ is a $T_0$ space. It is a good exercise to check that $\pt(F)$ satisfies this condition for any frame $F$.
\end{remark}

We are now in a position to give the statement of Stone duality.

\begin{theorem}\label{thm:stoneduality}
The adjunction of Theorem~\ref{thm:stoneadjunction} restricts to an equivalence between $\mathsf{SFrm}^\op$ and $\mathsf{Sob}$.
\end{theorem}

We can actually restrict this equivalence even further to rather important special classes of spaces and lattices.

\begin{definition}
Let $C$ be a complete lattice. We say an element $c\in C$ is \emph{finitely presented} (also sometimes called \emph{compact} or \emph{quasi-compact}) if whenever $c\leq \vee S$ for some set of elements $S\subseteq C$ there is a finite subset $S'\subseteq S$ such that $c\leq \vee S'$. We denote the collection of finitely presented elements of $C$ by $C^\omega$. 

The lattice $C$ is said to be \emph{algebraic} if it is generated by $C^\omega$, i.e.\ every element of $C$ is a join of finitely presented elements.
\end{definition}

\begin{definition}
A frame $F$ is a \emph{coherent frame} if:
\begin{enumerate}
\item $F$ is algebraic;
\item the finitely presented elements $F^\omega$ form a bounded sublattice of $F$.
\end{enumerate}
We note that the finitely presented elements are always closed under finite joins and so this second condition is really asking that $F^\omega$ is closed under binary meets and that the top element $1$ is finitely presented. 

A morphism of coherent frames is a frame map which sends finitely presented elements to finitely presented elements. This gives a subcategory $\mathsf{CohFrm}$ (not full) of $\mathsf{Frm}$.
\end{definition}

\begin{remark}
All of these lattice-theoretic definitions are very transparent when viewed through the lens of category theory. A poset is simply a (skeletal) category such that there is at most one morphism between any two objects (one calls such a category thin). The top and bottom are then terminal and initial objects and a bounded lattice is such a category which admits all finite limits and colimits. A frame is then a thin category with limits and colimits such that finite limits commute with colimits. Being a coherent frame is just asking that this category is locally finitely presented and that the finitely presented objects are closed under finite limits.
\end{remark}

It turns out that coherent frames are automatically spatial.

\begin{proposition}
If $F$ is a coherent frame then $F$ is spatial.
\end{proposition}
\begin{proof}
We give a sketch of how to check the condition of Remark~\ref{Grem:spatial}. Suppose that $x\nleq y$ are elements of $F$. Both $x$ and $y$ are joins of finitely presented elements and so there exist finitely presented elements $a\leq x$ and $b\leq y$ such that $a\nleq b$. It is enough to find a point separating $a$ and $b$, for instance it would be enough to find a maximal ideal of $F$ containing $b$ but not $a$. In this case the corresponding point is given by identifying the elements of the maximal ideal with $0$ (cf.\ Lemma~\ref{Glem:prime}).

One can find such an ideal using Zorn's lemma. We can get started because the downward closure of $b$, i.e.\ the principal ideal it generates, does not contain $a$ by assumption. For further details see \cite[Theorem~II.3.4]{Johnstone}
\end{proof}

Thus under Stone duality the subcategory $\mathsf{CohFrm}$ is equivalent to a subcategory of spaces. These spaces are called \emph{coherent spaces} or \emph{spectral spaces} and the corresponding morphisms are quasi-compact continuous maps. Such spaces can be characterized as follows.

\begin{lemma}
A space $X$ is coherent if and only if the following conditions are satisfied:
\begin{enumerate}
\item $X$ is sober, i.e.\ $T_0$ and every non-empty irreducible closed subset has a generic point;
\item $X$ is quasi-compact;
\item $X$ has a basis of quasi-compact open subsets;
\item the intersection of two quasi-compact open subsets of $X$ is quasi-compact.
\end{enumerate}
\end{lemma}
\begin{proof}
Exercise.
\end{proof}

\begin{remark}
By a result of Hochster \cite{Hochster1969} a space $X$ is coherent precisely if it is the spectrum of some commutative ring\textemdash{}hence the alternative name spectral. 
\end{remark}

\begin{remark}\label{Grem:coherentequivalence}
Given a bounded distributive lattice $L$ we can consider its lattice of ideals $\Idl(L)$ which is a frame. In fact, one can check that $\Idl(L)$ is a coherent frame with finitely presented elements the principal ideals, i.e.\ $\Idl(L)^\omega \cong L$. This defines an equivalence of categories between bounded distributive lattices and coherent frames. It follows that coherent spaces are dual to the category of bounded distributive lattices.
\end{remark}
%----------------------------------------------------------------------------------------------------------------------------------------------

%----------------------------------------------------------------------------------------------------------------------------------------------

\subsection{Preliminaries on tt-categories}

We now introduce some notation and terminology and record some useful facts about triangulated and `tensor triangulated' categories. 

We will generally denote by $\mathsf{K}$ an essentially small triangulated category. We will use $\Sigma$ to denote the suspension functor of $\mathsf{K}$. 

\begin{definition}
An \emph{essentially small tensor triangulated category}, usually abbreviated to \emph{tt-category}, is a triple $(\mathsf{K}, \otimes, \mathbf{1})$, where $\mathsf{K}$ is an essentially small triangulated category and $(\otimes, \unit)$ is a symmetric monoidal structure on $\mathsf{K}$. We furthermore insist that these are compatible in the sense that $\otimes$ is an exact functor in each variable.
\end{definition}

\begin{remark}
We will systematically omit associators and so on from the discussion. One could also insist on Koszul signs governing the interaction of the tensor product and the suspension (a natural assumption in the sense that they occur for any example arising as the homotopy category of a symmetric monoidal stable $\infty$-category). For the purposes of our discussion we can safely ignore all of this. In fact, the elementary aspects of the theory are only really concerned with when $x\otimes y$ vanishes for $x,y\in \mathsf{K}$. In particular, they are not too sensitive to the behaviour of the tensor product on maps and mostly depend on $x\otimes y$ and $y\otimes x$ vanishing simultaneously.
\end{remark}

\begin{remark}
We will assume that $\unit \neq 0$ in order to avoid having to insist on this repeatedly to avoid corner cases of statements.
\end{remark}

\begin{example}\label{Gex:ttcats}
There are many examples, of many different flavours. Here are a few typical ones.
\begin{enumerate}
\item\label{Gitem:perf} Let $R$ be a commutative ring. We can consider 
\[
\mathsf{D}^\mathrm{perf}(R) \cong \mathsf{K}^\mathrm{b}(\mathrm{proj} R)
\]
the category of perfect complexes over $R$, i.e.\ bounded complexes of finitely generated projective $R$-modules up to homotopy. This is symmetric monoidal via the tensor product over $R$ with unit object $R$ concentrated in degree $0$.
\item There is a more general version of the first example: for $X$ a quasi-compact and quasi-separated scheme we can consider $\mathsf{D}^\mathrm{perf}(X)$ the perfect complexes over $X$ with the left derived tensor product and unit $\mathcal{O}_X$ in degree $0$. 
\item\label{Gitem:hopf} If $A$ is a finite dimensional Hopf algebra over a field $k$ then consider $\modu A$ the category of finite dimensional right $A$-modules. The algebra $A$ is automatically self-injective and so $\modu A$ is a Frobenius category. The Hopf algebra structure on $A$ makes $\modu A$ into a closed monoidal category with exact tensor product. If $A$ is cocommutative then $\modu A$ is symmetric monoidal and its stable category $\underline{\modu} A$ is a tt-category via the tensor product over $k$ with unit the trivial representation.
\item The homotopy category of finite spectra $\mathrm{ho}(\mathsf{Sp}^\omega)$ is a tt-category via the smash product with unit object the sphere spectrum $S$.
\item One can generalize (\ref{Gitem:perf}) in a different direction. Suppose that $\mathsf{M}$ is an essentially small symmetric monoidal additive category. Then $\mathsf{K}^\mathrm{b}(\mathsf{M})$ the bounded homotopy category of $\mathsf{M}$ is a tt-category. There are many interesting examples, for instance the recent work of Balmer and Gallauer on permutation modules \cite{balmer2023geometry} is concerned with categories of this form.
\end{enumerate}
\end{example}

\begin{remark}
Example~\ref{Gex:ttcats} (\ref{Gitem:hopf}) gives a natural source of closed monoidal triangulated categories which are not symmetric. If the Hopf algebra is quasi-triangular they are braided and a lot of what we describe below can be implemented. In the honestly noncommutative case one needs a different approach and different expectations. The noncommutative case has been studied extensively by Nakano, Yakimov, and Vashaw starting with \cite{NVY2022}. There has also been some work in the case where there is no monoidal structure at all. We refer to \cite{Gratz2023} and the references therein.
\end{remark}

\begin{definition}\label{Gdef:thick}
A full subcategory $\mathsf{J}$ of $\mathsf{K}$ is \emph{thick} if it is closed under suspensions, cones, and direct summands. In other words, $\mathsf{J}$ is thick if it is a triangulated subcategory and if $k\oplus k'\in \mathsf{J}$ then both $k$ and $k'$ lie in $\mathsf{J}$. Given a set of objects $\mathcal{A}\subseteq \mathsf{K}$ we write $\thick(\mathcal{A})$ for the smallest thick subcategory containing $\mathcal{A}$ (also called the thick subcategory generated by $\mathcal{A}$.)

A thick subcategory $\mathsf{J}$ of $\mathsf{K}$ is a thick \emph{tensor-ideal} (or $\otimes$-ideal) if given any $k\in \mathsf{K}$ and $l\in \mathsf{J}$ the tensor product $k\otimes l$ lies in $\mathsf{J}$, that is
\begin{displaymath}
\mathsf{K} \times \mathsf{J} \stackrel{\otimes}{\longrightarrow} \mathsf{K}
\end{displaymath}
factors through $\mathsf{J}$. We will usually omit the word thick and just speak of tensor ideals as we will only consider ideals that are closed under summands. Given a set of objects $\mathcal{A}\subseteq \mathsf{K}$ we write $\thick^\otimes(\mathcal{A})$ for the smallest thick tensor-ideal containing $\mathcal{A}$ (also called the thick tensor-ideal generated by $\mathcal{A}$.)

Finally, a proper thick tensor-ideal $\mathsf{P}$ of $\mathsf{K}$ is \emph{prime} if whenever $k\otimes l \in \mathsf{P}$, for $k,l\in \mathsf{K}$, then either $k$ or $l$ lies in $\mathsf{P}$. When no confusion can occur we will often just speak of prime ideals or primes of $\mathsf{K}$. 
\end{definition}

\begin{xca}
Show that
\[
\thick^\otimes(\mathcal{A}) = \thick(a\otimes k \mid a\in \mathcal{A}, k\in \mathsf{K}).
\]
\end{xca}

Since we have assumed that $\mathsf{K}$ is essentially small the collections of thick subcategories, thick tensor-ideals, and prime tensor-ideals each form a set. Each of these sets is naturally ordered by inclusion giving rise to posets $\Thick(\mathsf{K})$, $\Thick^\otimes(\mathsf{K})$, and $\Spc \mathsf{K}$.

\begin{lemma}\label{Glem:small_unit_gen}
Suppose that $\mathsf{K}$ is generated by the tensor unit, i.e.\ $\thick(\unit) = \mathsf{K}$. Then every thick subcategory is a tensor-ideal:
\begin{displaymath}
\Thick(\mathsf{K}) = \Thick^\otimes(\mathsf{K}).
\end{displaymath}
\end{lemma}
\begin{proof}
Exercise.
\end{proof}

\begin{lemma}\label{Glem:fpideal}
The partially ordered sets $\Thick(\mathsf{K})$ and $\Thick^\otimes(\mathsf{K})$ are complete algebraic lattices. The finitely presented thick subcategories (resp.\ tensor-ideals) are precisely the ones generated by a single object as thick subcategories (resp.\ tensor-ideals).
\end{lemma}
\begin{proof}
We just give the proof for $\Thick^\otimes(\mathsf{K})$ as the arguments are essentially identical. For completeness, it is enough to show that arbitrary meets exist. The intersection of any collection of thick tensor-ideals is again a thick tensor-ideal and so this gives the meet (the empty intersection being all of $\mathsf{K}$). It follows, or one can check directly, that the join of a family of tensor-ideals $\{\mathsf{J}_\lambda\mid \lambda \in \Lambda\}$ is the tensor-ideal $\thick^\otimes(\cup_\lambda \mathsf{J}_\lambda)$ generated by their union. In particular, for a finite family of objects $k_1,\ldots, k_n$ we have
\begin{align*}
\thick^\otimes(k_1) \vee \cdots \vee \thick^\otimes(k_n) &= \thick^\otimes(\thick^\otimes(k_1) \cup \cdots \cup \thick^\otimes(k_n)) \\
&=  \thick^\otimes(k_1,\ldots, k_n) \\
&= \thick^\otimes(k_1 \oplus \cdots \oplus k_n).
\end{align*}

 Any tensor-ideal $\mathsf{I}$ is the join of its principal ideals:
\[
\mathsf{I} = \bigvee_{k\in \mathsf{I}} \thick^\otimes(k).
\]
The operations defining a thick tensor-ideal are finitary and so if an object is in the thick tensor-ideal generated by some set of objects $\mathcal{A}$ it is already in the ideal generated by some finite subset of $\mathcal{A}$. Thus a thick tensor-ideal $\mathsf{I}$ is finitely presented if and only if it is principal, i.e.\ generated by a single object. We already noted that every ideal is a join of principal ideals and so we are done.
\end{proof}

\begin{remark}
As in the proof we call a thick tensor-ideal generated by a single object principal. They are also sometimes referred to as finitely generated ideals. In fact, an ideal is finitely generated if and only if it is principal.
\end{remark}

\begin{xca}\label{Gxca:sublattice}
Show that $\Thick^\otimes(\mathsf{K})$ is a bounded sublattice of $\Thick(\mathsf{K})$.
\end{xca}

\begin{example}
The complete lattice $\Thick(\mathsf{K})$ need not be distributive and one should view distributivity as rather special behaviour. As concrete examples, if one takes $\mathsf{K} = \mathsf{D}^\mathrm{b}(\mathrm{coh}\;\mathbb{P}^1)$ the bounded derived category of coherent sheaves on $\mathbb{P}^1$ or $\mathsf{K} = \mathsf{D}^\mathrm{b}(\modu kA_2)$ the bounded derived category of representations of the quiver $A_2 = (1 \longrightarrow 2)$ over the field $k$ (aka the abelian category of maps of vector spaces) then the lattice will not be distributive. 

In fact $\modu kA_2$ is rather simple. There are two indecomposable projective modules $P_1$ and $P_2$ corresponding to the vertices, where $P_1$ is simple, and another simple module $S_2$. Every representation is a direct sum of copies of these three and every complex is quasi-isomorphic to a sum of suspensions of these three objects. The lattice of thick subcategories is
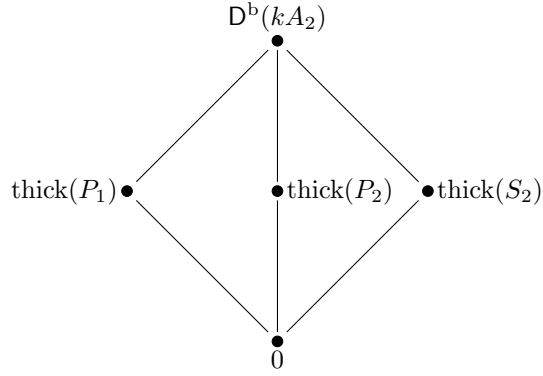
\[
\begin{tikzpicture}%[shorten >= 2pt, shorten <= 2pt]
    
		\node (v0) at (0,0) {};
    \node (va) at (-2,2) {};
    \node (vb) at (0,2) {};
		\node (vc) at (2,2) {};
		\node (v1) at (0,4) {};
    
    \draw[fill] (v0)  circle (2pt) node [below] {0};
		\draw[fill] (va)  circle (2pt) node [left] {$\thick(P_1)$};
		\draw[fill] (vb)  circle (2pt) node [right] {$\thick(P_2)$};
		\draw[fill] (vc)  circle (2pt) node [right] {$\thick(S_2)$};
		\draw[fill] (v1)  circle (2pt) node [above] {$\mathsf{D}^\mathrm{b}(kA_2)$};
    %\draw[fill] (v2)  circle (2pt) node [above] {\footnotesize 2};

		\path[-] (v0) edge  node [above] {} (va);
		\path[-] (v0) edge  node [above] {} (vb);
		\path[-] (v0) edge  node [above] {} (vc);
		\path[-] (va) edge  node [above] {} (v1);
		\path[-] (vb) edge  node [above] {} (v1);
		\path[-] (vc) edge  node [above] {} (v1);

  \end{tikzpicture}
\]
\end{example}

Similarly, the lattice $\Thick^\otimes(\mathsf{K})$ does not need to be distributive in general. This should be compared to the fact that for a commutative ring $R$ the ideals of $R$ (in the usual sense) form a complete algebraic lattice but need not be distributive! In the next section we discuss how to address this issue and make the comparison to algebraic geometry.

%\textcolor{red}{
%\begin{example}
%Is there a reasonable example that I could explain where the lattice is not distributive.
%\end{example}
%}

%----------------------------------------------------------------------------------------------------------------------------------------------

%----------------------------------------------------------------------------------------------------------------------------------------------

\subsection{An extended meta-example: spectra and spectra}

%----------------------------------------------------------------------------------------------------------------------------------------------

%----------------------------------------------------------------------------------------------------------------------------------------------

We claimed that $\Thick^\otimes(\mathsf{K})$ is not necessarily distributive and so it need not correspond to a sober space under Stone duality. In the commutative algebra context this is solved by passing to \emph{radical} ideals so that one can take advantage of the coincidence of intersection and multiplication. We can do exactly the same thing in the tt context. In this section we discuss the commutative and tt cases in parallel leading up to the construction of the spectrum.

In this section we fix a tt-category $\mathsf{K}$ and a commutative ring $R$. We will denote by $\mathrm{Idl}(R)$ the complete lattice of ideals of $R$ ordered by inclusion. This lattice is generally not distributive: the meet, which is given by intersection, need not distribute over the sum of ideals. However, given ideals $I$ and $J$ of $R$ one can form the product ideal $IJ \subseteq I\cap J$. Because multiplication distributes over addition the operation of ideal multiplication is distributive.

\begin{definition}
An ideal $J$ of $R$ is \emph{radical} if $r^n \in J$ implies that $r\in J$. For an ideal $J$ we denote by $\sqrt{J}$ its radical and we let $\mathrm{Rad}(R)$ be the complete lattice of radical ideals.
\end{definition}

The key point is the following.

\begin{lemma}\label{lem:radring}
Let $I$ and $J$ be radical ideals of $R$. Then $I\cap J = \sqrt{IJ}$ and so the meet is given by `multiplication of radical ideals'.
\end{lemma}
\begin{proof}
Since $I$ and $J$ are radical so is $I\cap J$. We always have $IJ \subseteq I\cap J$ and hence $\sqrt{IJ}\subseteq I\cap J$. On the other hand if $r\in I\cap J$ then $r^2 \in IJ$ and so $r$ lies in $\sqrt{IJ}$. This shows $I\cap J \subseteq \sqrt{IJ}$ and so we get equality as desired.
\end{proof}

Using this we can produce the space underlying the affine scheme corresponding to $R$ and show it is coherent.

\begin{theorem}
The lattice $\mathrm{Rad}(R)$ is a coherent frame. The corresponding space under Stone duality is $\Spec R$ the prime ideal spectrum of $R$ with the Zariski topology.
\end{theorem}
\begin{proof}
We will only sketch the argument. It is straightforward, using that we only have access to finite sums, to check that principal radical ideals (i.e.\ the radicals of principal ideals) are finitely presented elements of $\mathrm{Rad}(R)$. Every radical ideal is the join of the principal radical ideals it contains. Hence $\mathrm{Rad}(R)$ is an algebraic lattice. The ring $R$ is the top element of $\mathrm{Rad}(R)$ and is principal, generated by $1$, so the top element of $\mathrm{Rad}(R)$ is finitely presented. Finitely generated radical ideals are finite joins of principal ones and so are finitely presented in $\mathrm{Rad}(R)$. Using that every radical ideal is a join of principal ones it follows easily that any finitely presented element of $\mathrm{Rad}(R)$ must be the radical of a finitely generated ideal.

Let us next show that $\mathrm{Rad}(R)$ is distributive. Because it is an algebraic lattice this already implies it is a frame. If $I,J$ and $K$ are radical ideals then using Lemma~\ref{lem:radring} we see
\begin{align*}
I\cap (J\vee K) &= I\cap \sqrt{J+K} \\
&= \sqrt{I\sqrt{J+K}} \\
&= \sqrt{I(J+K)} \\
&= \sqrt{(IJ +IK)} \\
&= \sqrt{\sqrt{IJ} + \sqrt{IK}} \\
&= \sqrt{(I\cap J) + (I\cap K)} \\
&= (I\cap J) \vee (I\cap K).
\end{align*}

It also follows from Lemma~\ref{lem:radring} that the finitely presented elements are closed under meets: for elements $r_i$ and $s_j$ of $R$ with $1\leq i \leq m$ and $1\leq j \leq n$ we have
\[
\sqrt{(r_1,\ldots,r_m)} \cap \sqrt{(s_1,\ldots,s_n)} = \sqrt{(r_1,\ldots,r_m)(s_1,\ldots,s_n)} = \sqrt{(r_is_j \mid i,j)}.
\]

By Lemma~\ref{Glem:prime} we can describe the points of $\pt(\mathrm{Rad}(R))$ as the prime elements of the frame $\mathrm{Rad}(R)$. These are exactly the prime ideals and it is an exercise to show that the topology is correct.
\end{proof}

The remarkable thing is we can follow exactly the same recipe in the tt setting\textemdash{}the proofs are more or less identical.

\begin{definition}
A thick tensor-ideal $\mathsf{J}$ of $\mathsf{K}$ is \emph{radical} if whenever $x^{\otimes n} \in \mathsf{J}$ for some $x\in \mathsf{K}$ we already have $x\in \mathsf{J}$. For a tensor-ideal $\mathsf{J}$ we denote its radical, i.e.\ the smallest radical ideal containing it, by $\sqrt{\mathsf{J}}$.

We use $\Thick^{\sqrt{\otimes}}(\mathsf{K})$ to denote the complete lattice of radical tensor-ideals of $\mathsf{K}$.
\end{definition}

It turns out that, although a tensor-ideal is a more complicated object that an ideal of a ring, one can still take the radical using the naive formula.

\begin{lemma}\label{Glem:radical}
Let $\mathsf{J}$ be a tensor-ideal of $\mathsf{K}$. Then 
\[
\sqrt{\mathsf{J}} = \{x\in \mathsf{K} \mid x^{\otimes n}\in \mathsf{J} \text{ for some } n\geq 1\}.
\]
\end{lemma}
\begin{proof}
See \cite[Lemma~4.2]{Balmer:2005a} or \cite[Lemma~3.1.5]{KockPitsch}.
\end{proof}

\begin{xca}
Prove that an ideal $\mathsf{P}$ is prime in the sense of Definition~\ref{Gdef:thick} if and only if it is prime in the lattice of radical tensor-ideals of $\mathsf{K}$ as in Definition~\ref{Gdef:prime}.
\end{xca}

One then has an analogue of the key lemma from above. For a pair of tensor-ideals $\mathsf{I}$ and $\mathsf{J}$ we let $\mathsf{I}\otimes \mathsf{J}$ denote the thick tensor-ideal generated by all objects of the form $x\otimes y$ with $x\in \mathsf{I}$ and $y\in \mathsf{J}$.

\begin{lemma}\label{Glem:meettensor}
Let $\mathsf{I}$ and $\mathsf{J}$ be radical tensor-ideals of $\mathsf{K}$. Then $\mathsf{I}\cap \mathsf{J} = \sqrt{\mathsf{I} \otimes \mathsf{J}}$ and so the meet is given by `tensor product of radical ideals'.
\end{lemma}
\begin{proof}
The intersection $\mathsf{I}\cap \mathsf{J}$ is radical and contains $\mathsf{I} \otimes \mathsf{J}$ and hence also $\sqrt{\mathsf{I} \otimes \mathsf{J}}$. On the other hand, if $x\in \mathsf{I}\cap \mathsf{J}$ then $x\otimes x \in \mathsf{I} \otimes \mathsf{J}$ and so $x\in \sqrt{\mathsf{I} \otimes \mathsf{J}}$ as required.
\end{proof}

\begin{theorem}
The lattice $\Thick^{\sqrt{\otimes}}(\mathsf{K})$ is a coherent frame. 
\end{theorem}
\begin{proof}
One can follow the same outline as the proof for radical ideals of a commutative ring. The key points are the previous lemma and that witnessing membership of a single object in a radical thick tensor-ideal only involves finitely many other objects (cf.\ Lemma~\ref{Glem:fpideal}). This fact appears as Theorem~3.1.9 of \cite{KockPitsch} where full details are given.
\end{proof}

\begin{remark}
One should also see \cite{BKS-support} where it was already observed that $\Spc \mathsf{K}$ is a coherent space.
\end{remark}

So we know that $\Thick^{\sqrt{\otimes}}(\mathsf{K})$ is Stone dual to some coherent space whose points are the prime elements of $\Thick^{\sqrt{\otimes}}(\mathsf{K})$, i.e.\ the points are exactly the prime thick tensor-ideals (see Definition~\ref{Gdef:thick}). So we have a topology on the set $\Spc \mathsf{K}$ and its open subsets recover the lattice of radical tensor-ideals of $\mathsf{K}$. We denote this space by $(\Spc \mathsf{K})^\vee$.

The reason for the decoration is that, due to the order reversal that occurs when passing from a ring to its derived category, the space $(\Spc \mathsf{K})^\vee$ will not recover $\Spec R$ when applied to $\mathsf{K} = \mathsf{D}^\mathrm{perf}(R)$. This will be illustrated in detail in Section~\ref{Gsec:realexample} when we do a concrete example. The space we obtained is actually \emph{dual} to the original definition of the spectrum.

\begin{definition}
Let $X$ be a coherent space. The \emph{Hochster dual} of $X$, denoted by $X^\vee$, is the space with underlying set $X$ and topology given by taking the closed subsets with quasi-compact complement as a basis of opens. Thus a general open subset $V$ of $X^\vee$ has the form
\[
V = \cup_\lambda V_\lambda
\]
where each $V_\lambda$ is closed and the open subset $X\setminus V_\lambda$ of $X$ is quasi-compact. We call a subset $V$ of this form a \emph{Thomason subset}. We denote by $\mathrm{Thom}(X)$ the coherent frame of Thomason subsets of $X$ (which is nothing but $\Omega(X^\vee)$). 
\end{definition}

\begin{xca}
Check that $X^\vee$ is again a coherent space and that $X^{\vee \vee} = X$.
\end{xca}

\begin{remark}
Under the equivalence of Remark~\ref{Grem:coherentequivalence} Hochster duality just corresponds to taking the opposite bounded distributive lattice. There is, of course, a corresponding duality on coherent frames. This is given by sending a coherent frame $F$ to $\Ind ((F^\omega)^\op)$ the Ind-completion of the opposite lattice of finitely presented elements. Here the Ind-completion can be modeled by the lattice of ideals of $(F^\omega)^\op$ (cf.\ Exercise~\ref{Gxca:indcompletion}). 
\end{remark}

\begin{definition}
Let $\mathsf{K}$ be a tt-category. The \emph{spectrum} of $\mathsf{K}$ is defined to be
\[
\Spc \mathsf{K} = \pt(\Thick^{\sqrt{\otimes}}(\mathsf{K}))^\vee
\]
the Hochster dual of the Stone dual of the coherent frame of radical tensor-ideals.
\end{definition}

We thus have, by construction, an isomorphism of lattices between the Thomason subsets of $\Spc \mathsf{K}$ and $\Thick^{\sqrt{\otimes}}(\mathsf{K})$. This space has been constructed in a `free' way without losing any information about $\Thick^{\sqrt{\otimes}}(\mathsf{K})$. It has a corresponding universal property and we describe this next.

Given a monoidal exact functor between tt-categories $F\colon \mathsf{K} \to \mathsf{L}$ and a radical tensor-ideal $\mathsf{J}$ of $\mathsf{K}$ we can consider the  radical tensor-ideal of $\mathsf{L}$ defined by 
\[
F_*(\mathsf{J}) = \sqrt{\thick^\otimes(F(a) \mid a\in \mathsf{J})}
\]
i.e.\ the radical ideal generated by the image of $\mathsf{J}$. This defines a map of coherent frames
\[
\Thick^{\sqrt{\otimes}}(\mathsf{K}) \to \Thick^{\sqrt{\otimes}}(\mathsf{L})
\]
and hence a corresponding map $f\colon \Spc \mathsf{L} \to \Spc \mathsf{K}$. This makes $\Spc$ into a functor on the category of tt-categories with exact monoidal functors. Concretely, $f$ sends a prime ideal $\mathsf{Q}$ of $\mathsf{L}$ to its preimage $F^{-1}\mathsf{Q}$.

%\begin{definition}
%Constructible topology
%\end{definition}

%----------------------------------------------------------------------------------------------------------------------------------------------

%----------------------------------------------------------------------------------------------------------------------------------------------

\subsection{Support and the universal property unwound}

In this section we make explicit the connection between the universal property of $\Spc \mathsf{K}$ as defined by Stone duality and the universal property given by Balmer when he introduced the spectrum \cite{Balmer:2005a}. We also use our lattice theoretic description to make some observations which recover, more or less for free, some fundamental results about the spectrum.

We know there is a lattice isomorphism
\[
\Thick^{\sqrt{\otimes}}(\mathsf{K}) \cong \mathrm{Thom}(\Spc \mathsf{K}),
\]
that $\Spc \mathsf{K}$ is a coherent space, and correspondingly $\Thick^{\sqrt{\otimes}}(\mathsf{K})$ is a coherent frame. Given a map of spaces $f\colon X\to \Spc \mathsf{K}$ there is an induced map from closed Thomason subsets of the spectrum to the complete lattice $\Phi(X)$ of closed subsets of $X$
\[
f^{-1}\colon \Thom(\Spc \mathsf{K})^\omega \to \Phi(X)
\]
sending a closed subset with quasi-compact open complement to its preimage, a closed subset of $X$. (We note that $\Phi(X) \cong \Omega(X)^\op$ but as this equivalence involves taking the complement it can make what follows confusing.) As we just noted $\Thom(\Spc \mathsf{K})^\omega$ is lattice isomorphic to $\Thick^{\sqrt{\otimes}}(\mathsf{K})^\omega$ the lattice of principal radical ideals of $\mathsf{K}$ and so such a map determines a way of assigning closed subsets of $X$ to principal ideals, i.e.\ objects of $\mathsf{K}$, which is compatible with the lattice structure. Making this compatibility explicit gives the following properties.
 %Let us see how to exploit this a bit.

\begin{definition}\label{Gdef:supportdatum}
A support datum $\sigma$ on $\mathsf{K}$ with values in $X$ assigns to each $a\in \mathsf{K}$ a closed subset $\sigma(a)$ of $X$ such that:
\begin{itemize}
\item[(SD1)] $\sigma(0) = \varnothing$ and $\sigma(\unit) = X$
\item[(SD2)] $\sigma(a\oplus b) = \sigma(a)\cup \sigma(b)$
\item[(SD3)] $\sigma(\Sigma a) = \sigma(a)$
\item[(SD4)] given a triangle $a\to b\to c \to \Sigma a$ we have $\sigma(b)\subseteq \sigma(a)\cup \sigma(c)$
\item[(SD5)] $\sigma(a\otimes b) = \sigma(a)\cap \sigma(b)$.
\end{itemize}
\end{definition}

\begin{xca}
Check that the above conditions are equivalent to giving a bounded lattice map $\Thick^{\sqrt{\otimes}}(\mathsf{K})^\omega \to \Phi(X)$.
\end{xca}

We can extend the morphism $\Thom(\Spc \mathsf{K})^\omega \to \Phi(X)$ to a map $\Thom(\Spc \mathsf{K}) \to \mathcal{P}(X)$ by taking preimages of arbitrary Thomason subsets. These will not necessarily be closed, but they will be specialization closed subsets and they would even be Thomason provided $f$ were a quasi-compact map. If we instead chose to work with the dual topology on $\Spc \mathsf{K}$ then our life would be easier as we'd just have a map of frames
\[
\Thick^{\sqrt{\otimes}}(\mathsf{K}) = \Omega(\Spc \mathsf{K}^\vee) \to \Omega(X).
\]
If $X$ is coherent and $f$ is quasi-compact, which is the case we really care about if we want to compute the spectrum, then we can take duals freely and everything works nicely.

There is obviously, from our perspective, a universal support datum: it comes from the identity map $\Spc \mathsf{K} \to \Spc \mathsf{K}$ i.e.\ it is essentially the counit 
\[
\Thick^{\sqrt{\otimes}}(\mathsf{K})^\omega \stackrel{\sim}{\longrightarrow} \Thom(\Spc \mathsf{K})^\omega.
\] Indeed, given a support datum $\sigma$ or equivalently a lattice map $\Thom(\Spc \mathsf{K})^\omega \to \Phi(X) \cong \Omega(X)^\op$ we can (take opposites then) apply the adjunction to get maps
\[
X \to \pt\Omega(X) \to \Spc \mathsf{K}
\]
such that taking preimages recovers the support datum. 

We denote the universal support datum by $\supp$ i.e.\ for each $a\in \mathsf{K}$ we associate to it the closed subset with quasi-compact open complement $\supp(a)$ of $\Spc \mathsf{K}$. This is the subset corresponding to the lattice element $\sqrt{\thick^\otimes(a)}$. 

We can uniquely extend the map $\Thom(\Spc \mathsf{K})^\omega \to \Phi(X)$ to a map of complete lattices $\Thom(\Spc \mathsf{K}) \to \mathcal{P}(X)$, i.e.\ a map preserving joins and finite meets. Concretely if $\sigma$ is the corresponding support data this is just defined on a set of objects $\mathcal{A} \subseteq \mathsf{K}$ by
\[
\sigma(\mathcal{A}) = \sigma\left(\sqrt{\thick^\otimes(\mathcal{A})}\right) = \sigma\left(\bigvee_{a\in \mathcal{A}}\sqrt{\thick^\otimes(a)}\right) = \bigcup_{a\in \mathcal{A}} \sigma(a).
\]

\begin{remark}
Here we see a deficit of working with closed subsets: when we take unions we don't necessarily have a closed subset anymore and there is no reason that each $\sigma(a)$ should have quasi-compact open complement and so we don't necessarily get Thomason subsets but only unions of closed subsets. 
\end{remark}

We close this section with a couple of useful observations (all originally from Balmer's paper \cite{Balmer:2005a}). Let us start by unpacking how to describe $\Spc \mathsf{K}$ and the support in a more traditional way using primes.

\begin{proposition}
For an object $a\in \mathsf{K}$ we have
\[
\supp a = \{\mathsf{P} \in \Spc \mathsf{K} \mid a\notin \mathsf{P}\}
\]
and these subsets form a basis for the topology of $\Spc \mathsf{K}$.
\end{proposition}
\begin{proof}
The second statement is clear from the construction, so we only need to identify the support as claimed. By definition $\supp(a)$ is the Thomason subset (i.e.\ open subset of $(\Spc \mathsf{K})^\vee$) corresponding to the principal ideal $\sqrt{\thick^\otimes(a)}$. This is the set of points $p\colon \Thick^{\sqrt{\otimes}}(\mathsf{K}) \longrightarrow \{0,1\}$ such that $p\left(\sqrt{\thick^\otimes(a)}\right) = 1$. Under the bijection of points with prime elements of the lattice this is precisely those $\mathsf{P} \in \Spc \mathsf{K}$ such that \\ 
$\sqrt{\thick^\otimes(a)} \nsubseteq \mathsf{P}$ or equivalently such that $a\notin \mathsf{P}$. 
\end{proof}

A number of other useful properties follow in a straightforward way from the Stone duality point of view (in some cases the arguments are essentially identical to those in Balmer's paper which is also instructive). We indicate how one can deduce these results, but do not provide full details.

\begin{proposition}\label{Gprop:ttfacts}%[\cite{BaSpec}*{Proposition~2.3}]
Let $\mathsf{K}$ be a tt-category.
\begin{itemize}
\item[$(a)$] Let $S$ be a set of objects of $\mathsf{K}$ containing $\unit$ and such that if $k,l \in S$ then $k\otimes l \in S$. If $S$ does not contain $0$ then there exists a $\mathsf{P}\in \Spc \mathsf{K}$ such that $\mathsf{P} \cap S = \varnothing$.
\item[$(b)$] For any proper thick tensor-ideal $\mathsf{I} \subsetneq \mathsf{K}$ there exists a maximal proper thick tensor-ideal $\mathsf{M}$ with $\mathsf{I} \subseteq \mathsf{M}$.
\item[$(c)$] Maximal proper thick tensor-ideals are prime.
\item[$(d)$] The spectrum is not empty: $\Spc \mathsf{K} \neq \varnothing$. 
\item[$(e)$] For every closed subset $V\subseteq \Spc \mathsf{K}$ with quasi-compact open complement there exists an object $a\in \mathsf{K}$ with $\supp(a) = V$.
\end{itemize}
\end{proposition}
\begin{proof}
Let us briefly indicate the argument in each case (this text is important and definitely not only here to avoid bad spacing in the itemize environment).
\begin{itemize}
\item[(a)] The radical ideals generated by subsets of $S$ form a base for a filter and so this follows by a Zorn's lemma argument (see \cite[I.2.3 and I.2.4]{Johnstone}).
\item[(b)] This follows from the general properties of coherent spaces.
\item[(c)] A general fact about lattices, see \cite[Theorem~I.2.4]{Johnstone}.
\item[(d)] We assume that $\unit \neq 0$ and so $\Thick^{\sqrt{\otimes}}(\mathsf{K})$ has at least two elements and thus at least one point (the Stone dual of the empty set is the unique lattice structure on a singleton).
\item[(e)] By construction $V$ corresponds to a finitely presented element of the lattice $\Omega((\Spc \mathsf{K})^\vee)$ and hence to a finitely presented radical tensor-ideal. Every finitely generated ideal is principal and so picking $a$ to be any generator for this ideal gives the required object.
\end{itemize}
\end{proof}

Another very important property of the support is that it can detect whether or not an object is nilpotent.

\begin{lemma}
Given $a\in \mathsf{K}$ we have $\supp a = \varnothing$ if and only if $a^{\otimes n} = 0$ for some $n\geq 1$.
\end{lemma}
\begin{proof}
We have $\supp a = \varnothing$ precisely if every point of $\Thick^{\sqrt{\otimes}}(\mathsf{K})$ sends $a$ to $0$. This can only happen if $\sqrt{\thick^\otimes(a)}$ is the minimal element of the lattice of radical tensor-ideals. The minimal element is $\sqrt{0}$ which consists exactly of the tensor-nilpotent objects by Lemma~\ref{Glem:radical}.

%By \cite{BaSpec}*{Corollary~2.4} the support of $k$ is empty if and only if $k$ is tensor-nilpotent i.e., $k^{\otimes n}\cong 0$. As $\mathsf{K}$ is rigid there are no non-zero tensor-nilpotent objects by Remark~\ref{rem_rigid_facts}.
\end{proof}

\begin{lemma}\label{Glem:closure}%[\cite{BaSpec}*{Proposition~2.9}]
Let $\mathsf{P} \in \Spc\mathsf{K}$. The closure of $\mathsf{P}$ is
\[
\overline{\{\mathsf{P}\}} = \{\mathsf{Q}\in \Spc\mathsf{K} \; \vert \; \mathsf{Q}\subseteq \mathsf{P}\}.
\]
\end{lemma}
\begin{proof}
In $(\Spc \mathsf{K})^\vee$ the prime element $\mathsf{P}$ gives rise to the open subset $\{\mathsf{Q} \in \Spc \mathsf{K} \mid \mathsf{P} \nsubseteq \mathsf{Q}\}$. This has closed complement $\{\mathsf{Q} \in \Spc \mathsf{K} \mid \mathsf{P} \subseteq \mathsf{Q}\}$. One can check easily that any closed subset of $(\Spc \mathsf{K})^\vee$ is upward closed with respect to containment and so this must be the closure of $\mathsf{P}$ (we note that it is irreducible with generic point $\mathsf{P}$).

Passing from $(\Spc \mathsf{K})^\vee$ to $\Spc \mathsf{K}$ reverses the closure ordering: in $(\Spc \mathsf{K})^\vee$ we have $\overline{\mathsf{Q_1}} \subseteq \overline{\mathsf{Q}_2}$ if and only if $\mathsf{Q}_2\subseteq \mathsf{Q}_1$ and this is swapped upon taking the Hochster dual. It follows that the closure of $\mathsf{P}$ in $\Spc \mathsf{K}$ consists of the primes \emph{contained in} $\mathsf{P}$ as claimed.
%Let $S_0 = \K\setminus \sfP$ denote the complement of $\sfP$. It is immediate that $\sfP \in Z(S_0)$ and one easily checks that if $\sfP\in Z(S)$ then $S\cie S_0$. Thus for any such $S$ we have $Z(S_0)\cie Z(S)$, i.e.\ $Z(S_0)$ is the smallest closed subset containing $\sfP$. This shows
%\begin{displaymath}
%\overline{\{\sfP\}} = Z(S_0) = \{\sfQ\in \Spc\K \; \vert \; \sfQ\cie \sfP\}
%\end{displaymath}
%as claimed. The assertion that $\Spc\K$ is $T_0$ follows immediately.
\end{proof}

\begin{remark}\label{rem_gymnastics}
We see that the dual we have taken will cause us some initial psychic pain. The closure relation on primes is the reverse of the one from algebraic geometry. We will illustrate this in a concrete example in Section~\ref{Gsec:realexample}. In some sense, the order reversal is natural though. In commutative algebra the larger a prime ideal the smaller its complement and so one inverts fewer elements by passing to the localization at that prime (loosely speaking of course since all these things can have the same cardinality). On the other hand in a triangulated category the larger the prime ideal the more morphisms we invert in the corresponding Verdier quotient which is the tt-version of localizing at a prime.
\end{remark}

%----------------------------------------------------------------------------------------------------------------------------------------------

%----------------------------------------------------------------------------------------------------------------------------------------------

%----------------------------------------------------------------------------------------------------------------------------------------------

%----------------------------------------------------------------------------------------------------------------------------------------------

\subsection{Rigid tt-categories}

%----------------------------------------------------------------------------------------------------------------------------------------------

%----------------------------------------------------------------------------------------------------------------------------------------------

In this section we introduce an important class of tt-categories.

\begin{definition}\label{Gdefn:rigid}
Let $\mathsf{K}$ be an essentially small tensor triangulated category. Assume that $\mathsf{K}$ is closed symmetric monoidal, i.e.\ for each $k\in \mathsf{K}$ the functor $k\otimes-$ has a right adjoint which we denote $\hom(k,-)$ and these functors can be assembled into a bifunctor $\hom(-,-)$ which we call the \emph{internal hom} of $\mathsf{K}$. By definition one has, for all $a,b,c \in \mathsf{K}$, the tensor-hom adjunction
\begin{displaymath}
\mathsf{K}(a\otimes b, c) \cong \mathsf{K}(a, \hom(b,c)),
\end{displaymath}
with corresponding units and counits
\begin{displaymath}
\eta_{a,b} \colon b \to \hom(a,a\otimes b) \quad \text{and} \quad \varepsilon_{a,b} \colon \hom(a,b)\otimes a \to b.
\end{displaymath}
We define the \emph{dual} of $a\in \mathsf{K}$ to be the object
\begin{displaymath}
a^\vee = \hom(a,\unit).
\end{displaymath}
Given $a,b \in \mathsf{K}$ there is a natural \emph{evaluation map}
\begin{displaymath}
a^\vee \otimes b \to \hom(a,b),
\end{displaymath}
which is defined by following the identity map on $b$ through the composite
\begin{displaymath}
\begin{tikzcd}
\mathsf{K}(b,b) \arrow[r, "\sim"] & \mathsf{K}(b\otimes \unit, b) \arrow[rr, "{\mathsf{K}(b\otimes \varepsilon_{a,\unit}, b)}"] && \mathsf{K}(a\otimes a^\vee \otimes b, b) \arrow[r, "\sim"] & \mathsf{K}(a^\vee \otimes b, \hom(a,b))
\end{tikzcd}
%\xymatrix{
%\mathsf{K}(b,b) \ar[r]^-\sim & \mathsf{K}(b\otimes \unit, b) \ar[rr]^-{\mathsf{K}(b\otimes \varepsilon_{a,\unit}, b)} && \mathsf{K}(a\otimes a^\vee \otimes b, b) \ar[r]^-\sim & \mathsf{K}(a^\vee \otimes b, \hom(a,b)).
%}
\end{displaymath}
We say that $\mathsf{K}$ is \emph{rigid} if for all $a,b \in \mathsf{K}$ this natural evaluation map is an isomorphism
\begin{displaymath}
a^\vee \otimes b \stackrel{\sim}{\to} \hom(a,b).
\end{displaymath}
\end{definition}

\begin{remark}
We are using the same notation for the dual internal to $\mathsf{K}$ and the Hochster dual but these are sufficiently different types of operations that this seems unlikely to cause us any confusion.
\end{remark}

\begin{remark}
If $\mathsf{K}$ is rigid then, given $a\in \mathsf{K}$, there is a natural isomorphism
\begin{displaymath}
(a^\vee)^\vee \cong a,
\end{displaymath}
and the functor $a^\vee\otimes -$ is both a left and a right adjoint to the functor $a\otimes-$ of tensoring with $a$ on the left (and hence also to tensoring on the right by symmetry).
\end{remark}

\begin{example}
Let us discuss the rigidity, or otherwise, of the examples from Example~\ref{Gex:ttcats}
\begin{itemize}
\item[(1),(2)] Given a commutative ring $R$ the category $\mathsf{D}^\mathrm{perf}(R)$ is rigid. This is clear if we view it as $\mathsf{K}^\mathrm{b}(\mathrm{proj} R)$ as we are then working with the usual tensor product over $R$ of projective modules and reduce to an easy computation with finite rank free modules. For a quasi-compact and quasi-separated scheme $X$ the category of perfect complexes $\mathsf{D}^\mathrm{perf}(X)$ is also rigid, which one can check locally using the fact for rings.
\item[(3)] If $A$ is a finite dimensional cocommutative Hopf algebra then $\underline{\modu}\: A$ is a rigid tt-category. In this case the required duality is essentially just the corresponding duality on the category of vector spaces over the ground field.
\item[(4)] The finite stable homotopy category $\mathrm{SH}^\mathrm{fin}$ together with the smash product of spectra is a rigid tt-category.
\item[(5)] These examples may or may not be rigid and this depends on the choice of symmetric monoidal additive category $\mathsf{M}$. In fact, $\mathsf{K}^\mathrm{b}(\mathsf{M})$ is rigid precisely when $\mathsf{M}$ is rigid (cf.\ (1) above).
\end{itemize}
\end{example}

Now let us discuss one of the main features of rigid tt-categories that simplifies the theory. Given $a\in \mathsf{K}$ the adjunction between $a\otimes-$ and $a^\vee\otimes-$ implies that $a$ is a summand of $a\otimes a \otimes a^\vee$ by the triangle identities for adjunction. Dually we have that $a^\vee$ is a summand of $a\otimes a^\vee \otimes a^\vee$. The former implies that every tensor-ideal is radical:
\[
\Thick^\otimes(\mathsf{K}) = \Thick^{\sqrt{\otimes}}(\mathsf{K})
\]
and hence the whole lattice of tensor-ideals is a coherent frame. The latter tells us that every tensor-ideal is closed under taking duals and hence $\supp(a) = \supp(a^\vee)$. 

As an important special case we see that if $a\in \mathsf{K}$ is tensor-nilpotent, i.e.\ $a^{\otimes n}\cong 0$ for some $n$, then $a\cong 0$.

\begin{remark}
Rigidity also plays a special role when we consider compactly generated triangulated categories and is crucial in the current approaches to classifying localizing tensor-ideals through supports.
\end{remark}

%----------------------------------------------------------------------------------------------------------------------------------------------

%----------------------------------------------------------------------------------------------------------------------------------------------

\section{The spectrum in action}

In this chapter we do some long overdue examples. We start by giving an explicit computation of the spectrum for a polynomial ring $k[x]$ over a field $k$ and compare the usual structures of algebraic geometry with those in the tt setting. We then discuss computations of the spectrum for various examples and give a proof, using tt techniques, of the computation of the spectrum for the perfect complexes on a quasi-compact and quasi-separated scheme.

%----------------------------------------------------------------------------------------------------------------------------------------------

%----------------------------------------------------------------------------------------------------------------------------------------------

\subsection{An honest example: the polynomial ring by hand}\label{Gsec:realexample}

In this section we will compute many things in the example $\mathsf{D}^\mathrm{perf}(k[x])$, for some fixed ground field $k$, explicitly. This is possible because the ring $k[x]$ is rather simple, but it is rich enough to indicate many features of the theory.

The prime ideals of $k[x]$ are the zero ideal $(0)$ and the ideals $(f)$ where $f$ is an irreducible polynomial. Thus $\Spec k[x]$ is an irreducible space with generic point $(0)$ and a closed point for each irreducible ideal. In the special case that $k$ is algebraically closed the irreducible polynomials are precisely the linear ones and the closed points are naturally identified with $k$.

Every ideal of $k[x]$ is finitely generated (and in fact principal) and so the lattice $\mathrm{Rad}(k[x])$ consists entirely of finitely presented elements. It is naturally isomorphic to the lattice of open subsets of $\Spec k[x]$ as a consequence of Stone duality. Under this correspondence a radical ideal $(f)$, that is an ideal generated by a square-free product of irreducible polynomials $f$, corresponds to the basic open subset $D(f)$ of prime ideals not containing $f$ (in other words the irreducible polynomials not dividing $f$).

We have a complete understanding of the finitely generated $k[x]$-modules. Every finitely generated projective module is free. The remaining modules are finite dimensional and every finite dimensional module is a direct sum of indecomposable finite dimensional modules in an essentially unique way. The finite dimensional indecomposable modules have the form $k[x]/(f^n)$ where $f$ is an irreducible polynomial (that is $(f)$ is a prime ideal). We call $k[x]/(f)$ the \emph{residue field} at $(f)$.

Thus there is a family of indecomposable finite dimensional modules for each closed point of $\Spec k[x]$. For a fixed irreducible polynomial $f$ one has an indecomposable module $k[x]/(f^n)$ for each $n\geq 1$ which has filtration
\[
0 \subseteq (f^{n-1})/(f^n) \subseteq (f^{n-2})/(f^n) \subseteq \cdots \subseteq (f)/(f^n) \subseteq k[x]/(f^n)
\]
where each subquotient is isomorphic to $k[x]/(f)$. In particular, we see that
\[
\thick(k[x]/(f)) \supseteq \thick(k[x]/(f^n)).
\]
In fact, these two thick subcategories are equal. The proof illustrates a general principle which is the key to understanding the case of an arbitrary commutative ring. This is a good point to remind ourselves that, because $k[x]$ generates $\mathsf{D}^\mathrm{perf}(k[x])$, it follows from Lemma~\ref{Glem:small_unit_gen} that every thick subcategory is an ideal.

\begin{lemma}\label{Glem:torsionkx}
There is an equality of thick subcategories 
\[
\thick(k[x]/(f)) = \thick(k[x]/(f^n))
\]
for each $n\geq 1$.
\end{lemma}
\begin{proof}
We already observed one containment and so it is enough to check that $k[x]/(f)$ lies in $\thick(k[x]/(f^n))$. The derived tensor product $k[x]/(f)\otimes^\mathrm{L} k[x]/(f^n)$ can be computed by tensoring the projective resolution
\[
k[x] \stackrel{f^n}{\longrightarrow} k[x]
\]
of $k[x]/(f^n)$ with $k[x]/(f)$. We see that
\[
k[x]/(f)\otimes^\mathrm{L} k[x]/(f^n) \cong k[x]/(f) \oplus \Sigma k[x]/(f)
\]
and so $k[x]/(f) \in \thick^\otimes(k[x]/(f^n)) = \thick(k[x]/(f^n))$.
\end{proof}

\begin{remark}\label{Grem:field}
Actually the fact that $k[x]/(f)\otimes^\mathrm{L} k[x]/(f^n)$ is a direct sum of shifts of $k[x]/(f)$ is true for completely formal reasons. We can factor tensoring with $k[x]/(f)$ through the category of $k[x]/(f)$-vector spaces where every complex is a sum of shifts of $k[x]/(f)$. This only relies on the fact that $k[x] \to k[x]/(f)$ is a map to a field; the same argument works for any map $R\to k$ from a commutative ring $R$ to a field $k$.
\end{remark}

\begin{remark}
One can also see this equality of thick subcategories directly, without tensoring, using the existence of short exact sequences
\[
0 \longrightarrow k[x]/(f^n) \longrightarrow k[x]/(f^{n+1}) \oplus k[x]/(f^{n-1}) \longrightarrow k[x]/(f^n) \longrightarrow 0.
\]
These are a special instance of \emph{Auslander-Reiten sequences}.
\end{remark}

Now let us continue on by introducing an important structural property of $\mathsf{D}^\mathrm{perf}(k[x])$. The ring $k[x]$ is hereditary, i.e.\ it has global dimension $1$, and so every bounded complex of finitely generated $k[x]$-modules is quasi-isomorphic to a bounded complex of projectives. Thus $\mathsf{D}^\mathrm{perf}(k[x]) = \mathsf{D}^\mathrm{b}(\modu k[x])$. Moreover, every complex over $k[x]$ is quasi-isomorphic to its cohomology modules placed in the appropriate degrees. We say that such a complex is \emph{formal}.

\begin{lemma}\label{Glem:hereditary}
If $A$ is a ring of global dimension $1$ then every complex $C$ of $A$-modules satisfies
\[
C \cong \bigoplus_i \Sigma^{-i}H^i(C)
\]
in the derived category $\mathsf{D}(A)$.
\end{lemma}
\begin{proof}
Consider for $n\in \mathbb{Z}$ the exact sequence
\[
0 \to \ker d^{n-1} \to C^{n-1} \to \ker d^{n} \to H^n(C) \to 0
\]
which represents a class in $\Ext^2(H^n(C), \ker d^{n-1})$. This Ext group vanishes as $A$ has global dimension $1$ and so this class must be trivial. Thus there exists an $A$-module $E^n$ fitting into a commutative diagram
\[
\begin{tikzcd}
            &                                                       & 0 \arrow[d]                                     & 0 \arrow[d]                   &   \\
0 \arrow[r] & \ker d^{n-1} \arrow[r] \arrow[d, equal] & C^{n-1} \arrow[r] \arrow[d]                     & \im d^{n} \arrow[r] \arrow[d] & 0 \\
0 \arrow[r] & \ker d^{n-1} \arrow[r]                                & E^n \arrow[r] \arrow[d]                         & \ker d^n \arrow[r] \arrow[d]  & 0 \\
            &                                                       & H^n(C) \arrow[r, equal] \arrow[d] & H^n(C) \arrow[d]              &   \\
            &                                                       & 0                                               & 0                             &  
\end{tikzcd}
\]
where the square in the top right is a pullback. Using this we can define a morphism in the derived category as the span
\[
\begin{tikzcd}
\cdots \arrow[r] \arrow[d] & 0 \arrow[r]                     & 0 \arrow[r]                           & H^n(C) \arrow[r]                  & 0 \arrow[r]                     & \cdots \arrow[d] \\
\cdots \arrow[r] \arrow[d] & 0 \arrow[u] \arrow[d] \arrow[r] & C^{n-1} \arrow[u] \arrow[d] \arrow[r] & E^n \arrow[u] \arrow[d] \arrow[r] & 0 \arrow[u] \arrow[d] \arrow[r] & \cdots \arrow[d] \\
\cdots \arrow[r]           & C^{n-2} \arrow[r]               & C^{n-1} \arrow[r]                     & C^{n} \arrow[r]                   & C^{n+1} \arrow[r]               & \cdots          
\end{tikzcd}
\]
giving a morphism $\Sigma^{-n}H^n(C) \to C$ which induces an isomorphism in $H^n$. Taking the induced map $\oplus_i \Sigma^{-i}H^i(C) \to C$ gives the desired quasi-isomorphism.
\end{proof}

\begin{remark}
The proof doesn't require a ring and works for any hereditary abelian category.
\end{remark}

Thus giving an object of $\mathsf{D}^\mathrm{perf}(k[x])$ is equivalent to specifying finitely many finitely generated $k[x]$-modules and the degrees in which they live. The morphisms are given by usual module morphisms in degree $0$ and by $\Ext^1$ in degree $1$.

In particular, every thick subcategory can be generated by $k[x]$-modules and is, in fact, determined by the indecomposable modules it contains. As noted above Lemma~\ref{Glem:small_unit_gen} tells us that 
\[
\Thick(\mathsf{D}^\mathrm{perf}(k[x])) = \Thick^\otimes(\mathsf{D}^\mathrm{perf}(k[x]))
\]
and $\mathsf{D}^\mathrm{perf}(k[x])$ is rigid so every tensor-ideal is radical.

Let us now describe the prime ideals of $\mathsf{D}^\mathrm{perf}(k[x])$.

\begin{lemma}
If $f$ and $g$ are distinct irreducible polynomials then the residue fields $k[x]/(f)$ and $k[x]/(g)$ satisfy
\[
k[x]/(f) \otimes^\mathrm{L} k[x]/(g) = 0.
\]
In particular, if $\mathsf{P} \in \Spc \mathsf{D}^\mathrm{perf}(k[x])$ then there is at most one irreducible polynomial $f$ such that $k[x]/(f) \notin \mathsf{P}$. 
\end{lemma}
\begin{proof}
We can compute this directly by tensoring the resolution $k[x] \stackrel{f}{\longrightarrow} k[x]$ of $k[x]/(f)$ with $k[x]/(g)$ to get $k[x]/(g) \stackrel{f}{\longrightarrow} k[x]/(g)$. Because $g$ and $f$ are coprime we have that multiplication by $f$ is an isomorphism on $k[x]/(g)$ and so this complex is acyclic. The statement about prime ideals is then immediate from their definition as any prime ideal contains $0$ and hence at least one of any two objects which tensor to $0$.
\end{proof}

\begin{remark}
Again this is true for some formal reason that generalizes: the cohomology of $k[x]/(f) \otimes^\mathrm{L} k[x]/(g)$ would have to be a $k[x]/(f)$-vector space and a $k[x]/(g)$-vector space and this can only happen if it is zero.
\end{remark}

\begin{xca}
Verify that if $R$ is a commutative ring and $\mathfrak{p}$ and $\mathfrak{q}$ are distinct prime ideals then $k(\mathfrak{p})\otimes k(\mathfrak{q}) = 0$ where $k(\mathfrak{p}) = (R/\mathfrak{p})_{(0)}$ is the residue field, i.e.\ the field of fractions of the domain $R/\mathfrak{p}$.
\end{xca}

For a $k[x]$-module $M$ we denote by $\mathrm{ann}(M)$ its annihilator and for a perfect complex $E$ we let $H^*(E) = \oplus_i H^i(E)$.

\begin{proposition}\label{Gprop:kxprimes}
The prime ideals of $\mathsf{D}^\mathrm{perf}(k[x])$ are given by
\begin{align*}
\mathsf{D}^\mathrm{perf}_{\mathrm{tors}}(k[x]) &= \{E\in \mathsf{D}^\mathrm{perf}(k[x]) \mid \mathrm{ann}(H^*(E)) \neq 0\} \\
&= \{E\in \mathsf{D}^\mathrm{perf}(k[x]) \mid H^*(E)_{(0)} = 0\} \\
&= \thick(k[x]/(f) \mid (f)\in \Spec k[x]\setminus \{(0)\})
\end{align*}
and for each $(f) \in \Spec k[x] \setminus \{(0)\}$ the ideal
\begin{align*}
\mathsf{D}^\mathrm{perf}_{(f)}(k[x]) &= \{E\in \mathsf{D}^\mathrm{perf}(k[x]) \mid H^*(E)_{(f)} = 0\} \\
&= \thick(k[x]/(g) \mid (g)\in \Spec k[x]\setminus \{(0),(f)\}).
\end{align*}
\end{proposition}
\begin{proof}
It is straightforward to check that the ideals defined above are all proper and, using what we've explained about objects of $\mathsf{D}^\mathrm{perf}(k[x])$, that the two descriptions given in each case agree. 

The ideal $\mathsf{D}^\mathrm{perf}_{\mathrm{tors}}(k[x])$ is maximal: it contains every indecomposable module except for $k[x]$ and we clearly cannot add the tensor unit and still have a proper ideal. It also follows directly from this that it is prime since a tensor product of non-torsion (i.e.\ free) modules is not torsion. One can also deduce it is prime from Proposition~\ref{Gprop:ttfacts}(c).

It remains to show that $\mathsf{D}^\mathrm{perf}_{(f)}(k[x])$ is prime. One can see this in several ways. The most direct is to use our description of the objects of $\mathsf{D}^\mathrm{perf}(k[x])$. It is sufficient to check that if the derived tensor product of two indecomposable modules lands in $\mathsf{D}^\mathrm{perf}_{(f)}(k[x])$ then it already contained one of them. This is clear if one of the modules is $k[x]$ and so the only case to consider is when we tensor two finite dimensional indecomposable modules. We already checked that $k[x]/(f) \otimes^\mathrm{L} k[x]/(g) = 0$ and so 
\[
\thick(k[x]/(f)) \otimes^\mathrm{L} \thick(k[x]/(g)) = 0
\]
i.e.\ $k[x]/(f^n) \otimes^\mathrm{L} k[x]/(g^m) = 0$ for all $n,m\geq 1$. Thus the only ways to tensor two indecomposable finite dimensional modules and get an object of $\mathsf{D}^\mathrm{perf}_{(f)}(k[x])$ are either to tensor two modules which are already contained in it, or to tensor with $k[x]/(f^n)$ and get zero which is ok since all other families $k[x]/(g^m)$ are already contained in $\mathsf{D}^\mathrm{perf}_{(f)}(k[x])$.
\end{proof}

The phrasing of the proposition gives an obvious bijection between the sets $\Spec k[x]$ and $\Spc \mathsf{D}^\mathrm{perf}(k[x])$ which is \emph{order reversing} the minimal prime ideal $(0)$ of $k[x]$ corresponds to the maximal prime ideal $\mathsf{D}^\mathrm{perf}_{\mathrm{tors}}(k[x])$ of $\mathsf{D}^\mathrm{perf}(k[x])$. This coincides with the order reversal in the closure operation as in Lemma~\ref{Glem:closure}. 

For brevity, and to avoid case distinctions, let us denote by $\mathsf{P}(\mathfrak{p})$ the prime tensor-ideal corresponding to the prime ideal $\mathfrak{p}$ of $k[x]$. This bijection is really natural as the following result demonstrates (the first part is already implicit in Proposition~\ref{Gprop:kxprimes}).

\begin{proposition}
We have  
\[
\mathsf{P}(\mathfrak{p}) = \ker ((-)\otimes k[x]_\mathfrak{p}\colon \mathsf{D}^\mathrm{perf}(k[x]) \to \mathsf{D}^\mathrm{perf}(k[x]_\mathfrak{p})).
\]
and in fact $\mathsf{D}^\mathrm{perf}(k[x]) / \mathsf{P}(\mathfrak{p})$ is equivalent to $\mathsf{D}^\mathrm{perf}(k[x]_\mathfrak{p})$.
\end{proposition}
\begin{proof}
The map $k[x] \to k[x]_\mathfrak{p}$ is a flat epimorphism. Thus the induced base change functor $\Modu k[x] \to \Modu k[x]_\mathfrak{p}$ is exact and its right adjoint, restriction of scalars, is fully faithful and preserves injectives. Deriving this we see that $\mathsf{D}(k[x]_\mathfrak{p})$ is the Verdier localization of $\mathsf{D}(k[x])$ at the coproduct closed triangulated subcategory
\[
\{E \in \mathsf{D}(k[x]) \mid E\otimes k[x]_\mathfrak{p} \cong 0\}.
\]
(Looking ahead, $k[x]_\mathfrak{p}$ is an idempotent algebra in $\mathsf{D}(k[x])$, see Section~\ref{Gsec:loc}.) This localization preserves compacts and one then easily checks the category $\mathsf{D}^\mathrm{perf}(k[x]_\mathfrak{p})$ is the idempotent completion of the Verdier quotient of $\mathsf{D}^\mathrm{perf}(k[x])$ by 
\[
\{E \in \mathsf{D}(k[x]) \mid E\otimes k[x]_\mathfrak{p} \cong 0\} \cap \mathsf{D}^\mathrm{perf}(k[x]) = \ker(\mathsf{D}^\mathrm{perf}(k[x]) \to \mathsf{D}^\mathrm{perf}(k[x]_\mathfrak{p})).
\]
This kernel is the thick subcategory generated by those indecomposable modules $M$ such that $M\otimes k[x]_\mathfrak{p}=0$, i.e.\ those $M$ which are annihilated by some non-zero polynomial in $k[x] \setminus \mathfrak{p}$. This is precisely $\mathsf{P}(\mathfrak{p})$. Given this description one can check by hand that the idempotent completion is not necessary.
\end{proof}

\begin{remark}
In general the idempotent completion after taking the localization is required.
\end{remark}

\begin{remark}\label{Grem:rho}
One can also start with the prime tensor-ideal $\mathsf{P}$ and produce a prime ideal of $k[x]$ following \cite{Balmer:2010b}. We set
\[
\rho(\mathsf{P}) = \{f\in k[x] \mid \mathrm{cone}(k[x] \stackrel{f}{\to} k[x]) \notin \mathsf{P} \}.
\]
This is a prime ideal, and $\rho$ gives an inverse to the bijection $\mathfrak{p} \mapsto \mathsf{P}(\mathfrak{p})$. (We note that this doesn't really depend on knowing we're working with $\mathsf{D}^\mathrm{perf}(k[x])$ as we can recover $k[x]$ as the endomorphism ring of the unit object.)

The intuition is that killing $\mathsf{P}$ inverts every morphism $f$ whose cone is in $\mathsf{P}$ and in particular, if $\mathrm{cone}(k[x] \stackrel{f}{\to} k[x])$ were in $\mathsf{P}$ then it would become invertible in the Verdier quotient by $\mathsf{P}$. So $\mathsf{P}$ is a closer relative of $k[x]\setminus \mathfrak{p}$, the set of maps we invert when localizing at $\mathfrak{p}$, than of $\mathfrak{p}$ itself. This gives another explanation for the order reversal we have observed.
\end{remark}

By now, if we didn't already know, we would have guessed that $\Spc \mathsf{D}^\mathrm{perf}(k[x])$ is just $\Spec k[x]$. Let's verify this by describing explicitly the lattice of tensor-ideals and identifying it with $\Thom(\Spec k[x])$ the dual of the frame of radical ideals of $k[x]$.

\begin{remark}
Because $k[x]$ is noetherian every open subset of $\Spec k[x]$ is quasi-compact and so a subset is Thomason if and only if it is a possibly infinite union of closed subsets. We call such a subset \emph{specialization closed}. Thus the proper non-empty Thomason subsets are just given by arbitrary subsets of closed points of $\Spec k[x]$.
\end{remark}

Let us define functions 
\[
\phi\colon \Thick^{\sqrt{\otimes}}(\mathsf{D}^\mathrm{perf}(k[x])) \to \Thom(\Spec k[x])
\]
and 
\[
\psi \colon \Thom(\Spec k[x]) \to \Thick^{\sqrt{\otimes}}(\mathsf{D}^\mathrm{perf}(k[x]))
\]
as follows. For a thick tensor-ideal $\mathsf{J}$ we set
\[
\phi(\mathsf{J}) = \{(f) \in \Spec k[x] \mid k[x]/(f) \in \mathsf{J} \}.
\]
For example
\begin{itemize}
\item $\phi(0) = \varnothing$
\item $\phi(\mathsf{D}^\mathrm{perf}(k[x])) = \Spec k[x]$
\item $\phi(\mathsf{D}^\mathrm{perf}_{\mathrm{tors}}(k[x])) = \Spec k[x] \setminus \{(0)\}$
\item $\phi(\mathsf{D}^\mathrm{perf}_{(f)}(k[x])) = \Spec k[x] \setminus \{(0),(f)\}$
\item $\phi(\thick(k[x]/(g))) = V(g) = \{(f)\in \Spec k[x] \mid g\in (f)\}$
\end{itemize}
For a Thomason subset $V \subseteq \Spec k[x]$ we set
\[
\psi(V) = \thick(k[x]/(f) \mid (f) \in V).
\] 
So for instance
\begin{itemize}
\item $\psi(\varnothing) = 0$
\item $\psi(\Spec k[x]) = \mathsf{D}^\mathrm{perf}(k[x])$
\item $\psi(\Spec k[x] \setminus \{(0)\}) = \mathsf{D}^\mathrm{perf}_{\mathrm{tors}}(k[x])$
\end{itemize}
and by now you probably get the idea.

\begin{proposition}
The maps $\phi$ and $\psi$ give an isomorphism of coherent frames
\[
\Thom(\Spec k[x]) \cong \Thick^{\sqrt{\otimes}}(\mathsf{D}^\mathrm{perf}(k[x])).
\]
Hence there is a homeomorphism $\Spec k[x] \to \Spc \mathsf{D}^\mathrm{perf}(k[x])$ and it acts on prime ideals by $\mathfrak{p} \mapsto \mathsf{P}(\mathfrak{p})$.
\end{proposition}
\begin{proof}
It is clear that both $\phi$ and $\psi$ are order preserving and they both preserve joins more or less by definition. It's a straightforward exercise to check they both preserve finite (and in fact all) meets and finitely presented objects. Let us argue that they are inverse to one another to complete the proof. Suppose that $\mathsf{J}$ is a tensor-ideal and $V$ is a Thomason subset.

It is evident that $\phi\psi(V) \supseteq V$ and we need to verify that no additional points can appear. If $V$ is proper, i.e.\ there is something to check, then $\psi(V)$ is generated by torsion modules and so we cannot have $k[x] \in \psi(V)$ (since being torsion is preserved by finite sums and extensions). If $(f)\notin V$ is a closed point then the exact functor $k[x]/(f)\otimes^\mathrm{L}(-)$ kills the generators of $\psi(V)$ and hence kills $\psi(V)$. But $k[x]/(f) \otimes^\mathrm{L} k[x]/(f) \neq 0$ and so we can't have $(f)\in \phi\psi(V)$.

Starting with $\mathsf{J}$ it is evident that $\psi\phi(\mathsf{J}) \subseteq \mathsf{J}$. By Lemma~\ref{Glem:hereditary} it is enough to check that these two thick subcategories contain the same indecomposable modules. If $\mathsf{J}$ contains $k[x]$ then $(0)\in \phi(\mathsf{J})$ and so $\psi\phi(\mathsf{J})$ is also all of $\mathsf{D}^\mathrm{perf}(k[x])$. So we are reduced to the case that $\phi(\mathsf{J})$ is a collection of closed points and we need to show that if $k[x]/(f^n)$ lies in $\mathsf{J}$ then it is also in $\psi\phi(\mathsf{J})$. We already saw in Lemma~\ref{Glem:torsionkx} that this boils down to specifying which residue fields $\mathsf{J}$ contains and this is exactly what $\phi$ does. 
\end{proof}

Finally we should discuss supports. The abstract machinery gives a support for objects of $\mathsf{D}^\mathrm{perf}(k[x])$ valued in $\Spc \mathsf{D}^\mathrm{perf}(k[x])$ which we now know is homeomorphic to $\Spec k[x]$. There is also a support on $\mathsf{D}^\mathrm{perf}(k[x])$ coming from geometry which we will temporarily denote by $\sigma$. It is defined on a perfect complex $E$ by
\begin{align*}
\sigma(E) &= \{\mathfrak{p}\in \Spec k[x] \mid E_\mathfrak{p} \neq 0\} \\
&= \{\mathfrak{p}\in \Spec k[x] \mid H^*(E)_\mathfrak{p} \neq 0\} \\
&= \cup_i \sigma(H^i(E))
\end{align*}
here $E_\mathfrak{p} = E\otimes k[x]_\mathfrak{p}$ and the equalities are a small exercise in exactness of localization and commutative algebra. The subset $\sigma(H^i(E))$ is nothing but the usual support of the $k[x]$-module $H^i(E)$. In particular, $\sigma(E)$ is always a closed subset of $\Spec k[x]$ and so setting, for a tensor-ideal $\mathsf{J}$, 
\[
\sigma(\mathsf{J}) = \bigcup_{E\in \mathsf{J}} \sigma(E)
\]
gives a Thomason subset of $\Spec k[x]$.

\begin{xca}
Show that $\sigma(\mathsf{J}) = \phi(\mathsf{J})$.
\end{xca}

Thus the support $\sigma$ of a perfect complex gives the lattice isomorphism between Thomason subsets and thick tensor-ideals and thus corresponds to the universal support $\supp$ under the associated bijection of spectra.

In Section~\ref{Gsec:thomasonproof} we will discuss how one can generalize this discussion to an arbitrary commutative ring. One cannot work explicitly in that generality, but the general yoga of using the residue fields to control the situation remains intact. There are also other arguments using tensor-nilpotence, for instance this is the approach taken by Thomason in his paper \cite{Thomason:1997a} giving the first proof of the homeomorphism between the spectra of the ring and of its perfect complexes (although not in that language) without any noetherian hypothesis.

%----------------------------------------------------------------------------------------------------------------------------------------------

%----------------------------------------------------------------------------------------------------------------------------------------------

%----------------------------------------------------------------------------------------------------------------------------------------------

%----------------------------------------------------------------------------------------------------------------------------------------------

\subsection{Further computations from the literature}

In this section we will give a brief overview of some of the families of examples where the spectrum has been computed. We don't try to be exhaustive, but rather to give some important and/or illustrative examples. A good overview, at the time of its writing, is given in \cite{Balmerguide}.

We note that many of the examples were computed before the language of tt-geometry. We will not take any care to make this distinction when discussing the examples, i.e.\ we will retcon the older work to exist in a universe where tt-geometry has always been.

%----------------------------------------------------------------------------------------------------------------------------------------------

\subsubsection{Algebraic geometry}

Let us begin in the affine setting. If $R$ is a commutative ring then we have the rigid tt-category $\mathsf{D}^\mathrm{perf}(R)$ of perfect complexes over $R$. This is generated by the unit $R$ and so the lattices of radical tensor-ideals, tensor-ideals, and thick subcategories all coincide. The map $\rho$ of Remark~\ref{Grem:rho} still works at this level of generality and gives a continuous map 
\[
\rho\colon \Spc \mathsf{D}^\mathrm{perf}(R) \to \Spec R, \quad \rho(\mathsf{P}) = \{r\in R\mid R\stackrel{r}{\to}R \notin \mathsf{P}\}.
\]
It was proved by Hopkins and Neeman \cite{Neeman:1992a} in the noetherian case and by Thomason \cite{Thomason:1997a} in general that $\rho$ is a homeomorphism. Thus:
\begin{itemize}
\item $\Spc \mathsf{D}^\mathrm{perf}(R) \cong \Spec R$ recovers the Zariski spectrum of $R$
\item thick subcategories of $\mathsf{D}^\mathrm{perf}(R)$ are classified by Thomason subsets of $\Spec R$
\item the universal support theory is given by the $R$-module theoretic support of the cohomology.
\end{itemize}

The approach in Neeman's work passes through the unbounded derived category and utilizes the residue fields $k(\mathfrak{p}) = (R/\mathfrak{p})_{(0)}$ at the prime ideals of $R$. Thomason's approach is based on a tensor-nilpotence argument for morphisms. In Section~\ref{Gsec:thomasonproof} we will give a tensor triangular proof of Thomason's result using residue fields; we will not really discuss tensor-nilpotence in these notes although it is an important and useful perspective.

In fact, Thomason proved the extension not just to not necessarily noetherian commutative rings but to a very general class of schemes. Suppose that $X$ is a quasi-compact and quasi-separated scheme. These conditions guarantee that the underlying space of $X$ is coherent and so is eligible to be the spectrum of a tt-category. Then, as we have mentioned, $\mathsf{D}^\mathrm{perf}(X)$ is a rigid tt-category. Thomason showed that 
\[
\Spc \mathsf{D}^\mathrm{perf}(X) \cong X
\]
as spaces and so the tensor-ideals of $X$ are classified by the Thomason subsets of $X$. In general one really needs the tensor product here as the lattice of thick subcategories of $\Spc \mathsf{D}^\mathrm{perf}(X)$ need not be distributive! In fact, this already happens for the projective line over a field. Classifying general thick subcategories of $\Spc \mathsf{D}^\mathrm{perf}(X)$ is currently completely out of reach: the only projective examples in which this is fully understood are the projective line and elliptic curves where one has a complete understanding of the objects in the bounded derived category of coherent sheaves.

\begin{remark}\label{Grem:lrs}
In fact, one can do better than just recovering $X$ as a space via the spectrum of the perfect complexes and its the derived tensor product. For a tt-category $\mathsf{K}$ one can equip $\Spc \mathsf{K}$ with a sheaf of rings making it into a locally ringed space (see \cite{Balmer:2005a} and \cite{Balmer:2010b}). This is done by defining a presheaf valued on the quasi-compact open subsets of $\Spc \mathsf{K}$ by
\[
U \mapsto \Hom_{\mathsf{K}/\mathsf{K}_{U^c}}(\unit, \unit)
\]
where $\mathsf{K}_{U^c}$ is the radical tensor-ideal corresponding to the complement of $U$. Since the quasi-compact open subsets are a base for the topology this determines a sheaf, namely the structure sheaf $\mathcal{O}_\mathsf{K}$.

If one applies this construction in the case of $\mathsf{D}^\mathrm{perf}(X)$ we recover $X$ as a locally ringed space, i.e.\ we get a scheme. A categorification of this procedure is discussed in \cite{highergeometry}.
\end{remark}

\begin{remark}
The use of the tensor product to recover $X$ from $\mathsf{D}^\mathrm{perf}(X)$ is really necessary in general. There are many examples where $X$ and $Y$ are non-isomorphic schemes such that $\mathsf{D}^\mathrm{perf}(X) \cong \mathsf{D}^\mathrm{perf}(Y)$. What the work discussed above shows is that such an equivalence cannot be monoidal.

However, there are cases where the property of being a tensor-ideal is intrinsic and so one can reconstruct $X$ from $\mathsf{D}^\mathrm{perf}(X)$ without recourse to the monoidal structure. For instance, this happens for smooth schemes where the canonical bundle is either ample or anti-ample: up to a shift the category of perfect complexes over a smooth scheme knows the canonical bundle and the (anti-)ampleness can be used to detect tensor-ideals.
\end{remark}

%----------------------------------------------------------------------------------------------------------------------------------------------

%----------------------------------------------------------------------------------------------------------------------------------------------

\subsubsection{Homotopy theory}

The example of finite spectra was computed in \cite{Devinatz/Hopkins/Smith:1988a}. It was the first of these computations and an inspiration for much of what followed. It is fair to say that the whole subject has its roots in (chromatic) homotopy theory.

The triangulated category in question is $\mathrm{ho}(\mathsf{Sp}^\omega)$ the compact objects in the stable homotopy category of spectra or, equivalently, the Spanier-Whitehead category of finite CW complexes. Roughly speaking, this category is built by inverting suspension on the homotopy category of finite CW complexes and captures the stable homotopy groups of finite CW complexes (motivated by the Freudenthal suspension theorem). We don't recall the precise construction here, but just recall that the objects are (de)suspensions of finite CW complexes, either from the point of view of suspension spectra or of pairs $(X,n)$ of a CW complex and an integer, the symmetric monoidal structure is given by smash product $\wedge$ and the tensor unit is the sphere spectrum or $(S^0,0)$ depending on the perspective.

The category $\mathrm{ho}(\mathsf{Sp}^\omega)$ is rather complicated and displays very non-noetherian behaviour. It is generated by the tensor unit, but has infinite Rouquier dimension (i.e.\ arbitrarily many cones are needed to get all objects starting with the sphere spectrum), and understanding morphisms in this category is an extremely difficult open problem.

However, as mentioned above the space $\Spc \mathrm{ho}(\mathsf{Sp}^\omega)$ has been computed and has a rather rich structure. It can be illustrated as follows where an upward line indicates specialization i.e.\ higher points are in the closure of the lower points:

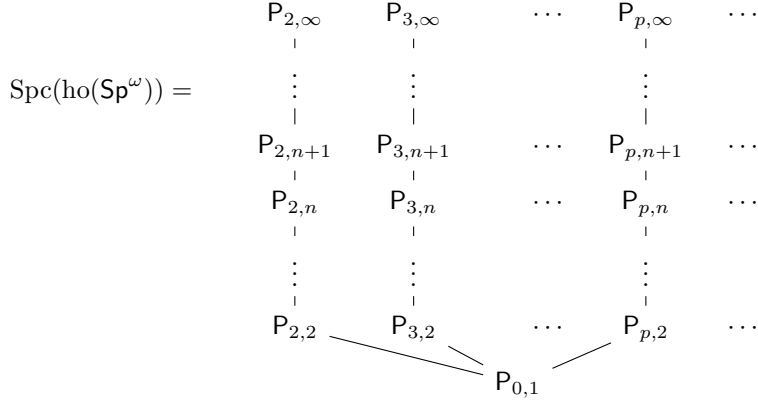
\[
\begin{tikzcd}[column sep=0.8em, row sep=0.4em]
&& \mathsf{P}_{2,\infty} \arrow[d, no head] & \mathsf{P}_{3,\infty} \arrow[d, no head] && \hspace{-2em}\cdots & \mathsf{P}_{p,\infty} \arrow[d, no head] & \cdots \\
\mathrm{Spc}(\mathrm{ho}(\mathsf{Sp}^\omega)) = && \vdots \arrow[d, no head] & \vdots \arrow[d, no head] &&& \vdots \arrow[d, no head] \\
&& \mathsf{P}_{2,n+1} \arrow[d, no head] & \mathsf{P}_{3,n+1} \arrow[d, no head] && \hspace{-2em}\cdots & \mathsf{P}_{p,n+1} \arrow[d, no head] & \cdots \\
&& \mathsf{P}_{2,n} \arrow[d, no head] & \mathsf{P}_{3,n} \arrow[d, no head] && \hspace{-2em}\cdots & \mathsf{P}_{p,n} \arrow[d, no head] & \cdots \\
&& \vdots \arrow[d, no head] & \vdots \arrow[d, no head] &&& \vdots \arrow[d, no head] \\
&& \mathsf{P}_{2,2} \arrow[rrd, no head] & \mathsf{P}_{3,2} \arrow[rd, no head] && \hspace{-2em}\cdots & \mathsf{P}_{p,2} \arrow[lld, no head] & \cdots \\
&&&& \mathsf{P}_{0,1} &&&
\end{tikzcd}
%\xymatrix@C=.8em @R=.4em{
%&&\mathsf{P}_{2,\infty} \ar@{-}[d]
%&\mathsf{P}_{3,\infty} \ar@{-}[d]
%&& \kern-2em{\cdots}
%&\mathsf{P}_{p,\infty} \ar@{-}[d]
%& {\cdots}
%\\
%\Spc(\mathrm{ho}(\mathsf{Sp}^\omega))=
%&&{\vdots} \ar@{-}[d]
%& {\vdots} \ar@{-}[d]
%&&& {\vdots} \ar@{-}[d]
%\\
%&&\mathsf{P}_{2,n+1} \ar@{-}[d]
%& \mathsf{P}_{3,n+1} \ar@{-}[d]
%&& \kern-2em{\cdots}
%& \mathsf{P}_{p,n+1} \ar@{-}[d]
%& {\cdots}
%\\
%&&\mathsf{P}_{2,n} \ar@{-}[d]
%& \mathsf{P}_{3,n} \ar@{-}[d]
%&& \kern-2em{\cdots}
%& \mathsf{P}_{p,n} \ar@{-}[d]
%& {\cdots}
%\\
%&&{\vdots} \ar@{-}[d]
%& {\vdots} \ar@{-}[d]
%&&& {\vdots} \ar@{-}[d]
%\\
%&&\mathsf{P}_{2,2} \ar@{-}[rrd]
%& \mathsf{P}_{3,2} \ar@{-}[rd]
%&& \kern-2em{\cdots}
%& \mathsf{P}_{p,2} \ar@{-}[lld]
%& {\cdots}
%\\
%&&&& \mathsf{P}_{0,1}
%}
\]

Let us explain, very briefly, the notation without attempting to go into the relevant constructions from stable homotopy theory.

\begin{enumerate}
\item The prime tensor-ideal $\mathsf{P}_{0,1}$ is the subcategory of spectra with torsion stable homotopy groups. This can be viewed as the kernel of the rational homology functor. It is the generic point of~$\Spc(\mathrm{ho}(\mathsf{Sp}^\omega))$.
\item For each prime integer~$p$, the prime ideal $\mathsf{P}_{p,\infty}$ is the kernel of the localization at $p$ functor $\mathrm{ho}(\mathsf{Sp}^\omega)\to \mathrm{ho}(\mathsf{Sp}^\omega)_{(p)}$. In other words, $\mathsf{P}_{p,\infty}$ consists of those objects whose stable homotopy groups are annihilated by some integer $n$ not divisible by $p$. These $\mathsf{P}_{p,\infty}$ are exactly the closed points of~$\Spc(\mathrm{ho}(\mathsf{Sp}^\omega))$.
\item For each prime integer~$p$ and each integer~$2\le n<\infty$, the prime ideal~$\mathsf{P}_{p,n}$ is the kernel of localization at $p$ followed by smashing with the $(n-1)^\textrm{st}$ Morava $K$-theory~$K_{p,n-1}$.
\item The support of an object~$A$ is empty when $A=0$, it is everything when $A$ is not torsion, and for a non-zero torsion spectrum it is a finite union of `columns'
    \[
    \overline{\mathsf{P}_{p,m_p}}= \{\mathsf{P}_{p,n} \mid m_p\leq n\leq\infty \}
    \]
    for integers $2\leq m_p<\infty$ depending on the prime~$p$.
\item The Thomason subsets of~$\Spc(\mathrm{ho}(\mathsf{Sp}^\omega))$ are the empty set, the whole space, and arbitrary unions of columns $\overline{\mathsf{P}_{p,m_p}}$ with $m_p\ge2$ as above.
\end{enumerate}

We see that the spectrum has infinite Krull dimension and is irreducible.

There has since been a lot of work on understanding other homotopical examples. Notably, a lot of work has been invested in investigating equivariant stable homotopy categories for various classes of groups and the perfect complexes over various ring spectra.

%----------------------------------------------------------------------------------------------------------------------------------------------

%----------------------------------------------------------------------------------------------------------------------------------------------

\subsubsection{Modular representation theory}

We start with the work of \cite{Benson/Carlson/Rickard:1997a} which concerns the stable module category. Let $G$ be a finite group and let $k$ be a field (the characteristic of which divides the order of $G$ if we want anything interesting to happen). We form the group algebra $kG$ and consider $\modu kG$ the category of finite dimensional $kG$-modules, or equivalently, the category of finite dimensional representations of $G$. The algebra $kG$ is a cocommutative Hopf algebra, with comultiplication determined by $g\mapsto g\otimes g$ for every group element $g$ and so this fits into Example~\ref{Gex:ttcats} (\ref{Gitem:hopf}). Thus $\modu kG$ is a symmetric monoidal Frobenius category and $\underline{\modu}\: kG$ is a rigid tt-category with unit the trivial representation $k$.

It is not always the case that $k$ generates $\underline{\modu}\: kG$ and this only happens if $G$ is $p$-nilpotent where $p$ is the characteristic of $k$. This is the case if $G$ is a $p$-group. Thus in general we need to consider tensor-ideals and a classification of thick subcategories seems, in general, out of reach outside of the $p$-nilpotent case.

We can also consider the slightly larger subcategory $\mathsf{D}^\mathrm{b}(\modu kG)$ which is also a rigid tt-category with the tensor product over $k$ (no need to derive as it is exact) and unit $k$. We have a localization sequence
\[
\mathsf{D}^\mathrm{perf}(kG) \to \mathsf{D}^\mathrm{b}(\modu kG) \to \underline{\modu}\: kG
\]
where the perfect complexes are a tensor-ideal (by Frobenius reciprocity also known as the projection formula in general) and so the localization functor is monoidal.

We can consider the cohomology ring of $G$
\[
H^\ast(G,k) = \Ext^\ast_{kG}(k,k)
\]
which is the non-negative part of the graded endomorphism ring of the tensor unit of $\underline{\modu}\: kG$ and the graded endomorphism ring of $k$ in the bounded derived category. This is a graded commutative ring, and so it makes sense to consider $\Spec^\mathrm{h} H^\ast(G,k)$ the set of homogeneous prime ideals of this graded ring (i.e.\ prime ideals generated by homogeneous elements) with the Zariski topology; this is just the Stone dual of the coherent frame of radical homogeneous ideals of the cohomology ring. Inside of $\Spec^\mathrm{h} H^\ast(G,k)$ we can consider the open subspace
\[
\Proj H^\ast(G,k) = \Spec^\mathrm{h} H^\ast(G,k) \setminus \{H^{\geq 1}(G,k)\}.
\]

The theorem of Benson, Carlson, and Rickard is that 
\[
\Spc \underline{\modu}\: kG \cong \Proj H^\ast(G,k) \quad \text{ and } \quad \Spc \mathsf{D}^\mathrm{b}(\modu kG) \cong \Spec^\mathrm{h} H^\ast(G,k).
\]
This is closely related to a lot of earlier work, for instance by Dade and Quillen, and sheds a lot of light on the representation theory of $G$. For instance, one can read off the rate of growth of the minimal projective resolution of a $kG$-module from the corresponding closed subset of $\Proj H^\ast(G,k)$ and Carlson proved \cite{Carlson:1984a} that an indecomposable module has connected support.

There are many other rigid tt-categories that one can produce from a finite group. For instance, Balmer and Gallauer have studied the bounded homotopy category of permutation modules in \cite{balmer2023geometry}, which turns out to be an interesting and complicated example.

%----------------------------------------------------------------------------------------------------------------------------------------------

%----------------------------------------------------------------------------------------------------------------------------------------------

\subsubsection{The rogues' gallery}

All of the examples discussed above are rigid. Thus one gets a classification of all tensor-ideals: being radical is for free. Let us start by mentioning a few non-rigid examples where one observes some behaviour which cannot occur in the rigid case.

Some work has been done on non-rigid analogues of the stable module categories discussed above. Instead of working over a field $k$ one considers representations over a more general ring, for instance $\mathbb{Z}/p^n\mathbb{Z}$. One still obtains a Hopf algebra but the tensor product over $\mathbb{Z}/p^n\mathbb{Z}$ is not exact and so one does not obtain a tt-category by taking the stable category. In order to rectify this one can adjust the exact structure by only considering a sequence to be exact if it is split when restricted to $\mathbb{Z}/p^n\mathbb{Z}$, see \cite{Benson/Iyengar/Krause:2013a} for details. One thus gets a non-rigid tt-category $\underline{\modu_\mathrm{rel}}\:\mathbb{Z}/p^n\mathbb{Z}G$ with unit object $\mathbb{Z}/p^n\mathbb{Z}$ with the trivial action. The spectrum of this category was computed in \cite{BalandChirvasituStevenson}
\[
\Spc \underline{\modu_\mathrm{rel}}\:\mathbb{Z}/p^n\mathbb{Z}G = \coprod_{i=1}^n \Spc \underline{\modu}\:\mathbb{Z}/p\mathbb{Z}G.
\]
Thus, even though the unit object $\mathbb{Z}/p^n\mathbb{Z}$ is indecomposable the spectrum is not connected. This cannot happen in the rigid case. In fact, the decomposition above corresponds to a semi-orthogonal decomposition into $n$ tensor-ideals which are all monoidally equivalent to $\underline{\modu}\:\mathbb{Z}/p\mathbb{Z}G$.

This behaviour seems somewhat common in the non-rigid setting. Aoki has computed the spectrum in the very general setting of compact objects in the category of sheaves on a coherent space with values in a well-behaved symmetric monoidal stable $\infty$-category \cite{Aoki:2023}. In this setting one frequently gets semi-orthogonal decompositions and the spectrum will typically be highly disconnected (see for instance \cite{Gallauer-constructible} for a prototypical example).

There are many natural examples of tt-categories which are not rigid, as discussed briefly in Example~\ref{Gex:ttcats} and it could be interesting to get a better understanding of how these behave and what one can expect. One aspect which has not really been so well explored is the behaviour of tt-categories with nilpotent objects. In the above examples tensor-nilpotent objects do not occur! The only examples I could remember while writing this which have nilpotent objects and where one understands the spectrum are very artificial (e.g.\ \cite[Section~5]{Stevenson:bool}).

\begin{xca}
Look at some article on tt-geometry, e.g.\ one of the ones mentioned above, and see what you recognise.
\end{xca}

%----------------------------------------------------------------------------------------------------------------------------------------------

%----------------------------------------------------------------------------------------------------------------------------------------------

%----------------------------------------------------------------------------------------------------------------------------------------------

%----------------------------------------------------------------------------------------------------------------------------------------------

%----------------------------------------------------------------------------------------------------------------------------------------------

\subsection{A tt-proof of Thomason's theorem}\label{Gsec:thomasonproof}

In this section we give a complete proof of Thomason's theorem using some techniques from tt-geometry. This serves to give an insight into how one proves such results, the techniques of tt-geometry, and some results and strategies we didn't discuss yet. This section does not aim to be self-contained and we invoke some  results from tt-geometry, and the theory of triangulated categories more generally, which were not discussed above.

We already mentioned in Remark~\ref{Grem:rho} and when discussing classifications in algebraic geometry the comparison map $\rho$. In fact, this map is defined in complete generality.

\begin{theorem}[\cite{Balmer:2010b}]\label{Gthm:rho}
Let $\mathsf{K}$ be a tt-category and set $R_\mathsf{K} = \mathsf{K}(\unit, \unit)$. There is a continuous comparison map
\[
\rho\colon \Spc \mathsf{K} \to \Spec R_\mathsf{K} \quad \text{ given by } \rho(\mathsf{P}) = \{f\in R_\mathsf{K} \mid \mathrm{cone}(f)\notin \mathsf{P}\}.
\]
It is natural in $\mathsf{K}$ and is surjective if $R_\mathsf{K}$ is a coherent ring or if $\mathsf{K}$ is \emph{connective}, which means that
\[
\Hom_\mathsf{K}(\unit, \Sigma^i \unit) = 0 \text{ for all } i\geq 1.
\]
\end{theorem}
\begin{proof}
This is a combination of 5.6, 7.3, and 7.13 of Balmer's paper.
\end{proof}

One can also say something precise about when $\rho$ is injective, although the condition is much harder to check on account of also being necessary. (This result appeared in \cite[Propositon~3.11]{Dellambrogio2014}.)

\begin{lemma}\label{lem_inj_cri}
Let $\mathsf{K}$ be a tensor triangulated category, and let 
\begin{displaymath}
\rho = \rho_{\mathsf{K}} \colon \Spc{\mathsf{K}} \to \Spec{R_{\mathsf{K}}}
\end{displaymath}
denote Balmer's comparison map. Then $\rho$ is injective if
\begin{displaymath}
\mathcal{B} = \{\supp{\mathrm{cone}(\unit \stackrel{r}{\longrightarrow} \unit)} \; \vert \; r\in R_{\mathsf{K}} \}
\end{displaymath}
is a basis of closed sets for the topology on $\Spc \mathsf{K}$. 
\end{lemma}
\begin{proof}
Suppose that $\mathcal{B}$ is a basis. We set $K(r) = \mathrm{cone}(\unit \stackrel{r}{\longrightarrow} \unit)$. Let $\mathsf{P}, \mathsf{Q} \in \Spc{\mathsf{K}}$ be such that $\rho(\mathsf{P}) = \rho(\mathsf{Q})$ i.e., 
$K(r)\notin \mathsf{P}$ if and only if $K(r) \notin \mathsf{Q}$ for every $r\in R_{\mathsf{K}}$. Using our basis $\mathcal{B}$ we see that
\begin{displaymath}
\overline{\{\mathsf{P}\}} = \bigcap_{\substack{r\in R_{\mathsf{K}} \\ K(r)\notin \mathsf{P}}} \supp{K(r)} = \bigcap_{\substack{r\in R_{\mathsf{K}} \\ K(r)\notin \mathsf{Q}}} \supp{K(r)} = \overline{\{\mathsf{Q}\}}
\end{displaymath}
where the middle equality follows from $\rho(\mathsf{P}) = \rho(\mathsf{Q})$ and the closure can be computed this way by \cite[Proposition~2.8]{Balmer:2005a} (or a small exercise). But then $\mathsf{P} = \mathsf{Q}$ since $\Spc{\mathsf{K}}$ is coherent and hence $T_0$.
\end{proof}

\begin{remark}
Motivated by commutative algebra the object 
\[
K(r) = \mathrm{cone}(\unit \stackrel{r}{\longrightarrow} \unit)
\]
for $r\in R_\mathsf{K}$ is sometimes called the Koszul complex on $r$.
\end{remark}

We want to prove that for any quasi-compact and quasi-separated scheme $X$ there is a canonical isomorphism of schemes
\begin{displaymath}
(\Spc \mathsf{D}^{\mathrm{perf}}(X), \mathcal{O}_{\mathsf{D}^{\mathrm{perf}}(X)}) \stackrel{\sim}{\to} X.
\end{displaymath}
We briefly discussed the locally ringed structure in Remark~\ref{Grem:lrs}, and we do not say much more about it here. In fact, provided $\rho$ is a homeomorphism it is automatically an isomorphism of locally ringed spaces and so one can safely choose to ignore this extra structure (safe in the knowledge it has come along for the ride if or when one chooses to engage with it).

We will reference some results from Thomason's paper, but we do not use any results whose proofs are difficult; the point is that we are able to prove the theorem without use of any tensor nilpotence theorems as they are hidden in the machine. By the construction of the spectrum we will then deduce Thomason's original result:

\begin{theorem}[Thomason]\label{Thomason:1997a}
Let $X$ be a quasi-compact and quasi-separated scheme. There is an order-preserving bijection (in fact an isomorphism of coherent frames) between $\Thick^{\sqrt{\otimes}}(\mathsf{D}^\mathrm{perf}(X))$ and $\Thom(X)$.
\end{theorem}

We prove the claimed isomorphism of locally ringed spaces exists in three steps. Contrary to common algebro-geometric practice we begin with affine noetherian schemes and build up to arbitrary schemes rather than starting with the general case and performing reductions. Before getting underway with the affine case we recall the notion of sheaf-theoretic support and use it to construct a morphism $X \to \Spc \mathsf{D}^{\mathrm{perf}}(X)$ via the universal property.

\begin{definition}\label{Gdef:supph}
Let $X$ be a quasi-compact and quasi-separated scheme and let $E\in \mathsf{D}^{\mathrm{perf}}(X)$. The \emph{sheaf-theoretic support} of $E$ is
\[
\supp_X E = \{x\in X \mid E_x\neq 0\} = \cup_i \supp_X H^i(E).
\]
Here $\supp_X H^i(E)$ is the usual support of a quasi-coherent sheaf on $X$. We use the subscript $X$ to distinguish it from the universal support. 
\end{definition}

The sheaf-theoretic support is easily checked to be a support datum in the sense of Definition~\ref{Gdef:supportdatum}. Thus, by the universal property of $\Spc$, there is a continuous map $\sigma \colon X \to \Spc \mathsf{D}^{\mathrm{perf}}(X)$ such that $\sigma^{-1}\supp E = \supp_X E$. It can be given explicitly by
\[
\sigma(x) = \{E\in \mathsf{D}^{\mathrm{perf}}(X) \mid E_x\cong 0\}.
\]
One should compare to the discussion of Section~\ref{Gsec:realexample} where we explicitly constructed precisely these primes in the case of the affine line.

We will make use of the unbounded derived category in this section. We denote by $\mathsf{D}(R)$ the unbounded derived category of $R$-modules and by $\mathsf{D}(X)$ the unbounded derived category of $\mathcal{O}_X$-modules with quasi-coherent cohomology on a quasi-compact and quasi-separated scheme $X$ (when $X$ is affine these are the same thing). We recall that $\mathsf{D}^{\mathrm{perf}}(X)$ is the full subcategory of compact objects in $\mathsf{D}(X)$. (We have not discussed compactly generated triangulated categories yet, but some reminders can be found at the start of Chapter~\ref{Gchap:cg}.) The unbounded derived category is a triangulated category with arbitrary coproducts. Given a collection of objects $\mathcal{S}$ we denote by $\loc(\mathcal{S})$ the smallest localizing, i.e.\ triangulated and coproduct closed, subcategory of $\mathsf{D}(X)$. 

\begin{lemma}
Let $R$ be a commutative ring. Then the map $\sigma$ corresponding to the sheaf-theoretic support satisfies $\rho\sigma = 1_{\Spec R}$ on the level of topological spaces. In particular, if $\rho$ is a bijection then it is a homeomorphism.
\end{lemma}
\begin{proof}
This is an explicit computation. Given $\mathfrak{p}\in \Spec R$ we have
\begin{displaymath}
\sigma(\mathfrak{p}) = \{E \in \mathsf{D}^\mathrm{perf}(R) \mid E_\mathfrak{p} \cong 0\}
\end{displaymath}
and to compute $\rho\sigma(\mathfrak{p})$ we need to identify precisely which Koszul complexes $K(r) = \mathrm{cone}(R \stackrel{r}{\to} R)$ are contained in the prime tensor-ideal $\sigma(\mathfrak{p})$. Now $K(r)_\mathfrak{p} = 0$ precisely if $r$ becomes an isomorphism when localizing at $\mathfrak{p}$. This is so exactly when $r\notin \mathfrak{p}$ and so
\[
\rho\sigma(\mathfrak{p}) = \{r\in R \mid K(r) \notin \sigma(\mathfrak{p})\} = \{r\in R\mid K(r)_\mathfrak{p}\neq 0\} = \mathfrak{p}
\]
as claimed.
\end{proof}

\begin{proposition}
Let $R$ be a noetherian ring. Then the comparison map 
\[
\rho \colon \Spc \mathsf{D}^\mathrm{perf}(R) \to \Spec R
\]
is an isomorphism of locally ringed spaces.
\end{proposition}
\begin{proof}
Throughout we will identify the endomorphism ring of the tensor unit with $R$. Since $\mathsf{D}^{\mathrm{perf}}(R)$ is connective, surjectivity of $\rho$ follows from Theorem~\ref{Gthm:rho}. Thus to show that $\rho$ is an isomorphism it is sufficient to show injectivity as noted above. We do this by checking the criterion given in Lemma~\ref{lem_inj_cri}. Thus we are reduced to showing that the supports (in the universal sense) of Koszul objects form a basis. Since we already know that the sets $\supp E$ for $E$ an object of $\mathsf{D}^{\mathrm{perf}}(R)$ form a basis it is sufficient to check that for any such $E$ and any $\mathsf{P} \in \Spc \mathsf{D}^{\mathrm{perf}}(R)$ with $\mathsf{P}\notin \supp E$ there is an $r\in R$ such that $\supp E \subseteq \supp K(r)$ and $\mathsf{P}\notin \supp K(r)$.

We know that the $\supp_R K(r) = V((r))$ form a basis for $\Spec R$ and that 
\begin{align*}
\rho^{-1}\supp_R K(r) &= \{\mathsf{P} \in \Spc \mathsf{D}^{\mathrm{perf}}(R) \mid \rho(\mathsf{P}) \in \supp_R K(r) \} \\
&= \{\mathsf{P} \in \Spc \mathsf{D}^{\mathrm{perf}}(R) \mid r\in \rho(\mathsf{P})\} \\
&= \{\mathsf{P} \in \Spc \mathsf{D}^{\mathrm{perf}}(R) \mid K(r)\notin \mathsf{P} \} \\
&=  \supp K(r).
\end{align*}
So provided we can show that $\rho^{-1}\supp_R E = \supp E$ it will follow that the supports of Koszul complexes form a basis and we will be done. This will follow relatively easily from the following lemma which relies on passing to the unbounded derived category.

\begin{lemma}
Every non-zero prime tensor-ideal in $\mathsf{D}^{\mathrm{perf}}(R)$ contains a non-zero Koszul complex, i.e.\ an object of the form
\[
K(r_1,\ldots,r_n) = K(r_1)\otimes \cdots K(r_n).
\]
\end{lemma}
\begin{proof}
Recall that $\mathsf{D}(R)$ denotes the unbounded derived category of all $R$-modules. Let $E \in \mathsf{D}(R)$ be a non-zero perfect complex and consider $\loc(E)$ the localizing subcategory generated by $E$. Because $E$ is non-zero the subset $\supp_R E$ of $\Spec R$ is non-empty and because $E$ is perfect it is closed. Pick a closed point $\mathfrak{m}\in \supp_R E$ with corresponding residue field $k = R/\mathfrak{m}$. Then $E \otimes^\mathrm{L} k$ is in $\loc(E)$ (by the analogue of Lemma~\ref{Glem:small_unit_gen}) and is a direct sum of suspensions of $k$ (cf.\ Remark~\ref{Grem:field}). By Nakayama's lemma $E \otimes^\mathrm{L} k$ is not zero and so $\loc(k) \subseteq \loc(E)$.

Using the standard $t$-structure on $\mathsf{D}(R)$ we see that $\loc(k)$ contains all bounded complexes with finite length homology supported at $\mathfrak{m}$. Since $R$ is noetherian we can take a generating set for $\mathfrak{m}$ say $(m_1,\ldots,m_n)$. The corresponding Koszul complex $K(\mathfrak{m}) = K(m_1,\ldots,m_n)$ has finite length homology and so is contained in $\loc(k)$. Moreover, $K(\mathfrak{m})$ is a perfect complex by construction. Because $E$ is compact we can apply \cite[Theorem~4.4.9]{Neeman:1992a} which asserts that 
\[
\loc(E) \cap \mathsf{D}^{\mathrm{perf}}(R) = \thick(E).
\]
Altogether, we have shown that $\thick(E)$ contains $K(\mathfrak{m})$ for every closed point $\mathfrak{m}$ in its sheaf-theoretic support. In particular, if $\mathsf{P}$ is a non-zero prime ideal it contains such an $E$ and hence a Koszul complex.
\end{proof}

Now we complete the proof of the proposition by showing $\rho^{-1}\supp_R E = \supp E$ for every perfect complex $E$. We will denote by $\mathsf{Q}$ a prime tensor-ideal and set $\mathfrak{q} = \rho\mathsf{Q}$. Associated to $\mathsf{Q}$ we have the tensor-ideal
\[
\mathsf{J} = \thick( K(s) \mid K(s)\in \mathsf{Q}) \subseteq \mathsf{Q}.
\]
By the theory of central localization (see \cite[Theorem~3.6]{Balmer:2010b}) we have
\[
\mathsf{D}^{\mathrm{perf}}(R)/\mathsf{J} \cong \mathsf{D}^{\mathrm{perf}}(R)_\mathfrak{q}
\]
where the right-hand side is obtained by tensoring all homomorphism modules in $\mathsf{D}^{\mathrm{perf}}(R)$ with $R_\mathfrak{q}$. This is again a tt-category, since it is the quotient by a tensor-ideal. (In fact, one can check that, up to idempotent completion, this gives $\mathsf{D}^{\mathrm{perf}}(R_\mathfrak{q})$ but we will not need this fact). This quotient and the description via localization at $\mathfrak{q}$ is the main tool we need in addition to the lemma we just proved.

We first show that for $E \in \mathsf{D}^{\mathrm{perf}}(R)$ we have the containment $\supp E \subseteq \rho^{-1}\supp_R E$. Indeed for any $\mathsf{Q} \in \supp E$ we have $\Hom_{\mathsf{D}^{\mathrm{perf}}(R)/\mathsf{Q}}(E,E) \neq 0$. With notation as above we have $\mathsf{D}^{\mathrm{perf}}(R)/\mathsf{Q} \cong (\mathsf{D}^{\mathrm{perf}}(R)/\mathsf{J})/(\mathsf{Q}/\mathsf{J})$ and hence
\begin{displaymath}
0 \neq \Hom_{\mathsf{D}^{\mathrm{perf}}(R)/\mathsf{J}}(E,E) \cong \Hom_{\mathsf{D}^{\mathrm{perf}}(R)}(E,E)_\mathfrak{q} \cong \Hom_{\mathsf{D}^{\mathrm{perf}}(R_\mathfrak{q})}(E_\mathfrak{q}, E_\mathfrak{q})
\end{displaymath}
since $E$ survives under the further localization killing $\mathsf{Q}$. Thus $E_\mathfrak{q}\neq 0$ or equivalently $\mathfrak{q}\in \supp_R E$ and so $\mathsf{Q} \in \rho^{-1}\supp_R E$ as claimed.

Let us now check the other containment. Suppose then that $E$ is a perfect complex and that $\mathsf{Q}\notin \supp E$ but $\mathsf{Q} \in \rho^{-1}\supp_R E$. 
Consider the functor
\begin{displaymath}
\begin{tikzcd}
\mathsf{D}^{\mathrm{perf}}(R) \arrow[r, "\phi"] & \mathsf{D}^{\mathrm{perf}}(R_\mathfrak{q})
\end{tikzcd}
%\xymatrix{
%\mathsf{D}^{\mathrm{perf}}(R) \ar[r]^{\phi} & \mathsf{D}^{\mathrm{perf}}(R_\mathfrak{q})
%}
\end{displaymath}
induced by localization, and the functor
\begin{displaymath}
\begin{tikzcd}
\mathsf{D}^{\mathrm{perf}}(R) \arrow[r, "\pi", pos=0.3] & \mathsf{D}^{\mathrm{perf}}(R)/\mathsf{J} \cong \mathsf{D}^{\mathrm{perf}}(R)_\mathfrak{q}
\end{tikzcd}
%\xymatrix{
%\mathsf{D}^{\mathrm{perf}}(R) \ar[r]^(0.3){\pi} & \mathsf{D}^{\mathrm{perf}}(R)/\mathsf{J} \cong \mathsf{D}^{\mathrm{perf}}(R)_\mathfrak{q}
%}
\end{displaymath}
described above. We note that both of the target categories have $R_\mathfrak{q}$ as the endomorphism ring of their respective tensor units. Now $\mathsf{Q} \in \rho^{-1}\supp_R E$ says that $\phi E \neq 0$ in $\mathsf{D}^{\mathrm{perf}}(R_\mathfrak{q})$. We can see that $\mathsf{J} = \ker\pi \subseteq \ker\phi$ as the objects $K(r)$ generating $\mathsf{J}$ become zero in $\mathsf{D}^{\mathrm{perf}}(R_\mathfrak{q})$. Indeed, if $K(r) \in \mathsf{Q}$ then $r\notin \mathfrak{q}$ and so $r$ is inverted in $R_\mathfrak{q}$ and its cone $\phi(K(r))$ becomes zero. The universal property of localization thus gives a commutative triangle
\begin{displaymath}
\begin{tikzcd}
\mathsf{D}^{\mathrm{perf}}(R) \arrow[r, "\pi"] \arrow[dr, "\phi"'] & \mathsf{D}^{\mathrm{perf}}(R)_\mathfrak{q} \arrow[d, "\phi'"] \\
& \mathsf{D}^{\mathrm{perf}}(R_\mathfrak{q})
\end{tikzcd}
%\xymatrix{
%\mathsf{D}^{\mathrm{perf}}(R) \ar[r]^\pi \ar[dr]_\phi & \mathsf{D}^{\mathrm{perf}}(R)_\mathfrak{q} \ar[d]^{\phi'} \\
%& \mathsf{D}^{\mathrm{perf}}(R_\mathfrak{q})
%} 
\end{displaymath}
and hence $\pi(E) \neq 0$. Now $\pi(\mathsf{Q}) = \mathsf{Q}/\mathsf{J} = \mathsf{Q}_\mathfrak{q} \in \Spc \mathsf{D}^{\mathrm{perf}}(R)_\mathfrak{q}$ is a prime ideal and is non-zero as it contains the image of $E$ which we just saw is non-zero. By hypothesis it contains no Koszul complexes. We claim that $\ker\phi \subseteq \mathsf{Q}$. Indeed,
\begin{displaymath}
\ker\phi = \{F\in \mathsf{D}^{\mathrm{perf}}(R) \mid F \otimes R_\mathfrak{q} = 0 \} = \{F \in \mathsf{D}^{\mathrm{perf}}(R) \mid \mathfrak{q} \notin \supp_R F\}
\end{displaymath}
so that if $F \in \ker\phi$ we can use the containment for supports we have already proved and the second description above to see that there is some $s\notin \mathfrak{q}$ such that $\supp F \subseteq \supp K(s)$ so that $F \in \thick(K(s)) \subseteq \mathsf{Q}$ (remember we know that the $V((r))$ are a basis for $\Spec R$ and their preimages under $\rho$ are the subsets $\supp K(r)$). It follows that $\phi\mathsf{Q} = \phi'\mathsf{Q}_\mathfrak{q}$ is a prime ideal in $\mathsf{D}^{\mathrm{perf}}(R_\mathfrak{q})$ and is non-zero since it contains $\phi E$. But by hypothesis it contains no non-zero Koszul complexes contradicting the lemma above.
\end{proof}

\begin{remark}
In fact, as follows from the classification, $\mathsf{Q} = \mathsf{J}$ and the quotient is $\mathsf{D}^{\mathrm{perf}}(R_\mathfrak{q})$ up to idempotent completion.
\end{remark}

It is now easy to extend the result to non-noetherian rings as well.

\begin{proposition}
Suppose that $R$ is a ring (not necessarily noetherian) then $\rho$ is an isomorphism.
\end{proposition}
\begin{proof}
We did not need to use the noetherian hypothesis above until we proved the key lemma that prime ideals contain Koszul complexes. Thus as above it is sufficient to show that $\rho^{-1}\supp_R = \supp$. We can write $R$ as
\begin{displaymath}
R = \colim R_\alpha
\end{displaymath}
where the colimit is taken over the subrings $R_\alpha$ of $R$ which are of finite type over $\mathbb{Z}$. In particular, each $R_\alpha$ is noetherian.

Now let $E$ be a perfect complex over $R$, which we may assume is in fact a bounded complex of finitely generated projective $R$-modules on the nose. As in \cite[3.6.4]{Thomason:1997a} we can find some $\beta$ and some $E_\beta \in \mathsf{D}^{\mathrm{perf}}(R_\beta)$, again a complex of finitely generated projectives on the nose, so that
\begin{displaymath}
E \cong R \otimes_{R_\beta} E_\beta.
\end{displaymath}
In other words the derived functor 
\begin{displaymath}
\begin{tikzcd}
\mathsf{D}^{\mathrm{perf}}(R_\beta) \arrow[r, "F"] & \mathsf{D}^{\mathrm{perf}}(R)
\end{tikzcd}
%\xymatrix{
%\mathsf{D}^{\mathrm{perf}}(R_\beta) \ar[r]^F & \mathsf{D}^{\mathrm{perf}}(R)
%}
\end{displaymath}
of base change along $f\colon R_\beta \to R$ satisfies $FE_\beta = E$. By naturality of $\rho$ this functor fits into a commutative square
\begin{displaymath}
\begin{tikzcd}
\mathrm{Spc} \, \mathsf{D}^{\mathrm{perf}}(R) \arrow[r, "\mathrm{Spc} \, F"] \arrow[d, "\rho"'] & \mathrm{Spc} \, \mathsf{D}^{\mathrm{perf}}(R_\beta) \arrow[d, "\rho_\beta"] \\
\mathrm{Spec} \, R \arrow[r, "\mathrm{Spec} \, f"'] & \mathrm{Spec} \, R_\beta
\end{tikzcd}
%\xymatrix{
%\Spc \mathsf{D}^{\mathrm{perf}}(R) \ar[r]^{\Spc F} \ar[d]_{\rho} & \Spc \mathsf{D}^{\mathrm{perf}}(R_\beta) \ar[d]^{\rho_\beta} \\
%\Spec R \ar[r]_{\Spec f} & \Spec R_\beta
%}
\end{displaymath}
Thus
\begin{align*}
\rho^{-1}\supp_R E &= \rho^{-1}(\Spec f)^{-1}\supp_{R_\beta} E_\beta \\
&= (\Spc F)^{-1} \rho_\beta^{-1} \supp_{R_\beta} E_\beta \\
&= \supp FE_\beta \\
&= \supp E
\end{align*}
where $\rho_\beta^{-1}\supp_{R_\beta} E_\beta = \supp E_\beta$ by the noetherian case and the formulas relating the preimage of the support and the supports of $FE_\beta$ and $\supp_R E$ and $\supp_{R_\beta} E_\beta$ are just Stone duality (or alternatively the former is \cite[Proposition~3.6]{Balmer:2005a} and the latter a small exercise in commutative algebra).
\end{proof}

\begin{remark}
The point above is that a perfect complex $E$ over $R$ is essentially given by specifying finitely many matrices over $R$, i.e.\ giving a finite set of elements of $R$. These elements always lie in some finitely generated subalgebra of $R$, e.g.\ the one obtained by adjoining the required elements to the image of $\mathbb{Z}$ in $R$, and so we can realize $E$ over this subalgebra. Understanding the spectrum only ever involves dealing with finitely many perfect complexes at once and so we should be able to reduce to the noetherian case.
\end{remark}

Finally we are ready to treat the case of $X$ a quasi-compact and quasi-separated scheme. 

\begin{theorem}
Let $X$ be a quasi-compact and quasi-separated scheme. There is a natural homeomorphism
\[
\Spc \mathsf{D}^{\mathrm{perf}}(X) \cong X
\]
which gives an isomorphism of locally ringed spaces. 
\end{theorem}
\begin{proof}
We begin with some setup. The sheaf-theoretic support on $X$, as in Definition~\ref{Gdef:supph}, is a support datum valued in $X$. Hence as discussed following said definition, there exists a unique continuous map
\begin{displaymath}
\sigma \colon X\to \Spc \mathsf{D}^{\mathrm{perf}}(X)
\end{displaymath}
such that $\sigma^{-1}\supp E = \supp_X E$ for every object $E$ of $\mathsf{D}^{\mathrm{perf}}(X)$. We claim that $\sigma$ is in fact a homeomorphism which can be enhanced to an isomorphism of locally ringed spaces. Let $X = \cup_{i=1}^n U_i$, where $U_i \cong \Spec R_i$, be an open affine cover of $X$ which we can take to be finite by quasi-compactness. Denote by $Z_i$ the complement of $U_i$ in $X$. We let $\mathsf{D}_{Z_i}^{\mathrm{perf}}(X)$ be the full subcategory of perfect complexes whose sheaf-theoretic support is contained in $Z_i$. Then we have a composite
\begin{displaymath}
\begin{tikzcd}
\mathsf{D}^{\mathrm{perf}}(X) \arrow[r, two heads] & \mathsf{D}^{\mathrm{perf}}(X)/\mathsf{D}_{Z_i}^{\mathrm{perf}}(X) \arrow[r, hook] & (\mathsf{D}^{\mathrm{perf}}(X)/\mathsf{D}_{Z_i}^{\mathrm{perf}}(X))^\natural \cong \mathsf{D}^{\mathrm{perf}}(U_i)
\end{tikzcd}
%\xymatrix{
%\mathsf{D}^{\mathrm{perf}}(X) \ar@{->>}[r] & \mathsf{D}^{\mathrm{perf}}(X)/\mathsf{D}_{Z_i}^{\mathrm{perf}}(X)\; \ar@{^{(}->}[r] & (\mathsf{D}^{\mathrm{perf}}(X)/\mathsf{D}_{Z_i}^{\mathrm{perf}}(X))^\natural \cong \mathsf{D}^{\mathrm{perf}}(U_i)
%}
\end{displaymath}
which we denote by $\pi_i$, where the $\natural$ indicates taking the idempotent completion and the final equivalence of categories is proved in \cite{ThomasonTrobaugh:1990}. Consider the square
\begin{equation}\label{comm_square_1}
\begin{tikzcd}
\mathrm{Spc} \, \mathsf{D}^{\mathrm{perf}}(U_i) \arrow[rr, hook, "\mathrm{Spc} \, \pi_i"] && \mathrm{Spc} \, \mathsf{D}^{\mathrm{perf}}(X) \\
U_i \arrow[u, "\rho_i^{-1}", "\wr"'] \arrow[rr, hook, "j_i"'] && X \arrow[u, "\sigma"']
\end{tikzcd}
%\xymatrix{
%\Spc \mathsf{D}^{\mathrm{perf}}(U_i) \ar@{^{(}->}[rr]^{\Spc \pi_i} && \Spc \mathsf{D}^{\mathrm{perf}}(X) \\
%U_i \ar[u]^{{\rho_i^{-1}}}_{\wr} \ar@{^{(}->}[rr]^{j_i} && X \ar[u]_{\sigma}
%}
\end{equation}
We assert that this square is commutative. This is a consequence of Stone duality. The support datum $\supp_{X}(-) \cap U_i$ defines a unique continuous map $\sigma_i\colon U_i \to \Spc \mathsf{D}^{\mathrm{perf}}(X)$ such that $(\sigma j_i)^{-1} = \sigma_i^{-1} =  (\Spc \pi_i {\rho_i}^{-1})^{-1}$ (the latter equality since $\pi_i$ identifies with the pullback). In particular, it agrees with the support datum $\supp_{U_i} \pi_i(-)$ and so the two composites in the square are equal (cf.\ \cite[Lemma~3.3]{Balmer:2005a}). This allows us to work locally and we are now ready to show $\sigma$ is an isomorphism. 

First we prove that $\sigma$ is surjective. Using the commutative square above it is clearly sufficient to show that 
\begin{displaymath}
\Spc \mathsf{D}^{\mathrm{perf}}(X) = \bigcup_{i=1}^n \Spc \mathsf{D}^{\mathrm{perf}}(U_i)
\end{displaymath}
(where we have identified the spaces on the right with their homeomorphic images). Equivalently, we need to show each prime $\mathsf{P}$ is in the preimage of some $\pi_i$. This is the same as showing $\pi_i\mathsf{P}$ is prime in $\mathsf{D}^{\mathrm{perf}}(U_i)$ for some $i$ (where by a mild abuse of the notation we are writing $\pi_i\mathsf{P}$ for the idempotent completion of the prime ideal in the quotient) i.e.\ that $\ker\pi_i$ is contained in $\mathsf{P}$ for some $i$. So let $\mathsf{P}$ be a prime tensor-ideal. For each $i$ we can consider
\[
\ker \pi_i = \{E\in \mathsf{D}^{\mathrm{perf}}(X) \mid \supp_X E \subseteq Z_i\}
\]
which is a radical tensor-ideal. The intersection of these kernels is
\[
\bigcap_{i=1}^n \ker \pi_i = \{E\in \mathsf{D}^{\mathrm{perf}}(X) \mid \supp_X E \subseteq \cap_i Z_i\} = \{E\in \mathsf{D}^{\mathrm{perf}}(X) \mid \supp_X E = \varnothing\}  
\]
since the $U_i$ cover $X$ and hence no point is contained in all the $Z_i$. But $\supp_X E = \varnothing$ if and only if $E\cong 0$ and so this intersection is $0$. Because $\mathsf{P}$ is meet-prime in the lattice of radical tensor ideals and contains the intersection of the $\ker \pi_i$ it must contain some $\ker \pi_i$. Hence $\mathsf{P}$ is in the image of $\Spc \mathsf{D}^{\mathrm{perf}}(U_i)$.

We next prove that $\sigma$ is injective. Suppose that $x,y\in X$ are such that $\sigma(x) = \sigma(y)$. Then since each $\sigma j_i$ is injective by the commutativity of (\ref{comm_square_1}) we must have that $x$ and $y$ lie in distinct opens occurring in the cover. Suppose then that $\sigma(x) \in \Spc \mathsf{D}^{\mathrm{perf}}(U_i)$ and $\sigma(y) \in \Spc \mathsf{D}^{\mathrm{perf}}(U_j)$ where $i\neq j$ and we are identifying the spectra of these idempotent completed quotients with their images via the canonical homeomorphisms as above. In other words we have $\pi_i\sigma(x)$ and $\pi_j\sigma(y)$ prime in $\mathsf{D}^{\mathrm{perf}}(U_i)$ and $\mathsf{D}^{\mathrm{perf}}(U_j)$ respectively. This is equivalent to
\begin{displaymath}
\mathsf{D}_{Z_i}^{\mathrm{perf}}(X) \subseteq \sigma(x) \quad \text{ and } \quad \mathsf{D}_{Z_j}^{\mathrm{perf}}(X) \subseteq \sigma(y).
\end{displaymath}
Since $\sigma(x) = \sigma(y)$ we thus have $\mathsf{D}_{Z_j}^{\mathrm{perf}}(X) \subseteq \sigma(x)$ so that $\pi_j\sigma(x)$ is a prime tensor-ideal in $\mathsf{D}^{\mathrm{perf}}(U_j)$. But we know that ${\rho_j}$ is an isomorphism from which we deduce that $x\in U_j$ yielding a contradiction.

Thus we have proved that $\sigma$ is bijective; it is immediate from this that $\supp = \supp_X$ on $\mathsf{D}^{\mathrm{perf}}(X)$. It follows that $\sigma$ sends quasi-compact opens to quasi-compact opens (using for example \cite[Lemma~3.4]{Thomason:1997a} or by pondering more commutative diagrams like (\ref{comm_square_1})). Thus $\sigma$ is open and hence a homeomorphism.
\end{proof}

%----------------------------------------------------------------------------------------------------------------------------------------------

%----------------------------------------------------------------------------------------------------------------------------------------------

%-----------------------------------------------------------------------
% End of chapter
%-----------------------------------------------------------------------

\section{Categorified lattices and the local-to-global principle}\label{Gchap:cg}

As we have seen, using Stone duality one can give a framework that completely solves (in principle!) the classification problem for radical tensor-ideals of an essentially small tt-category. In this chapter we consider the analogous problem for `big tt-categories' i.e.\ symmetric monoidal triangulated categories with arbitrary coproducts that are controlled by a set of small objects. This turns out to be much harder, we do not have a framework for attacking the corresponding classification problem in general, and there are many open questions and challenges.

%----------------------------------------------------------------------------------------------------------------------------------------------

%----------------------------------------------------------------------------------------------------------------------------------------------

\subsection{Big tt-categories}

A compactly generated triangulated category $\mathsf{T}$ is to an essentially small one $\mathsf{K}$ roughly as the category of all modules over a ring is to the category of finitely presented ones (the ring should be coherent for a particularly good analogy). This is meant in the sense that a compactly generated triangulated category has all coproducts, and so cannot be essentially small if it is not trivial, but it is determined in some sense by an essentially small subcategory of objects that behave well with respect to filtered colimits.

\begin{definition}
Let $\mathsf{T}$ be a triangulated category admitting all set-indexed coproducts. An object $t\in \mathsf{T}$ is \emph{compact} if for any set $\{M_\lambda \; \vert \; \lambda \in \Lambda\}$ of objects in $\mathsf{T}$ any morphism $t\to \coprod_\lambda M_\lambda$ factors via a summand $\coprod_{i=1}^n M_{\lambda_i}$ involving only finitely many of the $M_\lambda$. In more categorical language this says that the corepresentable functor $\mathsf{T}(t,-)$ preserves coproducts, i.e.\ for every such family  the natural morphism
\begin{displaymath}
\bigoplus_{\lambda \in \Lambda} \mathsf{T}(t, M_\lambda) \to \mathsf{T}(t, \coprod_{\lambda \in \Lambda} M_\lambda)
\end{displaymath}
is an isomorphism.

We say $\mathsf{T}$ is \emph{compactly generated} if there is a set $\mathcal{G}$ of compact objects of $\mathsf{T}$ such that an object $M\in \mathsf{T}$ is zero if and only if
\begin{displaymath}
\mathsf{T}(g, \Sigma^i M) = 0 \text{ for every } g\in \mathcal{G} \text{ and }  i\in \mathbb{Z}.
\end{displaymath}
We denote by $\mathsf{T}^\mathrm{c}$ the full subcategory of compact objects of $\mathsf{T}$ and note that it is an essentially small thick subcategory of $\mathsf{T}$.
\end{definition}

\begin{example}\label{Gex:bigttcats}
Each example of Example~\ref{Gex:bigttcats} can be completed to a compactly generated triangulated category.
\begin{enumerate}
\item\label{Gitem:bigperf} Let $R$ be a commutative ring. Then $\mathsf{D}^\mathrm{perf}(R)$ is the full subcategory of compact objects in the unbounded derived category $\mathsf{D}(R) = \mathsf{D}(\Modu R)$ of all $R$-modules.
\item More generally for $X$ a quasi-compact and quasi-separated scheme the category $\mathsf{D}^\mathrm{perf}(X)$ is the compact part of the unbounded derived category $\mathsf{D}(X)$ of $\mathcal{O}_X$-modules with quasi-coherent cohomology.
\item\label{Gitem:bighopf} If $A$ is a finite dimensional Hopf algebra over a field then $\underline{\modu}\: A$ is the full subcategory of compact objects of $\underline{\Modu}\: A$ the stable category of all $A$-modules.
\item The homotopy category of finite spectra $\mathrm{ho}(\mathsf{Sp}^\omega)$ is the subcategory of compact objects of the stable homotopy category $\mathrm{ho}(\mathsf{Sp})$ of all spectra.
\item One can also complete $\mathsf{K}^\mathrm{b}(\mathsf{M})$, the bounded homotopy category of an additive category $\mathsf{M}$, to a compactly generated triangulated category.
\end{enumerate}
\end{example}

In fact, each of these corresponds to a compactly generated triangulated category with a monoidal structure extending that of the compact objects.

\begin{definition}
A \emph{compactly generated tensor triangulated category} is a triple $(\mathsf{T}, \otimes, \unit)$ where $\mathsf{T}$ is a compactly generated triangulated category, and the pair $(\otimes, \unit)$ is a symmetric monoidal structure on $\mathsf{T}$ such that the tensor product $\otimes$ is a coproduct preserving exact functor in each variable and the compact objects $\mathsf{T}^c$ form a tensor subcategory. In particular, we require that the unit $\unit$ is compact.

It is automatic that $\mathsf{T}$ is actually closed monoidal, and we will usually assume that the internal hom, denoted $\hom(-,-)$ is exact in both variables. (This is true in practice, but also doesn't necessarily matter in practice.)
\end{definition}

\begin{remark}
We will usually just say $\mathsf{T}$ is a compactly generated tensor triangulated category and leave the tensor product and unit implicit.
\end{remark}

\begin{remark}
There are interesting examples where $\mathsf{T}^c$ does not form a tensor subcategory and one requires a different approach. We will, however, not discuss such examples in these notes.
\end{remark}

Unlike in the essentially small case one runs into delicate issues quite quickly without additional hypotheses. Our minimal working setup is as follows.

\begin{definition}
A \emph{rigidly-compactly generated tensor triangulated category} $\mathsf{T}$, called a \emph{big tt-category} for short, is a compactly generated tensor triangulated category such that the full subcategory $\mathsf{T}^c$ of compact objects is rigid. Thus $\mathsf{T}^\mathrm{c}$ is a tensor subcategory of $\mathsf{T}$, which is closed under the internal hom and satisfies Definition~\ref{Gdefn:rigid}.
\end{definition}

\begin{example}
The compactly generated triangulated categories $\mathsf{D}(X)$ for a quasi-compact and quasi-separated scheme $X$, $\underline{\Modu}\: A$ for a cocommutative Hopf algebra $A$, and $\mathrm{ho}(\mathsf{Sp})$ are all big tt-categories.
\end{example}

%----------------------------------------------------------------------------------------------------------------------------------------------

%----------------------------------------------------------------------------------------------------------------------------------------------

\subsection{Localizing and smashing ideals}\label{Gsec:loc}

We now introduce the analogue of the lattice of thick tensor-ideals for a big tt-category and, more or less immediately, throw our hands in the air and look for milder options. Throughout this section $\mathsf{T}$ is a big tt-category.

\begin{definition}
A \emph{localizing subcategory} $\mathsf{L}$ of $\mathsf{T}$ is a full triangulated subcategory which is closed under coproducts. A \emph{localizing tensor-ideal} is a localizing subcategory $\mathsf{L}$ which is closed under tensoring with arbitrary objects of $\mathsf{T}$.

We denote by $\Loc(\mathsf{T})$ and $\Loc^\otimes(\mathsf{T})$ the collections of localizing subcategories and of localizing tensor-ideals of $\mathsf{T}$. These are partially ordered by inclusion and clearly closed under intersections of arbitrary collections. This gives the meet on these posets, and hence also a join operation as we now make explicit.

Given a class of objects $\mathcal{M}$ we denote by $\loc(\mathcal{M})$ the smallest localizing subcategory containing $\mathcal{M}$, i.e.\ the intersection of all localizing subcategories containing $\mathcal{M}$, and ditto for $\loc^\otimes(\mathcal{M})$. We call these the localizing subcategory (resp.\ tensor-ideal) generated by $\mathcal{M}$. Thus one can take the join of a family of localizing subcategories (resp.\ ideals) by taking the localizing subcategory (resp.\ ideal) they generate.
\end{definition}

\begin{remark}
A localizing subcategory of a compactly generated triangulated category is automatically thick (i.e.\ closed under summands). This only requires countable coproducts (or products) and follows by taking a Milnor colimit along repeated action of the idempotent (see \cite[Proposition~1.6.8]{Neeman:2001a}).
\end{remark}

\begin{remark}
There is no guarantee that either $\Loc(\mathsf{T})$ or $\Loc^\otimes(\mathsf{T})$ forms a set. At least we do not know if this is necessarily the case! There are several examples where it is known, but the only method currently available to show this is to describe the collection of localizing subcategories or ideals completely. One does not need to go far to find examples where we have no idea about the size of these collections. This motivates us to introduce various special classes of localizing ideals which we know have mild behaviour.

Despite this, we will not introduce special terminology to distinguish the fact that these might be `large' and will simply call them lattices.
\end{remark}

\begin{remark}
There is a dual notion of colocalizing subcategory and a corresponding definition of colocalizing $\hom$-ideal. 
\end{remark}

The notions of compact generation and generation of localizing subcategories are related in the way that one would hope. A set $\mathcal{G}$ of compact objects of $\mathsf{T}$ generates $\mathsf{T}$ if and only if $\mathsf{T} = \loc(\mathcal{G})$. A proof of this can be found, for instance, in \cite[Lemma~3.2]{Neeman:1996a} (see also Theorem~\ref{Gthm:localizationthm} below). In particular, if $t\in \mathsf{T}^\mathrm{c}$ satisfies $\supp t = \Spc \mathsf{T}^\mathrm{c}$ then $\loc^\otimes(t) = \mathsf{T}$. 

We also need some notation for orthogonal subcategories.

\begin{definition}
Let $\mathcal{M}$ be a class of objects. The \emph{right orthogonal} (or \emph{right `perp'}) of $\mathcal{M}$ is
\[
\mathcal{M}^\perp = \{N\in \mathsf{T} \mid \mathsf{T}(M, \Sigma^i N) = 0 \text{ for all } i\in \mathbb{Z}\}.
\]
The category $\mathcal{M}^\perp$ is always closed under summands, cones, suspensions, and products i.e.\ it is colocalizing. Dually we may consider the \emph{left orthogonal} (or \emph{left `perp'})
\[
{}^\perp\mathcal{M} = \{N\in \mathsf{T} \mid \mathsf{T}(\Sigma^i N, M) = 0 \text{ for all } i\in \mathbb{Z}\}
\]
which is a localizing subcategory.
\end{definition}

\begin{remark}
It is not totally standard to define the right and left orthogonals to be closed under suspensions (and it is quite undesirable if one works with $t$-structures or related objects). We will only really care about suspension closed structures and so it's convenient for us.
\end{remark}

\begin{xca}
Show that for $\mathcal{M}$ a collection of objects one has equalities
\[
\mathcal{M}^\perp = \loc(\mathcal{M})^\perp \text{ and } {}^\perp\mathcal{M} = {}^\perp\coloc(\mathcal{M})
\]
where $\coloc(\mathcal{M})$ denotes the smallest colocalizing subcategory containing $\mathcal{M}$.
\end{xca}

We now start to explore localizing tensor-ideals and the connection to thick tensor-ideals.

\begin{lemma}\label{Glem:tensorclosed}
If $\mathcal{M}$ is a class of objects of $\mathsf{T}$ which is closed under tensoring with compact objects then $\loc(\mathcal{M})$ is a localizing tensor-ideal (ditto for the thick closure if we work inside of $\mathsf{T}^\mathrm{c}$). In particular, if $\mathsf{J}$ is a thick-tensor ideal of $\mathsf{T}^\mathrm{c}$ then $\loc(\mathsf{J})$ is a localizing tensor-ideal and $\Loc^\otimes(\mathsf{T})$ is a sublattice of $\Loc(\mathsf{T})$.
\end{lemma}
\begin{proof}
This is a good exercise (cf.\ Exercise~\ref{Gxca:sublattice}).
\end{proof}

\begin{xca}\label{lem_big_unit_gen}
Let $\mathsf{T}$ be a rigidly-compactly generated tensor triangulated category. If $\mathbf{1}$ is a compact generator for $\mathsf{T}$, i.e.\ if $\mathsf{T} = \loc(\mathbf{1})$, then every localising subcategory of $\mathsf{T}$ is a localising tensor-ideal.
\end{xca}

This is a good spot in which to state some extremely useful facts we will need in the sequel.

\begin{theorem}\label{Gthm:localizationthm}
Let $\mathsf{T}$ be a compactly generated triangulated category, $\mathcal{R}$ a set of compact objects, and set $\mathsf{R} = \loc(\mathcal{R})$. Then:
\begin{enumerate}
\item $\mathsf{R}$ is a compactly generated triangulated category with $\mathcal{R}$ as a generating set;
\item If $\mathcal{R}$ generates $\mathsf{T}$ then $\mathsf{R} = \mathsf{T}$;
\item $\mathsf{R} \cap \mathsf{T}^\mathrm{c} = \thick(\mathcal{R})$;
\item For any triangulated category $\mathsf{S}$ an exact functor $F\colon \mathsf{T} \to \mathsf{S}$ has a right adjoint $G$ if and only of $F$ preserves coproducts, and in this case $G$ preserves coproducts if and only if $F$ sends compacts to compacts;
\item The quotient $\mathsf{T}/\mathsf{R}$ is compactly generated with compact objects the idempotent completion of $\mathsf{T}^\mathrm{c}/\thick(\mathcal{R})$.
\end{enumerate}
\end{theorem}
\begin{proof}
The first three statements are part of the Neeman-Thomason localization theorem, see \cite[Theorem~2.1]{Neeman:1996a}, which also implies the idempotent completion statement in the last item. Part (4) is Theorem~4.1 and Theorem~5.1 of \emph{loc.\ cit.} 
\end{proof}

The procedure of taking localizing subcategories lets us inflate thick tensor-ideals to localizing ideals as in Lemma~\ref{Glem:tensorclosed} and this behaves well lattice-theoretically.

\begin{proposition}\label{Gprop:joinsfinite}
The assignment $\mathsf{J} \mapsto \loc(\mathsf{J})$ defines an injective morphism of posets $\Thick(\mathsf{T}) \to \Loc(\mathsf{T})$ which preserves arbitrary joins. This restricts to an injective map
\[
\Thick^\otimes(\mathsf{T}) \to \Loc^\otimes(\mathsf{T})
\]
which preserves arbitrary joins. Both maps have a section, which is given by sending a localizing subcategory (resp.\ ideal) $\mathsf{L}$ to $\mathsf{L}\cap \mathsf{T}^\mathrm{c}$, which preserves arbitrary meets.
\end{proposition}
\begin{proof}
Restricting to the compacts is a section by part (3) of Theorem~\ref{Gthm:localizationthm} and this gives injectivity. The assignment sends ideals in the compacts to ideals in $\mathsf{T}$ by Lemma~\ref{Glem:tensorclosed}. So we need to check the claimed joins and meets are preserved by these constructions. Because the tensor-ideals are closed under joins and meets in the lattices of localizing and thick subcategories it is enough to work without the tensor.

Let us start with $(-)^\mathrm{c}\colon \Loc(\mathsf{T}) \to \Thick(\mathsf{T}^\mathrm{c})$. If $\mathsf{L}_\lambda$ is a family of localizing subcategories then 
\[
\left(\bigcap_\lambda \mathsf{L}_\lambda \right) \cap \mathsf{T}^\mathrm{c} = \left( \bigcap_\lambda (\mathsf{L}_\lambda \cap \mathsf{T}^\mathrm{c}) \right).
\]
Now suppose that $\mathsf{J}_\lambda$ is a family of thick subcategories of $\mathsf{T}^\mathrm{c}$. Preservation of the join is also straightforward:
\[
\loc\left( \bigvee_\lambda \mathsf{J}_\lambda \right) = \loc\left( \thick\left( \bigcup_\lambda \mathsf{J}_\lambda \right) \right) = \loc\left( \bigcup_\lambda \mathsf{J}_\lambda \right) = \loc\left( \bigcup_\lambda \loc(\mathsf{J}_\lambda) \right) = \bigvee_\lambda \loc(\mathsf{J}_\lambda).
\]
\end{proof}

\begin{remark}
The two assignments, namely completion to a localizing subcategory and restriction to compacts, are left and right adjoint to one another cf.\ Section~\ref{Gsec:bigsupport}.
\end{remark}

\begin{definition}
We call the map $\Thick^\otimes(\mathsf{T}) \to \Loc^\otimes(\mathsf{T})$ inflation and the localizing ideals in its image \emph{finite}. 
\end{definition}

Inflation gives us a connection between $\Spc \mathsf{T}^\mathrm{c}$ and the collection of all localizing ideals. We would like to better understand the image of inflation and to exploit this connection to give a notion of support for arbitrary objects of $\mathsf{T}$. The key observation is Miller's theorem which leads to the notion of idempotent algebras and coalgebras. All of this works more generally, and so we present it in an abstract fashion which makes more transparent the connection to some relatively recent developments.

We combine several facts into the next theorem, which also is partially a definition.

\begin{theorem}\label{Gthm:finiteloc}
Let $\mathsf{T}$ be a rigidly-compactly generated tensor triangulated category and $\mathsf{J}\in \Thick^\otimes(\mathsf{T}^\mathrm{c})$. Then
%\begin{displaymath}
%\xymatrix{
%\sfS \ar[r]<0.5ex>^-{i_*} \ar@{<-}[r]<-0.5ex>_-{i^!} & \mathsf{T} \ar[r]<0.5ex>^-{j^*} \ar@{<-}[r]<-0.5ex>_-{j_*} & \sfS^\perp
%}
%\end{displaymath}
\begin{enumerate}
\item both $\loc(\mathsf{J})$ and $\mathsf{J}^\perp = \loc(\mathsf{J})^\perp$ are localizing tensor-ideals;
\item the inclusion $i_*\colon \loc(\mathsf{J}) \to \mathsf{T}$ admits a right adjoint $i^!$ and $C = i_*i^!\unit$ is an \emph{idempotent coalgebra} i.e.\ it is equipped with a morphism $\varepsilon\colon C \to \unit$ (the counit of adjunction) and $C\otimes \varepsilon \colon C\otimes C\to C$ is an isomorphism (the inverse of this map gives a unique coalgebra structure with counit $\varepsilon$);
\item the inclusion $j_*\colon \loc(\mathsf{J})^\perp \to \mathsf{T}$ admits a left adjoint $j^*$ and $A = j_*j^*\unit$ is an \emph{idempotent algebra} i.e.\ it is equipped with a morphism $\eta\colon \unit \to A$ (the unit of adjunction) and $A\otimes \eta \colon A\to A\otimes A$ is an isomorphism (the inverse of this map gives a unique algebra structure with unit $\eta$);
\item the objects $A$ and $C$ satisfy $A\otimes C\cong 0$;
\item the objects $A$ and $C$ realize these adjunctions in the sense that there are natural isomorphisms
\[
i_*i^! \cong C\otimes (-) \text{ and } j_*j^* \cong A\otimes (-)
\]
and identifications
\[
\loc(\mathsf{J}) = \Comodu_\mathsf{T}(C) = \ker(A\otimes(-)) 
\]
and
\[
\loc(\mathsf{J})^\perp = \Modu_\mathsf{T}(A) = \ker(C\otimes(-))
\]
where these are the categories of comodules over $C$ in $\mathsf{T}$ and modules over $A$ in $\mathsf{T}$ respectively, or equivalently objects $M$ equipped with an isomorphism $M \cong C\otimes M$ or $M\cong A\otimes M$;
\item for every $N \in \mathsf{T}$ there is a triangle
\[
C\otimes N \to N \to A\otimes N \to
\]
which is functorial in $N$ and with $C\otimes N\in \loc(\mathsf{J})$ and $A\otimes N\in \loc(\mathsf{J})^\perp$ which is obtained by tensoring $N$ with the triangle
\[
C\to \unit \to A \to
\]
where the maps are the counit and unit of $C$ and $A$ respectively.
\end{enumerate}
\end{theorem}
\begin{proof}
We give a very rough outline of the proof and provide some references. Let us begin with the sketch. Compactness of the objects of $\mathsf{J}$ guarantees that $\mathsf{J}^\perp$ is closed under coproducts and hence both localizing and colocalizing. The fact that $\mathsf{J}^\perp$ is a tensor-ideal, which is really the key to everything, is where rigidity enters. One can show that, for $t\in \mathsf{T}^\mathrm{c}$, the functor $\hom(t,\unit)\otimes(-)$ is both a right and left adjoint of $t\otimes(-)$ on all of $\mathsf{T}$. From this one shows, using the adjunction, that closure of $\mathsf{J}^\perp$ under tensoring with compacts is equivalent to closure of $\loc(\mathsf{J})$ under tensoring with compacts which follows from Lemma~\ref{Glem:tensorclosed}. The same lemma then implies that both of these localizing categories are localizing ideals.

The existence of the adjoints is a consequence of Brown representability in the form of Theorem~\ref{Gthm:localizationthm}(4) and these guarantee the existence of a triangle
\[
i_*i^!\unit \to \unit \to j_*j^*\unit \to
\]
One then uses that the outer terms lie in ideals to tensor this triangle around and everything else follows directly.

A convenient source for the details is \cite{Balmer/Favi:2011a} and there is also a discussion and further references in Section~2 of \cite{Stevensontour}.
\end{proof}

\begin{remark}
The theorem as stated is extremely redundant: everything is determined by any one of $\mathsf{J}$, $\mathsf{J}^\perp$, $A$, or $C$ and idempotence of $A$ is equivalent to idempotence of $C$ and to $A\otimes C\cong 0$.
\end{remark}

\begin{example}
Let us give a first general example (see Example~\ref{example:algebras2} for something more concrete). If $R$ is a commutative ring and $S$ is a multiplicative subset of $R$ then $S^{-1}R$ is an idempotent algebra in $\mathsf{D}(R)$. This is because $R \to S^{-1}R$ is a flat epimorphism. It corresponds to the localization $\mathsf{D}(R) \to \mathsf{D}(S^{-1}R)$ given by base change with right adjoint restriction of scalars. Restriction of scalars identifies $\mathsf{D}(S^{-1}R)$ with $S^{-1}R$-modules in $\mathsf{D}(R)$.
\end{example}

\begin{xca}
Give a description of the kernel of $\mathsf{D}(R) \to \mathsf{D}(S^{-1}R)$ and show that it is generated by perfect complexes, i.e.\ it is the inflation of a thick subcategory of $\mathsf{D}^\mathrm{perf}(R)$.
\end{xca}

With this we have the technology to prove that inflation $\Thick^\otimes(\mathsf{T}^\mathrm{c}) \to \Loc^\otimes(\mathsf{T})$ preserves finite meets.

\begin{proposition}\label{Gprop:meetsfinite}
The inflation map $\Thick^\otimes(\mathsf{T}^\mathrm{c}) \to \Loc^\otimes(\mathsf{T})$ preserves finite meets.
\end{proposition}
\begin{proof}
Let $\mathsf{J}_1$ and $\mathsf{J}_2$ be tensor-ideals of $\mathsf{T}^\mathrm{c}$. Because $\mathsf{T}^\mathrm{c}$ is rigid we know that
\[
\mathsf{J}_1 \cap \mathsf{J}_2 = \mathsf{J}_1 \otimes \mathsf{J}_2
\]
by Lemma~\ref{Glem:meettensor}. Let us denote $\mathsf{J}_1 \otimes \mathsf{J}_2$ by $\mathsf{J}_{12}$. We need to show that
\[
\loc(\mathsf{J}_1) \cap \loc(\mathsf{J}_2) = \loc(\mathsf{J}_{12}).
\]
We note that the containment $\loc(\mathsf{J}_{12}) \subseteq \loc(\mathsf{J}_1) \cap \loc(\mathsf{J}_2)$ is clear.

Using the theorem we know that $\loc(\mathsf{J}_1)$ and $\loc(\mathsf{J}_2)$ are determined by idempotent coalgebras $C_1$ and $C_2$. The tensor product $C_{12} = C_1\otimes C_2$ is again an idempotent coalgebra and we claim its category of comodules in $\mathsf{T}$ is exactly $\loc(\mathsf{J}_1) \cap \loc(\mathsf{J}_2)$. Indeed, for $M\in \loc(\mathsf{J}_1) \cap \loc(\mathsf{J}_2)$ we have $C_1\otimes M\cong M$ and $C_2\otimes M\cong M$ and hence $C_1\otimes C_2\otimes M\cong M$ and if $C_1\otimes C_2 \otimes M \cong M$ then 
\[
C_1 \otimes M \cong C_1 \otimes C_1 \otimes C_2 \otimes M \cong C_1 \otimes C_2 \otimes M \cong M
\]
and similarly for $C_2\otimes M$. The proposition then follows once we show that $C_{12} \in \loc(\mathsf{J}_{12})$. Heuristically the argument is that $C_i$ is generated by $\mathsf{J}_i$ and so $C_1 \otimes C_2$ is generated by $\mathsf{J}_1\otimes \mathsf{J}_2$. The precise statement is that
\[
\loc(\mathsf{J}_{12}) = \loc(\mathsf{J}_1 \otimes \mathsf{J}_2) = \loc(M\otimes N \mid M\in \loc(\mathsf{J}_1), N\in \loc(\mathsf{J}_2))
\]
where the non-trivial equality on the right is a special case of \cite[Lemma~3.11]{Stevenson:2013a} (which just makes the heuristic argument rigorous).
\end{proof}

\begin{remark}\label{remark:inflationmeets}
It is not true that inflation preserves arbitrary meets! A counterexample is given in \cite[Remark~5.12]{BKSframe}.
\end{remark}

\begin{example}\label{example:algebras2}
Let us return to the example of Section~\ref{Gsec:realexample} to explore all of this in a concrete setting. We know that the tensor-ideals of $\mathsf{D}^\mathrm{perf}(k[x])$ are given by specialization closed subsets of $\Spec k[x] = \mathbb{A}^1$. The inflation of $\mathsf{D}^\mathrm{perf}(k[x])$ is just $\mathsf{D}(k[x])$ and the corresponding idempotent coalgebra is $k[x]$ with complementary idempotent algebra $0$. Similarly, the inflation of $0$ is $0$ and now we have idempotent coalgebra $0$ and algebra $k[x]$. It is \emph{not} typical for an object to be both an idempotent algebra and a coalgebra: this implies it is a summand of the tensor-unit.

Now let us consider some proper thick tensor-ideals. We could consider the maximal tensor-ideal $\mathsf{M}$ consisting of perfect complexes with torsion cohomology
\[
\mathsf{M} = \thick(k[x]/\mathfrak{p} \mid \mathfrak{p} \in \mathbb{A}^1\setminus \{(0)\}).
\]
The inflation $\loc(\mathsf{M})$ is just
\begin{align*}
\loc(\mathsf{M}) &= \loc(k[x]/\mathfrak{p} \mid \mathfrak{p}\in \mathbb{A}^1\setminus \{(0)\}) \\
&= \{M\in \mathsf{D}(k[x]) \mid H^i(M) \text{ is torsion for each } i\in \mathbb{Z}\}.
\end{align*}
The localization, i.e.\ the right perp, is given by the full subcategory of complexes of $k(x)$-vector spaces 
\[
\mathsf{M}^\perp \cong \mathsf{D}(k(x)),
\]
where $k(x)$ is the field of rational functions, and the corresponding idempotent algebra is $k(x)$. This is not too surprising: $k(x)$ is a flat $k[x]$-algebra which is idempotent, since $k[x] \to k(x)$ is an epimorphism of rings, and so is also an idempotent algebra in the derived category. We can compute the corresponding idempotent coalgebra $C$ using the triangle
\[
C \to k[x] \to k(x) \to \Sigma C
\]
and see (using that $k[x] \to k(x)$ is injective and so this just comes from a short exact sequence) that 
\[
C = \Sigma^{-1} k(x)/k[x] = \Sigma^{-1} \bigoplus_{\mathfrak{p}\in \mathbb{A}^1\setminus \{(0)\}} E(k[x]/\mathfrak{p})
\]
where $E(k[x]/\mathfrak{p})$ is the injective envelope of the residue field $k[x]/\mathfrak{p}$. Concretely, if $\mathfrak{p} = (f)$ then 
\[
E(k[x]/\mathfrak{p}) = k[x]_f/k[x] = \colim_n k[x]/(f^n). 
\]

This is the typical behaviour: for any set $X\subseteq \mathbb{A}^1$ of closed points 
\[
C_X = \Sigma^{-1} \bigoplus_{\mathfrak{p}\in X} E(k[x]/\mathfrak{p})
\]
is an idempotent coalgebra and corresponds to the inflation of the tensor-ideal of perfect complexes supported on $X$. 

An instructive special case is $X = \mathbb{A}^1\setminus \{\mathfrak{p},(0)\}$ where one gets the idempotent algebra $k[x]_\mathfrak{p}$ and the idempotent coalgebra via the formula above.
\end{example}

All of this leads to a thought: can we skip the inflation and just look at idempotent (co)algebras? The answer is yes and the corresponding notion is that of a smashing ideal.

\begin{definition}
An ideal $\mathsf{S}$ of $\mathsf{T}$ is \emph{smashing} if $\mathsf{S}^\perp$ is closed under coproducts (and hence itself a localizing tensor-ideal). 
\end{definition}

The same arguments as Theorem~\ref{Gthm:finiteloc} give that for a smashing ideal $\mathsf{S}$:
\begin{enumerate}
\item The inclusion of $\mathsf{S}$ and its right perp $\mathsf{S}^\perp$ admit a right and a left adjoint respectively and these are realized by an idempotent coalgebra $C$ and an idempotent algebra $A$ (in particular all of these functors preserve coproducts);
\item The idempotent coalgebra $C$ determines $\mathsf{S}$ via
\[
\mathsf{S} = \Comodu_\mathsf{T}(C) = \{M\in \mathsf{T} \mid C\otimes M \cong M\}
\]
\item The idempotent algebra determines $\mathsf{S}^\perp$ via
\[
\mathsf{S}^\perp = \Modu_\mathsf{T}(A) = \{M\in \mathsf{T} \mid A \cong A\otimes M\}.
\]
\end{enumerate}

\begin{remark}
Not every smashing ideal is finite, i.e.\ inflated from a thick tensor-ideal of compact objects. There exist counterexamples for non-noetherian commutative rings (see \cite{Kellersmashing}). 
%For instance, we could take the local ring $\mathbb{F}_p[[t^{1/p^\infty}]]$
\end{remark}

\begin{definition}
Let us denote by $\mathcal{S}(\mathsf{T})$ the collection of smashing tensor-ideals of $\mathsf{T}$ ordered by inclusion. This forms a set and so it is an honest poset! (This follows from the analogous fact for all smashing subcategories as in \cite{Krause:2000a} and there are, by now, several other proofs e.g.\ it is proved in \cite{BKSframe} directly for smashing ideals.)
\end{definition}

\begin{theorem}\label{Gthm:smashingframe}
The poset $\mathcal{S}(\mathsf{T})$ is a frame with joins given by
\[
\bigvee_{\lambda\in \Lambda} \mathsf{S}_\lambda = \loc\left( \bigcup_{\lambda \in \Lambda} \mathsf{S}_\lambda \right)
\]
and finite meets given by
\[
\mathsf{S}_1 \wedge \mathsf{S}_2 = \mathsf{S}_1 \cap \mathsf{S}_2.
\]
In particular, the inclusion $\mathcal{S}(\mathsf{T}) \to \Loc(\mathsf{T})$ preserves joins and finite meets.
\end{theorem}
\begin{proof}
This can be deduced from \cite{BKSframe}. The argument is based on passing to the category $\Fun((\mathsf{T}^\mathrm{c})^\op, \Modu \mathbb{Z})$ of additive presheaves on the compacts. In the presence of an enhancement to a stably symmetric monoidal $\infty$-category one can give a more direct proof (see the discussion below).

Since the argument is left implicit in \cite{BKSframe} let us give the argument that the join is computed as claimed. We have a series of equalities
\begin{align*}
\loc\left( \bigcup_{\lambda \in \Lambda} \mathsf{S}_\lambda \right)^\perp &= \left( \bigcup_{\lambda \in \Lambda} \mathsf{S}_\lambda \right)^\perp \\
&= \left\{M \in \mathsf{T} \; \middle| \; \mathsf{T}(S, M) = 0 \; \forall \; S \in \bigcup_\lambda \mathsf{S}_\lambda \right\} \\
&= \bigcap_\lambda \{M \in \mathsf{T} \mid \mathsf{T}(S, M) = 0 \; \forall \; S \in \mathsf{S}_\lambda \} \\
&= \bigcap_\lambda (\mathsf{S}_\lambda^\perp)
\end{align*}
where the first is a small exercise (using that the left perp is always localizing) and the rest are clear. We see from the final description that the right perp is an intersection of localizing ideals and hence again a localizing ideal and so the join of the $\mathsf{S}_\lambda$ as localizing ideals is smashing as claimed.
\end{proof}

\begin{remark}
Taking infinite meets is not so straightforward: the intersection of a family of smashing subcategories can fail to be smashing (this is related to Remark~\ref{remark:inflationmeets} and the same counterexample \cite[Remark~5.12]{BKSframe} applies).
\end{remark}

\begin{definition}\label{Gdef:idempotentalgebra}
We denote by $\Idem(\mathsf{T})$ the set of idempotent algebras and we let $\Idem^\mathrm{fin}(\mathsf{T})$ be the subset of idempotent algebras arising from smashing ideals inflated from tensor ideals of $\mathsf{T}^\mathrm{c}$. The set $\Idem(\mathsf{T})$ is partially ordered as follows: for algebras $\eta_i\colon \unit \to A_i$ ($i=1,2$) we have $A_1 \leq A_2$ if there is a commutative triangle
\[
\begin{tikzcd}
                          & \unit \arrow[ld, "\eta_1"'] \arrow[rd, "\eta_2"] &     \\
A_1 \arrow[rr, "\varphi"] &                                                  & A_2
\end{tikzcd}
\]
This makes $\Idem(\mathsf{T})$ into a complete lattice (see \cite{BKSframe}) and it is moreover a frame. The join of $A_1$ and $A_2$ is given by $A_1\otimes A_2$ and the meet is given by the Mayer-Vietoris triangle
\[
A_1 \wedge A_2 = \mathrm{fibre}(A_1\oplus A_2 \to A_1\otimes A_2).
\] 
\end{definition}

\begin{remark}
For $A_1,A_2 \in \Idem(\mathsf{T})$ we have $A_1 \leq A_2$ if and only if any of the following equivalent conditions hold:
\begin{itemize}
\item $\Modu_\mathsf{T}(A_1) \supseteq \Modu_\mathsf{T}(A_2)$;
\item $\ker(A_1 \otimes -) \subseteq \ker(A_2\otimes -)$;
\item there exists an isomorphism $A_2 \cong A_1\otimes A_2$;
\item $\eta_1\otimes A_2 \colon A_2 \to A_1\otimes A_2$ is an isomorphism.
\end{itemize}
\end{remark}

\begin{xca}
Prove the equivalence of the conditions above.
\end{xca}

\begin{remark}
For an infinite family of idempotent algebras the join is the colimit over all finite joins, i.e.\ the `infinite tensor product' which can be constructed using the unit maps. If $\mathsf{T}$ were an $\infty$-category this would just make sense, but if $\mathsf{T}$ is only triangulated and not assumed to have an enhancement then one can make sense of this in the functor category.
\end{remark}

\begin{proposition}
There is an isomorphisms of frames
\[
\mathcal{S}(\mathsf{T}) \cong \Idem(\mathsf{T})
\]
sending the finite localizations to $\Idem^\mathrm{fin}(\mathsf{T})$.
\end{proposition}
\begin{proof}
This is simply given by sending a smashing ideal to the corresponding idempotent algebra (the image of $\unit$ under the corresponding localization) and an idempotent algebra to the ideal of objects it annihilates.
\end{proof}

\begin{remark}
In a similar way one can define the frame of idempotent coalgebras and there is the expected lattice isomorphism.
\end{remark}

As a final remark let us note that one cannot expect $\Idem(\mathsf{T})$ to be as nicely behaved as the coherent frame $\Idem^\mathrm{fin}(\mathsf{T}) \cong \Thick^\otimes(\mathsf{T}^\mathrm{c})$.

\begin{remark}
It is an open question as to whether $\Idem(\mathsf{T})$ is a spatial frame. It is known to be spatial in some special cases, but we have no idea in general. Even in the case that it is spatial it seems it need not be a coherent frame.
\end{remark}

\begin{remark}\label{Grem:idempotentnotation}
Let us make a comment on notation. One often sees some version of $e$ and $f$ used to denote a pair of an idempotent coalgebra and algebra respectively. When $V$ is a Thomason subset of $\Spc \mathsf{T}^\mathrm{c}$ it is customary to write $\Gamma_V\mathbf{1}$ and $L_V\mathbf{1}$ for the corresponding idempotent coalgebra and algebra.
\end{remark}

%----------------------------------------------------------------------------------------------------------------------------------------------

%----------------------------------------------------------------------------------------------------------------------------------------------

\subsection{Supports for big tt-categories}\label{Gsec:bigsupport}

We now describe a way to define some notions of support for objects of $\mathsf{T}$. The material in this section is inspired by \cite{Balmer/Favi:2011a} and \cite{BalchinStevenson} and the point of view presented is based on joint work with Scott Balchin. Our viewpoint is nonstandard, and allows us to capture naive and refined notions of support using the same framework. In certain cases, which we will highlight, our notion is not quite identical to what exists in the literature. However, as far as I am aware in these cases the old notion has not proved particularly useful and so I don't feel particular attachment to it. Moreover, we will see that we don't actually lose information in these cases.

We will temporarily leave the world of frames, at least in principle. Up to set-theoretic issues $\Loc^\otimes(\mathsf{T})$ is a complete lattice, and so it makes sense to probe it by complete lattices that we understand which are not necessarily frames. A morphism of complete lattices is a map preserving arbitrary joins and finite meets.

Let $F$ be a complete lattice (we use $F$ because we have in mind that it might be a frame and I lack imagination). %and let $\tau\colon F\to \Loc^\otimes(\mathsf{T})$ be a map of complete lattices. 

\begin{definition}\label{Gdef:catlattice}
We call a map of complete lattices $\tau\colon F \to \Loc^\otimes(\mathsf{T})$ a \emph{categorified complete lattice} (following Clausen and Scholze). 
\end{definition}

\begin{remark}
This is not so snappy, so it is probably better to just say categorified lattice.
\end{remark}

\begin{remark}\label{Grem:Ko}
If we assume that $\mathsf{T}$ is a presentably symmetric monoidal stable $\infty$-category, $F$ is a frame, and $\tau$ factors through $\Idem(\mathsf{T})$ then by \cite{aoki2023sheaves} Definition~\ref{Gdef:catlattice} is equivalent to giving a symmetric monoidal and colimit preserving functor $\mathrm{Shv}(F; \mathsf{Sp}) \to \mathsf{T}$ from sheaves of spectra on $F$ to $\mathsf{T}$. This determines an action of the category of spectral sheaves on $\mathsf{T}$ which gives rise to the structure we have observed and will observe.
\end{remark}

\begin{example}\label{Gex:catlattice}
We present some of the abstract examples we will have in mind throughout what follows:
\begin{enumerate}
\item We could take $\tau = \mathrm{id}\colon \Loc^\otimes(\mathsf{T}) \to \Loc^\otimes(\mathsf{T})$. This is the universal example. 
\item We could take $\tau$ to be the inclusion of $\Thick^\otimes(\mathsf{T}^c)$ which is valid by Propositions~\ref{Gprop:joinsfinite} and \ref{Gprop:meetsfinite}.
\item We could take $\tau$ to be the inclusion of $\mathcal{S}(\mathsf{T}) \cong \Idem(\mathsf{T})$ which is valid by Theorem~\ref{Gthm:smashingframe}. (In this section we will often regard this isomorphism as an identification, i.e.\ we will be sloppy with distinguishing smashing ideals and the corresponding idempotent algebra.)
\end{enumerate}
As we will see later, under the right conditions we can refine some of these examples to obtain better approximations of $\Loc^\otimes(\mathsf{T})$.
\end{example}

\begin{remark}
If we insisted that $\mathsf{T}$ were a presentably symmetric monoidal stable $\infty$-category, $F$ were a frame, and that $\tau$ factored via the frame $\Idem(\mathsf{T})$ this would be the notion of \emph{categorified locale} as introduced by Clausen and Scholze in their study of complex geometry through the lens of condensed mathematics \cite[Definition~7.1]{Complex}. This is a reasonable restriction to make (cf.\ also Remark~\ref{Grem:Ko}). It does not seem necessary in order to set up a reasonable definition though and as we don't know enough about $\Loc^\otimes(\mathsf{T})$ in general it may be worthwhile resisting the temptation to impose hypotheses.
\end{remark}

Because $\tau$ preserves joins, i.e.\ preserves colimits, it will always have a right adjoint $(-)^\mathrm{o}\colon \Loc^\otimes(\mathsf{T}) \to F$ which sends a localizing ideal $\mathsf{L}$ to
\[
\mathsf{L}^\mathrm{o} = \bigvee\{f\in F \mid \tau(f) \leq \mathsf{L} \}.
\]

\begin{example}
If we take (2) and (3) of Example~\ref{Gex:catlattice} then the above sends $\mathsf{L}$ to $\mathsf{L}\cap \mathsf{T}^\mathrm{c}$ and to the largest smashing subcategory of $\mathsf{L}$ respectively. This motivates the notation: $\mathsf{L}^\mathrm{o}$ is the interior of $\mathsf{L}$ relative to $\tau$.
\end{example}

This construction is useful, but not very helpful in the context of trying to understand general localizing ideals. In practice, many localizing ideals contain no compact objects and so we can't really build a useful invariant from this right adjoint. 

What about a left adjoint? As we have remarked, in examples $\tau$ need not preserve arbitrary meets and so the left adjoint won't necessarily exist. However, the formula which would define a left adjoint makes sense whether this left adjoint exists or not. This is the basis for Definition~\ref{Gdef:bigsupport}. In order to give it we  recall the notion of a filter on a poset.

Given a complete lattice $F$ a filter $\mathcal{V}$ on $F$ is a non-empty subset of $F$ which is upward closed and downward directed. Concretely, this means that if $f\leq f'$ and $f\in \mathcal{V}$ then $f'\in \mathcal{V}$ and whenever $f,g\in \mathcal{V}$ then so is $f\wedge g$. We let $\mathrm{Filt}(F)$ denote the set of filters on $F$ ordered by inclusion and note that this is a complete lattice and a frame if $F$ is distributive.

\begin{xca}
Construct the meet and join on $\mathrm{Filt}(F)$ and verify the above claims.
%join is upward closure of finite meets and meet is join. But can also just intersect filters and see meet that way and give join via filtered union (double check all this).
\end{xca}

\begin{definition}\label{Gdef:bigsupport}
Let $\tau \colon F \to \Loc^\otimes(\mathsf{T})$ be a map of complete lattices. We define maps of posets
\[
\widetilde{\sigma}\colon \Loc^\otimes(\mathsf{T})^\op \to \mathrm{Filt}(F) \quad \text{by} \quad \widetilde{\sigma}\mathsf{L} = \{f\in F \mid \mathsf{L}\leq \tau f \}
\]
and 
\[
 \sigma\colon \Loc^\otimes(\mathsf{T}) \to F \quad \text{by} \quad \sigma \mathsf{L} = \bigwedge \widetilde{\sigma}\mathsf{L}.
\]
%by setting
%\[
%\widetilde{\sigma}\mathsf{L} = \{f\in F \mid \mathsf{L}\leq \tau f \}
%\]
%and 
%\[
%\sigma \mathsf{L} = \wedge \widetilde{\sigma}\mathsf{L}.
%\]
We call these assignments the \emph{filter of supports} and \emph{support} of $\mathsf{L}$ respectively (implicitly with respect to $\tau$ and explicitly when necessary).

Given an object $M\in \mathsf{T}$ we set
\[
\sigma(M) = \sigma(\loc^\otimes(M)) \quad \text{and} \quad \widetilde{\sigma}(M) = \widetilde{\sigma}(\loc^\otimes(M))
\]
\end{definition}

\begin{remark}
If $\tau$ preserves meets then $\sigma$ is its left adjoint. In this case $\mathsf{L} \leq \tau(f)$ if and only if $\sigma(\mathsf{L}) \leq f$ and so $\widetilde{\sigma}(\mathsf{L})$ is just the upward closure of $\sigma(\mathsf{L})$, i.e.\ the support and filter of supports are equivalent data.
\end{remark}

We now go on a long digression concerning `examples' and the comparison with the traditional notion with values in the spectrum of the compacts. Let us consider a motivating example and corresponding computation. Take $\tau$ to be the inclusion $\Thick^\otimes(\mathsf{T}^\mathrm{c}) \to  \Loc^\otimes(\mathsf{T})$ given by inflation. We will identify $\Thick^\otimes(\mathsf{T}^c)$ with the lattice of Thomason subsets when convenient. For a localizing subcategory $\mathsf{L}$ we have
\[
\widetilde{\sigma}\mathsf{L} = \{ V \in \Thom(\Spc \mathsf{T}^\mathrm{c}) \mid \mathsf{L} \subseteq \loc(\mathsf{T}_V^\mathrm{c})\}
\]
the collection of Thomason subsets $V$ such that $\mathsf{L}$ is contained in $\loc(\mathsf{T}_V^\mathrm{c})$ where $\mathsf{T}_V^\mathrm{c}$ is the ideal corresponding to $V$. By definition $\sigma\mathsf{L}$ is the meet of these Thomason subsets, i.e.\ the largest Thomason subset contained in their intersection.

Let us compare this to the traditional notion in a pair of cases. In order to do this we need to recall the usual naive support. It is useful to have a bit of notation and a small fact which we leave as an exercise. For a point $x$ of a space $X$ we let
\[
\mathcal{Z}(x) = \{y\in T \mid x \notin \overline{\{y\}} \}
\]
i.e.\ we take the collection of points that don't specialize to $x$. For a Thomason subset $V$, with corresponding ideal of compacts $\mathsf{T}_V^\mathrm{c}$, we let $A_V$ be the corresponding idempotent algebra (mentally replace this by $L_V\mathbf{1}$ if you prefer).

\begin{xca}
Let $\mathsf{P}$ be a prime tensor ideal of the compacts. Show that $\supp \mathsf{P} = \mathcal{Z}(\mathsf{P})$ (so in particular $\mathcal{Z}(\mathsf{P})$ is Thomason).
\end{xca}

\begin{definition}
For a localizing ideal $\mathsf{L}$ the naive support is
\[
\supp \mathsf{L} = \{ \mathsf{P} \in \Spc \mathsf{T}^\mathrm{c} \mid A_{\mathcal{Z}(\mathsf{P})}\otimes \mathsf{L} \neq 0\}.
\]
Equivalently, this is the set of those primes $\mathsf{P}$ such that $\mathsf{L}$ survives in the localization by $\mathsf{P}$ i.e.\ $\mathsf{L}$ is not contained in $\loc(\mathsf{P}) = \loc(\mathsf{T}_{\mathcal{Z}(\mathsf{P})}^\mathrm{c})$.
\end{definition}

Let us consider 
\[
U(\mathsf{L}) = \Spc \mathsf{T}^\mathrm{c} \setminus \supp \mathsf{L} = \{\mathsf{P} \in \Spc \mathsf{T}^\mathrm{c} \mid \mathsf{L} \subseteq \loc(\mathsf{T}_{\mathcal{Z}(\mathsf{P})}^\mathrm{c}) \}.
\]

Continuing to take for $\tau$ the inclusion of $\Thick^\otimes(\mathsf{T}^\mathrm{c})$, and continuing to abuse the identification of thick tensor-ideals with Thomason subsets, we see that
\[
U(\mathsf{L}) = \widetilde{\sigma}(\mathsf{L}) \cap \{\mathcal{Z}(\mathsf{P}) \mid \mathsf{P}\in \Spc \mathsf{T}^\mathrm{c} \}
\]
i.e.\ the complement of the naive support is just $\widetilde{\sigma}$ restricted to the meet-prime Thomason subsets. If $\supp \mathsf{L}$ is a Thomason subset then
\[
\supp \mathsf{L} = \bigwedge_{\mathsf{P}\in U(\mathsf{L})} \mathcal{Z}(\mathsf{P}) = \sigma(\mathsf{L})
\]
and if we replace the meet by an intersection the first equality always holds. Thus $\widetilde{\sigma}$ is an honest generalization of the naive support and $\sigma$ is trying its best.

\begin{remark}
In the case that $\Idem \mathsf{T}$ is spatial there is a completely analogous notion of naive smashing support. This is recovered in the same way through our definition when $\tau$ is the inclusion of $\Idem \mathsf{T}$. The advantage of our construction is that it works without the assumption that the smashing ideals form a spatial frame.
\end{remark}

Now let us do an example to see that the notion we have defined can actually do much better than recovering the naive support. This will motivate the discussion of Section~\ref{Gsec:refined}.

We use $\mathcal{P}$ to denote the power set operation on sets. The power set of a set, ordered by inclusion, is a spatial frame which is dual to the discrete topology on the given set.

\begin{example}\label{Gex:Foxbysupport}
Let $R$ be a commutative noetherian ring and take $\mathsf{T} = \mathsf{D}(R)$. Consider the function $\tau\colon  \mathcal{P}(\Spec R) \to \Loc^\otimes(\mathsf{T})$ from the power set of $\Spec R$ defined by
\[
\tau(W) = \loc(k(\mathfrak{p}) \mid \mathfrak{p} \in W),
\]
where $k(\mathfrak{p})$ is the residue field at $\mathfrak{p}$. This is clearly join preserving and it preserves binary meets even without the noetherian hypothesis. Being noetherian guarantees (the non-trivial fact) that $\tau(\Spec R) = \mathsf{D}(R)$ which shows that $\tau$ is a map of complete lattices. In this case $\sigma$ is Foxby's small support:
\[
\sigma(\mathsf{L}) = \{\mathfrak{p}\in \Spec R \mid \mathsf{L} \otimes k(\mathfrak{p}) \neq 0 \}
\]
and Neeman's theorem asserts that $\tau$ is actually an isomorphism of frames. (See Example~\ref{Gex:Foxbysupport2} for further discussion.)
\end{example}

\begin{xca}
Verify that $\tau$ as in the example preserves binary meets (without the noetherian assumption if you feel brave).
\end{xca}

To conclude let us give the basic properties of the support and filter of supports. We fix some arbitrary $F$ and $\tau$, use variations on $M$ to refer to objects of $\mathsf{T}$, and let $\mathsf{L}$ and $\mathsf{M}$ be localizing ideals.

\begin{lemma}
The filter of supports $\widetilde{\sigma}$ satisfies the following properties:
\begin{enumerate}
\item $\widetilde{\sigma}(\Sigma^i M) = \widetilde{\sigma}(M)$ for all $i \in \mathbb{Z}$;
\item $\widetilde{\sigma}(0) = F$ and $\widetilde{\sigma}(\mathsf{T}) = \{f\in F \mid \tau(f) = \mathsf{T} \}$. In particular, $\widetilde{\sigma}(\mathsf{T}) = \{1_F\}$ if $\tau$ is injective.
\item If $\mathsf{L} \subseteq \mathsf{M}$ then $\widetilde{\sigma}(\mathsf{L}) \geq \widetilde{\sigma}(\mathsf{M})$.
\item $\widetilde{\sigma}(\coprod_\lambda M_\lambda) = \bigcap_\lambda \widetilde{\sigma}(M_\lambda)$;
\item Given a triangle $M' \to M \to M''$ we have $\widetilde{\sigma}(M) \geq (\widetilde{\sigma}(M') \cap \widetilde{\sigma}(M''))$.
\item $\widetilde{\sigma}(\vee_\lambda \mathsf{L}_\lambda) = \cap_\lambda \widetilde{\sigma}(\mathsf{L})$;
\item $(\widetilde{\sigma}(\mathsf{L}) \vee \widetilde{\sigma}(\mathsf{M})) \subseteq \widetilde{\sigma}(\mathsf{L} \cap \mathsf{M})$.
\end{enumerate}
\end{lemma}
\begin{proof}
Each of these is quickly disposed of as follows:
\begin{enumerate}
\item This is immediate from the definition as the suspensions of an object all generate the same localizing ideal.
\item These statements are true by definition.
\item If $\mathsf{M} \subseteq \tau(g)$ then we also have $\mathsf{L} \subseteq \tau(g)$ and so $\widetilde{\sigma}(\mathsf{L}) \geq \widetilde{\sigma}(\mathsf{M})$.
\item We have
\begin{align*}
\widetilde{\sigma}\left(\coprod_\lambda M_\lambda\right) &= \{f\in F \mid \coprod_\lambda M_\lambda \in \tau f \} \\
&= \cap_\lambda \{f\in F \mid M_\lambda \in \tau f \} \\
&= \cap_\lambda \widetilde{\sigma}(M_\lambda)
%\subseteq \{f\in F \mid M_\lambda \in \tau f \} \text{ for each } \lambda.
\end{align*}
where the non-trivial equality is the fact that localizing subcategories are closed under retracts and arbitrary coproducts.
\item If $M' \in \tau f$ and $M'' \in \tau f$, i.e.\ $f\in (\widetilde{\sigma}(M') \cap \widetilde{\sigma}(M''))$, then the given triangle implies that $M \in \tau f$, i.e.\ $f\in \widetilde{\sigma}(M)$, as localizing ideals are closed under extensions. 
\item We have 
\begin{align*}
\widetilde{\sigma}(\vee_\lambda \mathsf{L}_\lambda) &= \{f\in F \mid \cup_\lambda \mathsf{L}_\lambda \subseteq \tau f \} \\
&= \cap_\lambda \{f\in F \mid \mathsf{L}_\lambda \subseteq \tau f \} \\
&= \cap_\lambda \widetilde{\sigma}(\mathsf{L})
\end{align*}
%\item We have 
%\begin{align*}
%\widetilde{\sigma}(\mathsf{L}\vee \mathsf{M}) &= \{f\in F \mid \mathsf{L} \cup \mathsf{M} \subseteq \tau f \} \\
%&= \{f\in F \mid \mathsf{L} \subseteq \tau f \} \cap \{g\in F \mid \mathsf{M} \subseteq \tau g \} \\
%&= \widetilde{\sigma}(\mathsf{L}) \cap \widetilde{\sigma}(\mathsf{M})
%%\subseteq \{f\in F \mid X_\lambda \in \tau f \} \text{ for each } \lambda.
%\end{align*}
\item We have
\[
\widetilde{\sigma}(\mathsf{L} \cap \mathsf{M}) = \{f\in F \mid \mathsf{L} \cap \mathsf{M} \subseteq \tau f \} \\
\]
and
\begin{align*}
\widetilde{\sigma}(\mathsf{L}) \vee \widetilde{\sigma}(\mathsf{M}) &= \{f\in F \mid \mathsf{L}\subseteq \tau f\} \vee \{g\in F \mid \mathsf{M}\subseteq \tau g\} \\
&= \{f\wedge g \mid \mathsf{L}\subseteq \tau f \text{ and } \mathsf{M}\subseteq \tau g \} \\
\end{align*}
Since $\tau$ preserves meets we know that if $f$ and $g$ satisfy the above conditions then 
\[
\mathsf{L} \cap \mathsf{M} \subseteq \tau(f) \cap \tau(g) = \tau(f\wedge g).
\]
This shows $f\wedge g \in \widetilde{\sigma}(\mathsf{L} \cap \mathsf{M})$ and so 
\[
(\widetilde{\sigma}(\mathsf{L}) \vee \widetilde{\sigma}(\mathsf{M})) \subseteq \widetilde{\sigma}(\mathsf{L} \cap \mathsf{M}).
\]
\end{enumerate}
\end{proof}

\begin{remark}
The statement (4) is a special case of (6) using the fact that
\[
\bigvee_\lambda \loc\left(M_\lambda\right) = \loc\left(\coprod_\lambda M_\lambda\right).
\]
\end{remark}

%\textcolor{red}{Find some natural condition so the last guy can be an equality and get a complete lattice map. Maybe it just works and I'm missing it.}

\begin{lemma}
The support $\sigma$ satisfies the following properties:
\begin{enumerate}
\item $\sigma(\Sigma^i M) = \sigma(M)$ for all $i \in \mathbb{Z}$;
\item $\sigma(0) = 0_F$ and $\sigma(\mathsf{T}) = \wedge\{f\in F \mid \tau(f) = \mathsf{T} \}$. In particular, $\sigma(\mathsf{T}) = 1_F$ if $\tau$ is injective.
\item $\sigma(\coprod_\lambda M_\lambda) \geq \vee_\lambda \sigma(M_\lambda)$ with equality if $M_\lambda \in \tau\sigma(M_\lambda)$ for all $\lambda$.
\item Given a triangle $M' \to M \to M''$ we have $\sigma(M) \leq (\sigma(M') \vee \sigma(M''))$.
\item $\sigma(\mathsf{L}) \wedge \sigma(\mathsf{M}) \geq \sigma(\mathsf{L} \cap \mathsf{M})$;
\item $(\sigma(\mathsf{L}) \vee \sigma(\mathsf{M})) \leq \sigma(\mathsf{L} \vee \mathsf{M})$.
\end{enumerate}
\end{lemma}
\begin{proof}
Exercise.
\end{proof}

%When introduce ltg use it to prove that finite localizations closed under intersections - this works right?

%Do ltg as extending to powerset, cheat and use spatial. See later if can do using breadth and if there's a notion of discretization of a locale

%----------------------------------------------------------------------------------------------------------------------------------------------

%----------------------------------------------------------------------------------------------------------------------------------------------

\subsection{Refined supports for big tt-categories}\label{Gsec:refined}

We now discuss how one can start with a support theory via $\tau \colon F \to \Loc^\otimes(\mathsf{T})$ and attempt to refine it. This leads to the local-to-global principle and questions about generators for $\mathsf{T}$. 

\begin{hypothesis}
Throughout this section we will assume that $F$ is a spatial frame and $\tau$ is injective and factors through $\Idem \mathsf{T}$.
\end{hypothesis}

Thus $\tau$ is a categorified spatial frame (or categorified spatial locale). 

\begin{remark}
The above factorization takes place via identifying $\Idem \mathsf{T}$ with the frame of smashing ideals. We will need to consider smashing ideals, idempotent coalgebras, and idempotent algebras but as we have seen these are all isomorphic lattices and so we identify them as convenient. We will generally think of $\tau f$ as the smashing ideal and write $A_{\tau f}$ and $C_{\tau f}$ for the associated algebra and coalgebra.
\end{remark}

\begin{remark}
It should be possible to produce a framework that, at the very least, does not require $F$ to be spatial but there are some challenges involved. I have some ideas about this. %keep reading about breadth and think about how to define some discretization of $F$

This is also restrictive in other ways. In order for it to capture everything it requires that the smashing localizations are sufficient to determine all localizations. This already places us beyond what we can compute so in practice it is currently rather mild. However, we know that smashing ideals will not be enough in general.
\end{remark}

\begin{example}
If $A$ is an absolutely flat commutative ring (i.e. every module is flat) then the telescope conjecture holds for $\mathsf{D}(A)$ by \cite{BS-smashing}, that is
\[
\Idem^\mathrm{fin}(\mathsf{D}(A)) = \Idem(\mathsf{D}(A)).
\] 
However, we know from \cite{Stevensonvnr} that there are examples where one cannot obtain all localizing subcategories from $\Idem(\mathsf{D}(A))$ using support-theoretic methods. 
\end{example}

\begin{definition}
A \emph{discretization} of $\tau$ is an injective map of complete lattices $\upsilon$ making the following triangle commute
\[
\begin{tikzcd}
F \arrow[d, "\tau"'] \arrow[r] & \mathcal{P}(\pt F) \arrow[ld, "\upsilon", dashed] \\
\Loc^\otimes(\mathsf{T})                     &                                                    
\end{tikzcd}
\]
where the map $F \to \mathcal{P}(\pt F)$ sends $f$ to the corresponding open subset $U_f$.
\end{definition}

\begin{remark}
In particular, a discretization gives a localizing ideal $\upsilon(x)$ for each $x\in \pt(F)$ such that $\upsilon(x) \cap \upsilon(y) = 0$ for $x\neq y$ and $\tau(f) = \vee_{x\in U_f} \upsilon(x)$. Specializing to $1_F$ we see that the $\upsilon(x)$ have to generate $\mathsf{T}$. These requests amount to a lot.
\end{remark}

\begin{remark}
I don't see an \emph{a priori} reason why there shouldn't be more than one discretization in general. If one relaxes the condition that $\tau$ factors through $\Idem \mathsf{T}$ then there can be more than one choice of discretization. In the examples we understand they are unique, but to check this one needs a sufficiently complete understanding of $\Loc^\otimes(\mathsf{T})$ as to defeat the purpose.
\end{remark}

\begin{definition}\label{Grem:weak}
One could instead consider a \emph{weak discretization} where we only require a natural transformation $\upsilon(U_{(-)}) \to \tau(-)$ i.e.\ for each $f$ a containment $\upsilon(U_{f}) = \vee_{x\in U_f} \upsilon(x) \subseteq \tau(f)$ rather than strict commutativity. It's also natural to ask less of $\upsilon$. For instance, that it only preserves joins (and hence not necessarily the top element). It will automatically be oplax monoidal for the cartesian monoidal structures. This is the sort of structure one usually starts with and is sometimes the best we know how to do. 
\end{definition}

%\begin{remark}
%Once one is only asking for a comparison map it would make sense to add a universality condition and reformulate this as follows: the weak discretization of $\tau$ is the right Kan extension of $\tau$ along $F\to \mathcal{P}(\pt F)$. Using this point of view it is clear how to generalize from $F\to \mathcal{P}(\pt F)$ to any frame map with source $F$.
%\end{remark}

Let us give an example and then begin the process of trying to produce a discretization of our categorified frame.

\begin{example}\label{Gex:Foxbysupport2}
Let us take $\mathsf{T} = \mathsf{D}(R)$ for $R$ a commutative noetherian ring (cf.\ Example~\ref{Gex:Foxbysupport}). We take for $\tau$ the inflation
\[
\Thom(\Spec R) \cong \Thick(\mathsf{D}^\mathrm{perf}(R)) \to \Loc(\mathsf{D}(R)).
\]
This is always injective and we know that $\Thick(\mathsf{D}^\mathrm{perf}(R))$ is a spatial frame with corresponding space of points $(\Spec R)^\vee$ the Hochster dual of $\Spec R$ (we will only talk about points and subsets in this example and so this dual topology won't matter). For each prime ideal $\mathfrak{p} \in \Spec R$ we can consider the residue field $k(\mathfrak{p})$ and Neeman proved in \cite{Neeman:1992a} that
\[
\mathsf{D}(R) = \loc(k(\mathfrak{p}) \mid \mathfrak{p}\in \Spec R).
\]
We define $\upsilon \colon \mathcal{P}(\Spec R) \to \Loc^\otimes(\mathsf{D}(R))$ by
\[
\upsilon(W) = \loc(k(\mathfrak{p}) \mid \mathfrak{p}\in W).
\]
This is compatible with joins by construction, and it is injective since $\upsilon(W)$ contains precisely the residue fields corresponding to points of $W$. It was an exercise that this preserves finite meets and it extends $\tau$ since we can take the equality of localizing ideals $\mathsf{D}(R) = \loc(k(\mathfrak{p}) \mid \mathfrak{p}\in \Spec R)$ and tensor with the idempotent coalgebra $C_V$ for a Thomason subset $V$ to see
\begin{align*}
\tau(V) &= \loc(C_V) \\
&= C_V\otimes \mathsf{D}(R) \\
&= C_V \otimes \loc(k(\mathfrak{p}) \mid \mathfrak{p}\in \Spec R) \\
&= \loc(C_V \otimes k(\mathfrak{p}) \mid \mathfrak{p}\in \Spec R) \\
&= \loc(k(\mathfrak{p}) \mid C_V \otimes k(\mathfrak{p})\neq 0) \\
&= \loc(k(\mathfrak{p}) \mid \mathfrak{p}\in V) \\
&= \upsilon(V).
\end{align*}
Hence $\upsilon$ is a discretization of $\tau$ and in fact Neeman shows that $\upsilon$ is a lattice isomorphism. In particular $\upsilon^{-1}$ gives a classifying support and this support agrees with Foxby's small support and the Balmer-Favi support in this setting.
\end{example}

We now discuss a general procedure for refining a categorified spatial frame $\tau\colon F \to \Idem(\mathsf{T})$ (as in our standing hypotheses). The first ingredient is a topological restriction on the space of points.

\begin{definition}
Let $X$ be a topological space. We say that $X$ is $T_D$ if every point of $X$ is locally closed. If $F$ is a spatial frame we will say that $F$ is $T_D$ if $\pt(F)$ is. %(In this section we assume $F$ is spatial and so we just call it $T_D$.)
\end{definition}

This turns out to be a key condition on $\pt(F)$. Through $F$ we can only access the open and closed subsets of $\pt(F)$ and the $T_D$ condition allows us to cut out any point using a pair of elements of $F$.

\begin{definition}
Suppose that $F$ is $T_D$. For a point $p\in \pt(F)$ choose a pair of elements $f,g\in F$ such that $\{p\} = U_f \cap V_g$ where $V_g = \pt(F)\setminus U_g$. We define the idempotent at $p$ to be
\[
\Gamma_p = C_{\tau f} \otimes A_{\tau g}
\]
where $C_{\tau f}$ is the idempotent coalgebra corresponding to $\tau f$ and $A_{\tau g}$ is the idempotent algebra for $\tau g$.
\end{definition}

\begin{remark}
The object $\Gamma_p$ is sometimes used to refer to the corresponding functor and the object we have defined above is written as $\Gamma_p\unit$. We choose to prioritize the object and write $\Gamma_p \otimes(-)$ for the corresponding functor. The notation $\kappa(p)$ is used for the analogous construction in \cite{Balmer/Favi:2011a}.

Perhaps it would be better to use some different notation to suggest that $\Gamma_p$ is a combination of an idempotent coalgebra and idempotent algebra (but in general is neither).
\end{remark}

\begin{remark}
Given a space $X$ a subset $W$ is locally closed if and only if there is an extension by zero functor from sheaves of abelian groups on $W$ to sheaves of abelian groups on $X$ (see e.g.\ \cite[Chapter~3, Theorem~8.6]{Tennison}). Knowing this it is hard not to make some remark based purely on vibes.

Actually, this connects back to Remark~\ref{Grem:Ko}. Let us assume that $\mathsf{T}$ is a presentably symmetric monoidal stable $\infty$-category and so we have, by \cite{aoki2023sheaves}, an action of $\mathrm{Shv}(F; \mathsf{Sp}) = \mathrm{Shv}(\pt(F); \mathsf{Sp})$ on $\mathsf{T}$ corresponding to $\tau$. Then the extension by zero from each point of $F$ exists precisely when $\pt(F)$ is $T_D$ and extending the unit of $\mathsf{Sp}$ by zero along the different points gives rise to the $\Gamma_x$'s.
\end{remark}

\begin{lemma}\label{Gxca:unique}
The object $\Gamma_p$ is independent of the choice of $f$ and $g$.
\end{lemma}
\begin{proof}
Let us actually prove the stronger statement: 
\[
\text{if } U_f \cap V_g = U_m \cap V_n \text{ then } C_{\tau f} \otimes A_{\tau g} \cong C_{\tau m} \otimes A_{\tau n}.
\]
We can split this into two parts by varying only the open subset and by varying only the closed subset. In fact, for the lemma we only need to vary the open as we can always take the closed subset to be $V_p$ the closed set corresponding to $p$ (we think of $p$ as both a point and as a meet-prime element). We prove this case and leave the remainder as an exercise.

So suppose that $U_f\cap V_g = U_m \cap V_g$. We claim we can replace $C_{\tau f}$ by $C_{\tau(f\vee g)} = C_{\tau f}\vee C_{\tau g}$. Indeed, the join of the idempotent coalgebras is defined by the triangle
\[
C_{\tau f}\otimes C_{\tau g} \to C_{\tau f}\oplus C_{\tau g} \to C_{\tau f}\vee C_{\tau g}
\]
and tensoring with $A_{\tau g}$ kills $C_{\tau g}$ rendering $C_{\tau f}\otimes A_{\tau g}$ isomorphic to $(C_{\tau f}\vee C_{\tau g})\otimes A_{\tau g}$ and similarly for $m$. So we have reduced to the case that $U_g \subseteq U_f$ and $U_g \subseteq U_m$. By assumption
\[
U_f \cap V_g = U_m \cap V_g
\]
i.e.\ $U_f$ and $U_m$ differ by elements of $U_g$. But both $U_f$ and $U_m$ contain $U_g$ and so they must be equal. Thus $\tau f = \tau m$ and so tensoring the corresponding coalgebras with $A_{\tau g}$ gives isomorphic objects.
\end{proof}

\begin{remark}
This statement is a generalization of \cite[Lemma~7.4]{Balmer/Favi:2011a} and the proof is essentially the same.
\end{remark}

\begin{lemma}\label{Glem:nonzero}
If $p$ is a locally closed point of $F$ then $\Gamma_p \neq 0$.
\end{lemma}
\begin{proof}
Let us write $\Gamma_p$ as $C_{\tau f} \otimes A_{\tau g}$ for $f,g\in F$. If $\Gamma_p$ were trivial we would have to have 
\[
C_{\tau f} \in \ker(A_{\tau g}\otimes -) = \loc^\otimes(C_{\tau g}). 
\]
This says that $\tau f \leq \tau g$ and hence, by injectivity of $\tau$, we must have $f\leq g$. By construction $p(g) = 0$ and so it would follow that $p(f)=0$ but we have chosen $f$ such that $p(f)=1$, i.e.\ we have reached an absurd conclusion. Thus $\Gamma_p$ could not have been trivial.
\end{proof}

\begin{xca}\label{Gxca:tensoridempotent}
Prove that if $p$ is a point of $F$ then $\Gamma_p \otimes \Gamma_p \cong \Gamma_p$. Is it necessarily an idempotent algebra or coalgebra?
\end{xca}

\begin{xca}\label{Gxca:tensororthogonal}
Prove that if $p\neq q$ are points of $F$ then $\Gamma_p \otimes \Gamma_q = 0$. 
\end{xca}

\begin{proposition}\label{Gprop:weak}
Suppose that $F$ is $T_D$. Then $\upsilon \colon \mathcal{P}(\pt(F)) \to \Loc^\otimes(\mathsf{T})$ defined by
\[
\upsilon(W) = \loc^\otimes(\Gamma_p \mid p\in W)
\]
is a weak discretization, i.e.\ it preserves joins, is injective, and for each $f$ we have 
\[
\upsilon(U_{f}) = \bigvee_{x\in U_f} \upsilon(x) \subseteq \tau(f).
\] 
\end{proposition}
\begin{proof}
The assignment $\upsilon$ preserves joins by construction. Injectivity follows from the previous exercise: if $W\neq W'$ then without loss of generality there is a $p\in W \setminus W'$ and we have $\upsilon(W') \subseteq \ker(\Gamma_p \otimes -)$ but $\Gamma_p \otimes \upsilon(W) \neq 0$. 

Let $C_{\tau f}$ denote the idempotent coalgebra corresponding to $\tau(f)$. If we can show that $\Gamma_p \otimes C_{\tau f} \cong \Gamma_p$ for every $p\in U_f$ it follows that $\upsilon(U_{f})\subseteq \tau(f)$. Let us pick elements $g,h\in F$ presenting $\Gamma_p$ as $C_{\tau h} \otimes A_{\tau g}$. We compute that
\[
\Gamma_p \otimes C_{\tau f} \cong C_{\tau h} \otimes A_{\tau g} \otimes C_{\tau f} = C_{\tau f \wedge \tau h} \otimes A_{\tau g} = C_{\tau (f\wedge h)} \otimes A_{\tau g}.
\]
By assumption $p\in U_f$ and $p\in U_h$ so $p$ also lies in $U_f \cap U_h = U_{f\wedge h}$. Thus we know $p\in U_{f\wedge h} \cap V_{g} \subseteq U_{h} \cap V_g = \{p\}$. It follows from Lemma~\ref{Gxca:unique} that $\Gamma_p \otimes C_{\tau f} \cong \Gamma_p$ as required.
\end{proof}

We next discuss what it takes it takes to upgrade $\upsilon$ to a discretization.

\begin{definition}
Suppose that $F$ is $T_D$. We say that the \emph{local-to-global principle} holds for $\tau$ if $\upsilon(\pt(F)) = \mathsf{T}$ i.e.\ if
\[
\mathsf{T} = \loc^\otimes(\Gamma_p \mid p\in \pt(F)).
\]
\end{definition}

\begin{remark}
One can rephrase this as saying $\upsilon$ preserves the empty meet.
\end{remark}

\begin{remark}
In some sense this definition goes back to Neeman's paper \cite{Neeman:1992a} where he established that the residue fields generate the derived category of a noetherian ring. This was abstracted, and named, in work of Benson, Iyengar, and Krause \cite{Benson/Iyengar/Krause:2011a}. The definition presented above is a direct descendent of the formulation, in the setting of big tt-geometry, given in \cite{Stevenson:2013a}.
\end{remark}

\begin{corollary}\label{Gcor:ltg}
If $\tau$ satisfies the local-to-global principle then $\upsilon$ is a discretization of $\tau$ which preserves arbitrary meets.
\end{corollary}
\begin{proof}
We need to check that $\upsilon$ preserves meets and that $\tau(f) \subseteq \upsilon(U_{f})$. Let us start with the latter. We assume that $\upsilon(\pt(F)) = \mathsf{T}$. Given $f\in F$ we have 
\begin{align*}
\tau f &= C_{\tau f}\otimes \mathsf{T} \\
&= C_{\tau f} \otimes \loc^\otimes(\Gamma_p \mid p\in \pt(F)) \\
&= \loc^\otimes(C_{\tau f} \otimes \Gamma_p \mid p\in \pt(F)) \\
&= \loc^\otimes(\Gamma_p \mid p\in U_f) \\
&= \upsilon(U_f)
\end{align*}
where the first equality is by construction of the associated idempotent coalgebra, the second is the local-to-global principle, the third is \cite[Lemma~3.12]{Stevenson:2013a}, the fourth is the computation from the proof of Proposition~\ref{Gprop:weak} combined with a small exercise, and the last is the definition.

We next show that $\upsilon$ preserves meets. So suppose $W_i$ for $i\in I$ are subsets of $\pt(F)$ with intersection $W$. By the universal property of the meet we know that
\[
\upsilon(W) \subseteq \cap_i \upsilon(W_i).
\]
So suppose that $M$ lies in the intersection of the $\upsilon(W_i)$. We know that
\[
\loc^\otimes(M) = M\otimes \mathsf{T} = M\otimes \loc^\otimes(\Gamma_p \mid p\in \pt(F)) = \loc^\otimes(M\otimes \Gamma_p \mid p\in \pt(F))
\]
and so $M$ is generated by the $\Gamma_p$ such that $M\otimes \Gamma_p \neq 0$. If $p \in W_i$ but not in $W_j$ then by Exercise~\ref{Gxca:tensororthogonal} tensoring with $\Gamma_p$ kills $\upsilon(W_j)$. In particular, $M\otimes \Gamma_p = 0$. Thus $\Gamma_p \otimes M\neq 0$ implies that $p\in W_i$ for all $i$ and hence $M\in \upsilon(W)$.
\end{proof}

\begin{lemma}\label{Glem:support}
Suppose that $F$ is $T_D$ and the local-to-global principle holds. Then the left adjoint of $\upsilon$ is given on a localizing ideal $\mathsf{L}$ by
\[
\Supp \mathsf{L} = \{p \in \pt(F) \mid \Gamma_p \otimes \mathsf{L} \neq 0\}.
\]
\end{lemma}
\begin{proof}
We know that the formula for the left adjoint at $\mathsf{L}$ is
\[
Y = \bigcap \{W \subseteq \pt(F) \mid \mathsf{L} \subseteq \upsilon(W)\}
\]
and we need to show that $Y = \Supp \mathsf{L}$. If $p\notin Y$ then $p\notin W$ for some subset $W$ with $\mathsf{L} \subseteq \upsilon(W)$ and, by the now standard trick, $\Gamma_p \otimes \upsilon(W) = 0$. Thus $p \notin \Supp \mathsf{L}$ showing that $\Supp \mathsf{L} \subseteq Y$. On the other hand by the local-to-global principle we deduce that
\[
\mathsf{L} = \mathsf{L} \otimes \mathsf{T} = \mathsf{L} \otimes \loc^\otimes(\Gamma_p \mid p\in \pt(F)) = \loc^\otimes(\mathsf{L}\otimes \Gamma_p \mid p\in \Supp \mathsf{L}).
\]
This exhibits $\mathsf{L}$ as a subcategory of $\upsilon(\Supp \mathsf{L})$ and so $Y\subseteq \Supp \mathsf{L}$.

%suppose that $\Gamma_p \otimes \mathsf{L} \neq 0$. 
\end{proof}

%We only really have any understanding of the local-to-global principle in the case that $F = \Thick^\otimes(\mathsf{T}^\mathrm{c})$ and $\tau$ is inflation. However, 

%\textcolor{red}{I should check the exercises, proof read everything. Perhaps think about proving ltg for arbitrary guys but maybe save that if can do all using breadth without spatial and just do examples. I guess ponder the breadth stuff...}

\begin{remark}
It could be an interesting project to use Verasdanis' theory of support-cosupport pairs \cite{verasdanis2022costratification} to extend all of this to a theory capturing information about both localizing and colocalizing subcategories.
\end{remark}

\begin{corollary}
Suppose that $F$ is $T_D$ and the local-to-global principle holds. Then the support of Lemma~\ref{Glem:support} detects vanishing:
\[
M\in \mathsf{T} \text{ is } 0 \text{ if and only if } \Supp M = \varnothing
\]
and gives a retraction onto $\mathcal{P}(\pt F)$. 
\end{corollary}
\begin{proof}
If $M=0$ then $\loc^\otimes(M) = 0$ and so obviously $\Supp M = \varnothing$. On the other hand, if $\Supp M = \varnothing$ then by definition $M\otimes \Gamma_p = 0$ for every $p\in \pt F$. But the local-to-global principle tells us that
\[
\loc^\otimes(M) = \loc^\otimes(M\otimes \Gamma_p \mid p\in \pt(F))
\]
(as in the proof of Corollary~\ref{Gcor:ltg}) and the localizing ideal on the right is obviously zero. Hence $M \cong 0$.

Now we come to the retraction. Let $W$ be a subset of points of $F$. We have
\begin{align*}
\Supp \upsilon(W) &= \Supp(\loc^\otimes(\Gamma_q \mid q\in W)) \\
&= \{p \in \pt(F) \mid \Gamma_p \otimes \loc^\otimes(\Gamma_q \mid q\in W) \neq 0\} \\
&= \{p \in \pt(F) \mid \Gamma_p \otimes \Gamma_q \neq 0 \text{ for some } q\in W\} \\
\end{align*}
By Exercises~\ref{Gxca:tensoridempotent} and \ref{Gxca:tensororthogonal} and Lemma~\ref{Glem:nonzero} we know that $\Gamma_p \otimes \Gamma_p \neq 0$ and $\Gamma_p\otimes \Gamma_q = 0$ if $p\neq q$. Thus $\Supp\upsilon(W)$ is precisely $W$.
\end{proof}

To summarise, in the presence of the local-to-global principle we have an injective map
\[
\upsilon \colon \mathcal{P}(\pt(F)) \to \Loc^\otimes(\mathsf{T})
\]
preserving arbitrary joins and meets whose left adjoint, the refined support, gives a retraction. One is then naturally tempted to look for conditions guaranteeing that $\upsilon$ is also surjective and hence an isomorphism of lattices. This happens precisely when each $\upsilon(\{p\}) = \loc^\otimes(\Gamma_p)$ is an atom i.e.\ a minimal non-zero localizing ideal.

%----------------------------------------------------------------------------------------------------------------------------------------------

%----------------------------------------------------------------------------------------------------------------------------------------------

\subsection{Examples}

Let us now discuss some examples where one can get a handle on all of the objects described above and check the local-to-global principle holds for some choice of $F$ and $\tau$.

\subsection{Noetherian commutative rings}

Let $R$ be a commutative noetherian ring, an example which by now should seem somewhat familiar. As we saw in Example~\ref{Gex:Foxbysupport2} we know all the localizing subcategories of $\mathsf{D}(R)$ and that $\Idem\mathsf{D}(R) \cong \Thick^\otimes(\mathsf{D}(R)^\mathrm{c})$. However, our argument used the residue fields $k(\mathfrak{p})$ and these are not the $\Gamma_\mathfrak{p}$ which arise from $\tau\colon \Thick^\otimes(\mathsf{D}(R)^\mathrm{c}) \to \Loc^\otimes(\mathsf{D}(R))$. (Although they turn out to give equivalent information.)

The first observation, which we leave as an exercise, is that if $R$ is noetherian then
\[
\pt(\Thick^\otimes(\mathsf{D}(R)^\mathrm{c})) \cong (\Spec R)^\vee
\]
is $T_D$. Thus we are in the situation of Proposition~\ref{Gprop:weak}. The objects $\Gamma_\mathfrak{p}$ in this setting can be described as the localization at $\mathfrak{p}$ of the stable Koszul complex at $\mathfrak{p}$ (see \cite[Section~3.4]{Stevensontour} for a summary, including definitions, and further references). If we restrict to the case that $R$ is Gorenstein then we can describe this simply as
\[
\Gamma_\mathfrak{p} \cong \Sigma^{-\mathrm{ht}(\mathfrak{p})}E(k(\mathfrak{p}))
\]
as a shift, by the height $\mathrm{ht}(\mathfrak{p})$ of $\mathfrak{p}$, of the injective envelope $E(k(\mathfrak{p})$ of the residue field at $\mathfrak{p}$. In particular, this is already an object of $\mathsf{D}(R_\mathfrak{p})$ and one can check that $k(\mathfrak{p})$ and $E(k(\mathfrak{p}))$ generate the same localizing subcategory (this is true without any Gorenstein hypothesis).

In particular, Example~\ref{Gex:Foxbysupport2} then gives that $\tau$ satisfies the local-to-global principle. %This was the inspiration for a much more general fact: if $\Spc \mathsf{T}^\mathrm{c}$ is a noetherian space then inflation satisfies the local-to-global principle. 

\subsection{Absolutely flat rings}

In this section we study an example where the local-to-global principle relative to thick tensor ideals of the compacts can fail. It concerns a special class of commutative rings.

\begin{definition}
A commutative ring $A$ is \emph{absolutely flat} if every $A$-module is flat. This is equivalent to $A$ being reduced and zero dimensional.
\end{definition}

We fix an absolutely flat ring $A$. Its spectrum $\Spec A$ is a zero dimensional coherent space and hence a profinite space (so in particular it is Hausdorff and Hochster self-dual). So it is certainly $T_D$ and we can pick out a point $\mathfrak{p}$ by taking the complement of the open subset $\mathcal{Z}(\mathfrak{p}) = \Spec A \setminus \{\mathfrak{p}\}$ consisting of points that do not specialize to $\mathfrak{p}$. The corresponding idempotent algebra is the localization $A_\mathfrak{p}$. As $A$ is absolutely flat this is just the residue field $k(\mathfrak{p})$. In other words, $\Gamma_\mathfrak{p} = k(\mathfrak{p})$.

The local-to-global principle then asks if the $k(\mathfrak{p})$ generate $\mathsf{D}(A)$. Sometimes this is true and sometimes it is not (see \cite{Stevensonvnr} and \cite{Stevensonltg} as well as the proof of the local-to-global principle in the next section).

\begin{definition}
We say that $A$ is \emph{semiartinian} if every non-zero quotient of $A$ in $\Modu A$ contains a simple submodule (i.e.\ it has some $k(\mathfrak{p})$ as a submodule).
\end{definition}

\begin{theorem}
Suppose that $A$ is absolutely flat. The local-to-global principle holds for $A$ if and only if $A$ is semiartinian.
\end{theorem}

We will not give a proof, but merely provide some comments. As mentioned above the statement is equivalent to the residue fields generating if and only if $A$ is semiartinian. The fact that the residue fields generate if $A$ is semiartinian is a consequence of the theorem in the next section, together with the fact that $A$ is semiartinian exactly when $\Spec A$ has Cantor-Bendixson rank (see the next section for the definition). On the other hand, if $A$ is not semiartinian then using work of Trlifaj \cite{Trlifaj} there exists an injective module with no indecomposable summand which, because the $k(\mathfrak{p})$ are simple injectives, must lie in the right orthogonal of $\loc(k(\mathfrak{p}) \mid \mathfrak{p}\in \Spec A)$.

%----------------------------------------------------------------------------------------------------------------------------------------------

%----------------------------------------------------------------------------------------------------------------------------------------------

\subsection{Some cases of the local-to-global principle}

In this section we give the most general result that seems to be known concerning the local-to-global principle (most general is pretty cheap since I just invented the setting of this chapter). This has a number of predecessors. However, the results which were previously known were restricted to the case of $\tau\colon \Thick^\otimes(\mathsf{T}^\mathrm{c}) \to \Idem(\mathsf{T})$. This was originally proved by Benson, Iyengar, and Krause in the context of finite localizations coming from actions by noetherian rings of finite Krull dimension, see \cite[Corollary~3.5]{Benson/Iyengar/Krause:2011a}. This was extended to the tt-setting in \cite{Stevenson:2013a} and then generalized further in \cite{Stevensonltg} motivated by the case of absolutely flat rings. In \cite{Stevensonltg} the noetherian and profinite cases are treated separately. This is actually not necessary and in \cite[Theorem~7.18]{sanders2017support} it is shown that the same methods can be used to give the local-to-global principle with respect to inflation of the compacts whenever $(\Spc \mathsf{T}^\mathrm{c})^\vee$ has Cantor-Bendixson rank, as defined below, which turns out to be the fundamental invariant.

\begin{definition}
Let $X$ be a space. We let $X_{\leq 0}$ denote the set of isolated (i.e.\ open) points of $X$ and set $X_{>0} = X\setminus X_{\leq 0}$. We then define further subspaces by transfinite induction. Suppose that $X_{\leq \alpha}$ has been defined and denote by $X_{> \alpha}$ its complement. We set
\[
X_{\leq \alpha+1} = X_{\leq \alpha} \cup (X_{>\alpha})_{\leq 0}
\] 
i.e.\ we add the isolated points of $X_{> \alpha}$ to $X_{\leq \alpha}$. If $\lambda$ is a limit ordinal we let
\[
X_{\leq \lambda} = \bigcup_{\kappa<\lambda} X_{\leq \kappa}.
\]
This is the \emph{Cantor-Bendixson filtration} of $X$. If there is an $\alpha$ such that $X = X_{\leq \alpha}$ then we say $X$ \emph{has Cantor-Bendixson (CB) rank} and we take the rank to be the least ordinal $\beta$ such that $X = X_{\leq \beta}$. If there is no such ordinal, i.e.\ the filtration is not exhaustive, we say that the rank is undefined or that $X$ does not have rank.
\end{definition}

\begin{example}
If $X$ is a noetherian coherent space then $X^\vee$ has CB rank.
\end{example}

\begin{example}
If $A$ is a semiartinian absolutely flat ring then $\Spec A$ has CB rank.
\end{example}

We will now give a criterion for the local-to-global principle to hold. We fix a sober $T_D$ space $X$ with CB rank and a big tt-category $\mathsf{T}$. We assume we are given an injective map of frames $\tau\colon \Omega(X) \to \Idem(\mathsf{T})$ (and as usual we freely move between idempotent algebras, coalgebras, and smashing ideals). Before stating our criterion we give some preparatory lemmas. 

\begin{lemma}\label{Glem:ltg1}
The subset $X_{\leq \alpha}$ is open for every ordinal $\alpha$.
\end{lemma}
\begin{proof}
We proceed by transfinite induction. The base case is clear: $X_{\leq 0}$ is a union of open points and hence open. Suppose $\alpha$ is an ordinal and $X_{\leq \alpha}$ is open. By definition
\begin{displaymath}
X_{\leq \alpha+1} = X_{\leq \alpha} \cup \{x\in X_{>\alpha} \mid x \text{ is open in } X_{> \alpha}\}.
\end{displaymath}
If $x\in X_{\leq \alpha+1}\setminus X_{\leq \alpha}$, that is to say $x$ is open in $X_{>\alpha}$, then there is an open neighbourhood $U$ of $x$ such that $U\cap X_{>\alpha} = \{x\}$ i.e., $U\setminus \{x\} \subseteq X_{\leq \alpha}$. Thus $U$ is an open neighbourhood of $x$ contained in $X_{\leq \alpha+1}$ showing $X_{\leq \alpha+1}$ is open. If $\lambda$ is a limit ordinal then
\begin{displaymath}
X_{\leq \lambda} = \bigcup_{\kappa < \lambda} X_{\leq \kappa}.
\end{displaymath}
By the induction hypothesis each $X_{\leq \kappa}$ is open and thus $X_{\leq \lambda}$ is also open.
\end{proof}

The next two results allow us to pass the categorified locale structure to quotients of $\mathsf{T}$ by smashing ideals.

\begin{lemma}\label{Glem:ltg2}
Let $U$ be an open subset of $X$ with closed complement $Z$. Then the interval
\[
[U, X] = \{W \in \Omega(X) \mid U\subseteq W\}
\]
is a frame and naturally isomorphic to $\Omega(Z)$ by sending $W$ to $W\setminus U = W \cap Z$. 
\end{lemma}
\begin{proof}
The interval $[U,X]$ has bottom element $U$, top element $X$, and is closed under non-empty joins and meets in $\Omega(X)$. It is thus clearly a frame. Let us write $f\colon [U,X] \to \Omega(Z)$ for the map sending $W$ to its intersection with $Z$. We note that this is just the composite
\[
\begin{tikzcd}
{[U,X]} \arrow[r, hook] & \Omega(X) \arrow[r, "\Omega(i)"] & \Omega(Z)
\end{tikzcd}
\]
where $i\colon Z\to X$ is the inclusion. This is order preserving, injective since $X = U\coprod Z$ as sets, and surjective as if $V \in \Omega(Z)$ then there is an open subset $W$ of $X$ such that $W \cap Z = V$ and the open subset $W\cup U \in [U,X]$ satisfies $f(W\cup U) = V$. Finally, we need to check that $f$ reflects the ordering, but this is obvious. This shows $f$ is an isomorphism of posets and it is then automatically a frame isomorphism since being a frame is a property of a poset.

%We have $f(U) = \varnothing$ and $f(X) = Z$ so the top and bottom elements are preserved. Intersection commutes with unions and so $f$ preserves joins.
\end{proof}

\begin{lemma}\label{Glem:ltg3}
Let $\mathsf{S}$ be a smashing ideal of $\mathsf{T}$ with corresponding idempotent algebra $A$. Then $\mathsf{T}/\mathsf{S}$ is again a big tt-category and
\[
(-)\otimes A\colon [A, 0] \to \Idem(\mathsf{T}/\mathsf{S})
\]
is an isomorphism of frames where $[A,0] = \{B\in \Idem(\mathsf{T}) \mid A\leq B \leq 0\}$.
\end{lemma}
\begin{proof}
Because $\mathsf{S}$ is smashing the localization functor $\mathsf{T} \to \mathsf{T}/\mathsf{S}$ sends compacts to compacts (this is equivalent to the right adjoint preserving coproducts, see Theorem~\ref{Gthm:localizationthm}) and the image of $\mathsf{T}^\mathrm{c}$ thus gives a compact generating set for $\mathsf{T}/\mathsf{S}$. The localization inherits a monoidal structure such that the localization is monoidal because $\mathsf{S}$ is an ideal. Any monoidal functor sends rigid objects to rigid objects and so the compact generators for $\mathsf{T}/\mathsf{S}$ are also rigid and hence it is a big tt-category. The unit object of $\mathsf{T}/\mathsf{S}$ is $A$.

Let us show the morphism $(-)\otimes A$ is an order preserving bijection (see Definition~\ref{Gdef:idempotentalgebra} and the remark that follows it for the ordering in terms of idempotent algebras). Suppose that $A_1 \leq A_2$ lie in $[A,0]$, i.e.\ they come with algebra maps $A\to A_i$. Since $A_i \geq A$ tensoring the unit map for $A$ with $A_i$ gives isomorphisms $A_i \stackrel{\sim}{\to} A\otimes A_i$. Thus we see that the map $A\to A_i$ makes $A_i$ into an idempotent algebra in $\mathsf{T}/\mathsf{S}$ and the map $A_1 \to A_2$ remains. It's immediate from this construction that the map of the lemma is well-defined, order preserving, and injective. 

Let us check surjectivity. Suppose that $A\to B$ is an idempotent algebra in $\mathsf{T}/\mathsf{S}$. The composite $\unit \to A \to B$ makes $B$ into an idempotent algebra in $\mathsf{T}$ as one can see by tensoring with $B\cong (A\otimes B)$. It lies in $[A,0]$ by construction and tensoring with $A$ sends it back to $B$ viewed as an idempotent algebra in $\mathsf{T}/\mathsf{S}$.
\end{proof}

\begin{proposition}\label{Gprop:ltg}
Suppose that $\tau\colon \Omega(X) \to \Idem(\mathsf{T})$ is a categorified frame. Let $U\in \Omega(X)$ with closed complement $Z$. Then $\tau$ induces a categorified frame on the corresponding localization $\Omega(Z) \to \Idem(\mathsf{T}/\tau(U))$.
\end{proposition}
\begin{proof}
This just boils down to restricting $\tau$ to $[U,X] \to [\tau(U), 0]$ and applying the last two lemmas.
\end{proof}

\begin{remark}
The hypotheses above are overkill: the same argument, with obvious modifications, works for any frame map $F\to \Idem(\mathsf{T})$ and element $f\in F$.
\end{remark}

We are now ready to prove that Cantor-Bendixson rank is enough to guarantee the local-to-global principle. In fact this also guarantees the $T_D$ hypotheses: if $X$ has Cantor-Bendixson rank it is automatically $T_D$ (see \cite[Lemma~7.1]{balchin2026importance}).

\begin{theorem}
Suppose that $X$ is sober and has Cantor-Bendixson rank. Then any injective frame map $\tau\colon \Omega(X) \to \Idem(\mathsf{T})$ satisfies the local-to-global principle. In particular, the corresponding map $\upsilon$ of Proposition~\ref{Gprop:weak} is a discretization of $\tau$ which preserves arbitrary meets.
\end{theorem}
\begin{proof}
We will prove by transfinite induction that 
\[
\tau(X_{\leq \alpha}) = \loc^\otimes(\Gamma_x \mid x\in X_{\leq \alpha}) = \upsilon(X_{\leq \alpha})
\]
for all ordinals $\alpha$. In particular, if $\beta$ is the Cantor-Bendixson rank of $X$ then 
\[
\mathsf{T} = \tau(X) = \tau(X_{\leq \beta}) = \upsilon(X_{\leq \beta}) = \upsilon(X).
\]
The base case is $\tau(X_{\leq 0}) = \upsilon (X_{\leq 0})$. Every point in $X_{\leq 0}$ is open by definition. Thus, using that $\tau$ is a map of frames, we see
\[
\tau(X_{\leq 0}) = \tau(\cup_{x\in X_{\leq 0}}\{x\}) = \vee_{x\in X_{\leq 0}} \tau(\{x\})
\]
and so it is enough to identify $\tau(\{x\})$ with $\upsilon(\{x\})$ (we know $\upsilon$ preserves joins by Proposition~\ref{Gprop:weak}). Because $\{x\}$ is open we can isolate $x$ as $\{x\} = \{x\} \cap X$ (thinking of $X$ as the closed complement of the open subset $\varnothing$) which gives
\[
\Gamma_x = C_{\tau(\{x\})} \otimes A_{\varnothing} = C_{\tau(\{x\})} \otimes \unit = C_{\tau(\{x\})}
\]
and hence $\upsilon(\{x\}) = \loc^\otimes(\Gamma_x) = \loc^\otimes(C_{\tau(\{x\})}) = \tau(\{x\})$.

We now treat the case of limit ordinals which, using our setup, is essentially trivial. Suppose that $\lambda$ is a limit ordinal and that $\tau(X_{\leq \kappa}) =\upsilon(X_{\leq \kappa})$ for every $\kappa < \lambda$. We see that
\begin{align*}
\tau(X_{\leq \lambda}) &= \tau(\cup_{\kappa<\lambda} X_{\leq \kappa}) \\
&= \vee_{\kappa<\lambda}\, \tau(X_{\leq \kappa}) \\
&= \vee_{\kappa<\lambda}\, \upsilon(X_{\leq \kappa}) \\
&= \upsilon(\cup_{\kappa<\lambda} X_{\leq \kappa}) \\
&= \upsilon(X_{\leq \lambda})
\end{align*}
where we have used that both $\tau$ and $\upsilon$ commute with joins and the induction hypothesis.

Finally, we treat the case of successor ordinals. So suppose that $\tau(X_{\leq \beta}) = \upsilon(X_{\leq \beta})$ for all $\beta \leq \alpha$. We consider the commutative diagram
\[
\begin{tikzcd}
\tau(X_{\leq \alpha}) \arrow[r] & \tau(X_{\leq \alpha+1}) \arrow[r] \arrow[d] & \mathsf{Q} \arrow[d]             \\
                                & \mathsf{T} \arrow[r]                        & \mathsf{T}/\tau(X_{\leq \alpha})
\end{tikzcd}
\]
where $\mathsf{Q} = \tau(X_{\leq \alpha+1})/\tau(X_{\leq \alpha})$, i.e.\ the top row is a localization sequence, the vertical arrows are fully faithful, and every functor has a right adjoint. By Proposition~\ref{Gprop:ltg} we get a categorified locale $\tau'\colon \Omega(X_{>\alpha}) \to \mathsf{T}/\tau(X_{\leq \alpha})$. By definition of the Cantor-Bendixson rank we have
\[
X_{\leq \alpha+1} \cap X_{>\alpha} = (X_{>\alpha})_{\leq 0}
\]
and so using Lemmas~\ref{Glem:ltg2} and \ref{Glem:ltg3} we identify $\mathsf{Q}$ with $\tau'((X_{>\alpha})_{\leq 0})$. By the argument for the base case of the induction we know that $\upsilon'((X_{>\alpha})_{\leq 0}) = \tau'((X_{>\alpha})_{\leq 0})$ and using the right adjoint of the localization to regard $\mathsf{Q}$ as a localizing ideal contained in $\tau(X_{\leq \alpha+1})$ we see
\[
\tau(X_{\leq \alpha+1}) = \tau(X_{\leq \alpha}) \vee \mathsf{Q} = \upsilon(X_{\leq \alpha}) \vee \upsilon'((X_{>\alpha})_{\leq 0}).
\]
Thus it is enough to identify, for $x\in X$ with Cantor-Bendixson rank $\alpha+1$, the objects $\Gamma_x$ and $\Gamma'_x$ where the latter is constructed using $\tau'$ (in other words to verify the equality $\upsilon'((X_{>\alpha})_{\leq 0}) = \upsilon(X_{\leq \alpha+1} \cap X_{>\alpha})$). 

Let us denote by $A$ the idempotent algebra corresponding to $\tau(X_{\leq \alpha})$. We know that for $x$ viewed as a point of $(X_{>\alpha})_{\leq 0}$ the object $\Gamma'_x$ is just the idempotent coalgebra for $\tau'(\{x\})$ (cf.\ the proof of the base case for the induction). By another combo of Lemma~\ref{Glem:ltg2} and Lemma~\ref{Glem:ltg3} this is the fibre of $\unit\otimes A \to A_{\tau{X_{\leq \alpha} \cup \{x\}}} \otimes A$, which is just $C_{\tau{X_{\leq \alpha} \cup \{x\}}} \otimes A$. But $A$ is the idempotent algebra for $\tau(X_{\leq \alpha})$ and $X_{> \alpha} \cap (X_{\leq \alpha} \cup \{x\}) = \{x\}$, so this tensor product of an idempotent algebra and coalgebra presents $\Gamma_x$. This completes the proof.
\end{proof}

%--------------------------------------------------------------------------------------------------------------------------------------------------

%--------------------------------------------------------------------------------------------------------------------------------------------------

\bibliographystyle{amsplain}
\bibliography{./repcoh}

\end{document}